\documentclass[10pt,leqno]{amsart}
\usepackage{hyperref}
\usepackage[all,cmtip]{xy}
\usepackage{a4,latexsym}
\usepackage[utf8]{inputenc}
\usepackage{textcomp,fontenc}
\usepackage{exscale}
\usepackage{amsxtra,amsmath,amsfonts,amscd, amssymb, mathrsfs, amsthm}
\usepackage{mathtools}
\usepackage{orcidlink}

\newtheorem{thm}{Theorem}[subsection]
\newtheorem{prop}[thm]{Proposition}

\newtheorem{cor}[thm]{Corollary}
\newtheorem{lem}[thm]{Lemma}
\newtheorem{lemma}[thm]{Lemma}

\theoremstyle{definition}
\newtheorem{definition}[thm]{Definition}

\newtheorem{ex}[thm]{Example}
\newtheorem{rem}[thm]{Remark}
\newtheorem{art}[thm]{}

\newtheorem{cons}[thm]{Construction}

\newtheorem{remi}[thm]{Reminder}
\newtheorem{theointro}{Theorem}

\numberwithin{paragraph}{section}
\numberwithin{equation}{subsection}

\def\P{{\mathbb P}}
\def\N{{\mathbb N}}
\def\Z{{\mathbb Z}}

\def\R{{\mathbb R}}
\def\C{{\mathbb C}}

\def\A{{\mathbb A}}

\def\G{{\mathbb G}}
\def\L{{\mathbb L}}
\def\T{{\mathbb T}}

\def\KC{{\mathscr C}}

\def\KO{{\mathcal O}}

\newcommand{\metr}{{\|\hspace{1ex}\|}}
\newcommand{\Hom}{{\rm Hom}}

\def\an{{\rm an}}
\def\div{{\rm div}}
\def\trop{{\rm trop}}
\def\Trop{{\rm Trop}}
\def\tor{{\rm tor}}

\DeclareMathOperator{\TZ}{TZ}

\DeclareMathOperator{\rec}{rec}

\mathchardef\mdash="2D

\newcommand{\Xan}{{X^{\rm an}}}
\newcommand{\Yan}{{Y^{\rm an}}}
\newcommand{\Zan}{{Z^{\rm an}}}

\newcommand{\Uan}{{U^{\rm an}}}

\newcommand{\Ccal}{{\mathscr C}}
\newcommand{\Dcal}{{\mathscr D}}

\newcommand{\Ocal}{{\mathscr O}}

\newcommand{\Tcal}{{\mathscr T}}

\newcommand{\Div}{{\rm div}}
\newcommand{\cyc}{{\rm cyc}}

\newcommand{\Spec}{{\rm Spec}}

\newcommand{\im}{{\rm im}}

\newcommand{\Star}{{\rm star}}
\newcommand{\supp}{{\rm supp}}

\newcommand{\relint}{{\rm relint}}
\newcommand{\tint}{{\rm int}}

\newcommand{\dpa}{{d'_{\rm P}}}
\newcommand{\dpb}{{d''_{\rm P}}}

\newcommand{\Rsup}{{\R_{\infty}}}
\newcommand{\torus}{{\T}}
\newcommand{\alphatrop}{{\alpha_{\trop}}}
\newcommand{\betatrop}{{\beta_{\trop}}}
\newcommand{\varphitrop}{{\varphi_{\trop}}}
\newcommand{\psitrop}{{\psi_{\trop}}}

\DeclareMathOperator{\PS}{PS}

\DeclareMathOperator{\Th}{Th}
\DeclareMathOperator{\Conv}{Conv}
\DeclareMathOperator{\Cone}{Cone}
\DeclareMathOperator{\MA}{MA}

\begin{document}

\title{A tropical formula for non-archimedean local heights}

\author[J.I.Burgos Gil]{Jos\'e Ignacio Burgos Gil\,\orcidlink{0000-0003-2640-2190}}
\address{
J. I. Burgos Gil,
Instituto de Ciencias Matem\'aticas (CSIC-UAM-UCM-UCM3), 
Calle Nicol\'as Cabrera 15, Campus de la Universidad 
Aut\'onoma de Madrid, Cantoblanco, 28049 Madrid, Spain}
\email{burgos@icmat.es}

\author[W.~Gubler]{Walter Gubler \,\orcidlink{0000-0003-2782-5611}}
\address{W. Gubler, Mathematik, Universit{\"a}t 
Regensburg, 93040 Regensburg, Germany}
\email{walter.gubler@mathematik.uni-regensburg.de}

\author[K.~K\"unnemann]{Klaus K{\"u}nnemann\,\orcidlink{0000-0003-1109-9096}}
\address{K. K{\"u}nnemann, Mathematik, Universit{\"a}t 
Regensburg, 93040 Regensburg, Germany}
\email{klaus.kuennemann@mathematik.uni-regensburg.de}

 \thanks{J.~I.~Burgos was partially supported by grants
   PID2022-142024NB-I00 and  CEX2023-001347-S funded by
   MICIU/AEI/10.13039/501100011033.
W.~Gubler and K.~K{\"u}nnemann were supported by the collaborative research center SFB 1085 \emph{Higher Invariants - Interactions between Arithmetic Geometry and Global Analysis} funded by the Deutsche Forschungsgemeinschaft.}

\date{\today}

\begin{abstract}

We introduce $\delta$-forms on tropical toric varieties generalizing the construction of Mihatsch for $\R^n$. These $\delta$-forms will be used to define the star-product with Green functions of piecewise smooth type on a tropical toric variety. As an application, we show that non-archimedean local heights of projective varieties can be computed using the star-product on a suitable complete tropical toric variety. 
  
On the way, we show that open subsets of a simplicial tropical toric variety have a locally finite simplicial decomposition which is constant towards the boundary.
\bigskip

MSC: Primary 14G40; Secondary  14M25, 14T25, 32P05, 32U40, 52B20. 
\end{abstract}

\maketitle
\setcounter{tocdepth}{1}

\tableofcontents

\section{Introduction}\label{intro}

Arakelov theory provides a modern approach to diophantine geometry that uses schemes as models at finite places and the analysis of Green currents on complex manifolds at infinite places. 
In the case of curves, this ingenious construction goes back to Arakelov who was motivated by Mordell's conjecture. 
Faltings used Arakelov theory to solve this major conjecture from the previous century. Gillet and Soul\'e laid the foundations for Arakelov theory in higher dimension. 
The backbone of their theory is an arithmetic intersection product on projective regular schemes over the integers which is induced by the usual intersection product at geometric parts and by a $*$-product of Green currents on the associated complex manifold. 
The arithmetic intersection product is used to introduce heights of projective varieties as an arithmetic degree. 
This construction goes back to Faltings in his proof of the Mordell--Lang conjecture for closed subvarieties of abelian varieties.

As the canonical heights for closed subvarieties of abelian varieties cannot be handled properly by using models, there is a natural desire to handle the finite places also in an analytic way by using an analogue of Green currents on the underlying Berkovich analytic space. 
The theory of real valued differential forms and currents on Berkovich spaces, introduced by Chambert-Loir and Ducros,  
provides us with such a tool in non-archimedean geometry. 
However, the metric on line bundles coming from the models used in Arakelov theory are not smooth in the sense of Chambert-Loir and Ducros. 
This was the reason that $\delta$-forms were introduced in \cite{gubler-kuenne2017}, a generalization of the real valued differential form motivated by tropical geometry. 
It was shown in \cite{gubler-kuenne2017} that $\delta$-forms can be used to compute  Monge--Amp\'ere measures as a product of $\delta$-forms and that a star-product on the associated Berkovich space can be defined for projective varieties computing their local heights. 
The purpose of this paper is to show that such a local height can always be computed by a tropical formula using $\delta$-forms on tropical toric varieties. 

\subsection{Forms and currents on $\R^n$} \label{subsec: forms and currents on euclidean space}

Lagerberg introduced real valued smooth $(p,q)$-forms on $\R^n$ for all integers $0 \leq p,q \leq n$ \cite{lagerberg-2012}. 
For an open subset $U$ of $\R^n$, the space of such forms is denoted by $A^{p,q}(U)$. 
He showed that 
\[
A(U)= \bigoplus_{p,q} A^{p,q}(U)
\]
is naturally a bigraded $\R$-algebra with anticommuting differentials $d'\colon A^{p,q}(U) \to A^{p+1,q}(U)$ and $d''\colon A^{p,q}(U) \to A^{p,q+1}(U)$ behaving similarly as their complex analytic counterparts. We call the elements of $A(U)$ \emph{smooth Lagerberg forms} on $U$. 
They form a sheaf on $\R^n$. 
There is integration of compactly supported smooth Lagerberg forms of top-degree such that the formula of Stokes holds and the elements of $A^{0,0}(U)$ are by definition the smooth functions on $U$. 

Similarly as in the complex case, Lagerberg defines the space of currents 
\[
D^{n-p,n-q}(U)=D_{p,q}(U)
\] 
as the topological dual of the space of compactly supported smooth $(p,q)$-forms $A_c^{p,q}(U)$. 
The differentials $d'$ and $d''$ on $D(U)=\bigoplus_{p,q}D_{p,q}(U)$ are obtained by duality up to the usual sign.

Using integration, we have a canonical embedding $A^{p,q}(U)\hookrightarrow D^{p,q}(U)$. 
Other examples of currents are given by currents of integration over
weighted polyhedra. 
By linearity, any weighted rational polyhedral complex $C=(\Ccal,m)$ in $\R^n$ of pure codimension $k$  gives rise to a Lagerberg current $\delta_C \in D^{k,k}(\R^n)$. 
Then $C$ is a \emph{tropical cycle}, i.e.~balanced, if and only if $\delta_C$ is $d'$-closed (equivalently $d''$-closed), see \cite[Proposition 3.8]{gubler-forms}.

\subsection{Forms and currents on Berkovich spaces} \label{subsec: forms and currents on Berkovich spaces}

Chambert--Loir and Ducros \cite{chambert-loir-ducros} used Lagerberg forms to introduce real valued smooth $(p,q)$-forms on a Berkovich analytic space $X$ over a non-archimedean field $K$. 
They form a bigraded sheaf $A = \bigoplus_{p,q}A^{p,q}$ of differential $\R$-algebras with differentials $d',d''$ similarly as in the complex case and as for Lagerberg forms. 
In fact, any $\alpha \in A^{p,q}$ is locally given in around a given $x \in X$ by an open neighbourhood $W$, a morphism $\varphi\colon W \to \G_{\rm m}^n$ and a smooth Lagerberg form $\alphatrop$ on $\R^n$ by using a pull-back along $\varphitrop \coloneqq \varphi \circ \trop$. 
For a good Berkovich space $X$ of pure dimension, there is integration of forms of top-degree and a formula of Stokes. 
Moreover, Chambert-Loir and Ducros define the space $D(X)$ as the topological dual of the space $A_c(X)$ of compactly supported smooth forms, again with similar properties as in the complex case and in the case of Lagerberg currents.

\subsection{Tropical toric varieties} \label{subsec: tropical toric varieties}

For computational purposes, it is desirable that a differential form on a Berkovich analytic space $X$ can be given by a global tropicalization. 
Then we can compute products directly on the tropical side. 
Jell showed in his thesis \cite{jell-thesis} that smooth forms are
indeed induced by global tropicalizations of projective varieties. 
Note that such a result is only possible if we consider extended tropicalizations $\trop\colon X \to N_\Sigma$ for a tropical toric variety $N_\Sigma$ associated to a fan $\Sigma$ in $\R^n$ based on the lattice $N=\Z^n$. 
The space $N_\Sigma$ is a compactification of $\R^n$ with boundary orbits $N(\sigma)$ associated to the cones $\sigma \in \Sigma \setminus\{0\}$. 
The space $N_\Sigma$ was studied by Mumford, and later by Kajiwara and Payne in the context of tropicalizations. 
We refer to \S \ref{subsec: notation} for a brief account. 
For example, if $\Sigma$ is the fan generated by the cone $\{x \in \R^n \mid \forall i \Rightarrow x_i \geq 0\}$, then $N_\Sigma= \R_\infty^n$. 
To generalize Lagerberg forms to $N_\Sigma$, one has to impose the condition that they are \emph{constant towards the boundary}, see \cite{jell-shaw-smacka2015} and  \cite{jell-thesis}. 
In this way, the smooth Lagerberg forms build a sheaf $A$ of bigraded differential $\R$-algebras \emph{smooth Lagerberg forms} and the smooth Lagerberg currents build a sheaf of bigraded differential $\R$-vector spaces as in the case $\R^n$.

\subsection{$\delta$-forms on $\R^n$} \label{subsec: delta-forms on euclidean space}

By tropical intersection theory, tropical cycles form a graded $\R$-algebra  which we denote by $\TZ^\cdot(\R^n)$. 
Here, the upper grading is by codimension. 
The idea of $\delta$-forms is now to combine these two algebras. 
Mihatsch \cite{mihatsch2021} showed that there is a bigraded differential $\R$ algebra 
\[
B(\R^n)= \bigoplus B^{p,q}(\R^n)
\]
with respect to natural differentials $d'$ and $d''$ which includes $A(\R^n)$ as a bigraded  $\R$-subalgebra with the same differentials. By construction, $B^{p,q}(\R^n)$ is realized as an $\R$-subspace of $D^{p,q}(\R^n)$ which also leads to the differentials $d',d''$. 
Moreover, $\TZ(\R^n)$ is an $\R$-subalgebra of $B(\R^n)$ using $C \mapsto \delta_C$, hence $\TZ^{k}(\R^n) \subset B^{(k,k)}(\R^n)$. Again, we have integration of $\delta$-forms of top-degree and the formula of Stokes holds, see \cite{mihatsch2021} for details. The construction can be done for any open subset of $\R^n$ leading to a  sheaf of bigarded differential $\R$-algebras $B$ on $\R^n$.

A Lagerberg form $\alpha$ on $\R^n$ is called \emph{piecewise smooth} if there is a locally finite polyhedral complex $\Ccal$ covering $\R^n$ such that $\alpha|_\Delta$ is the restriction of a smooth Lagerberg form from $\R^n$ for every $\Delta \in \Ccal$. Especially important for this paper are the \emph{piecewise smooth functions} which are the piecewise smooth forms of bidegree $(0,0)$. Mihatsch showed that the bigraded $\R$-algebra $\PS(\R^n)$ of piecewise smooth Lagerberg forms is a bigraded $\R$-subalgebra of $B(\R^n)$. 

\subsection{$\delta$-forms on Berkovich analytic spaces} \label{subsec: delta-forms on Berkovich analytic spaces}

Mihatsch \cite{mihatsch2021a} introduced $\delta$-forms on any  boundaryless Berkovich analytic space $X$ over a non-archimedean field $K$. Relevant in this paper is that this can be applied for any open subset $W$ of the analytification of an algebraic variety over $K$ as $W$ is then boundaryless. Mihatsch shows that there is a bigraded  differential  $\R$-subspace $B^{p,q}(X)$ of the space $D^{p,q}(X)$ of currents from \S \ref{subsec: forms and currents on Berkovich spaces} such that 
$$B(X) = \bigoplus_{p,q} B^{p,q}(X)$$
is a differential $\R$-algebra with similar properties as the algebra of real valued smooth differential forms $A(X)$. In particular, $A(X)$ is a bigraded differential $\R$-subalgebra of $B(X)$. This leads to a sheaf $B$ of bigraded differential $\R$-algebras with anticommuting differentials $d',d''$ on $X$.

By definition, \emph{piecewise smooth forms} on $X$ are locally given by piecewise smooth Lagerberg forms similarly as in the smooth case explained in \S \ref{subsec: forms and currents on euclidean space}. They form a graded $\R$-subalgebra $\PS(X)$ of $B(X)$ \cite[Proposition 4.7]{mihatsch2021a}. We define \emph{piecewise smooth functions} on $X$ as piecewise smooth forms of bidegree $(0,0)$. 

\subsection{Guideline to this paper} \label{subsec: guideline}

Let $X$ be a projective variety over a non-archime\-dean field $K$. 
{Our}  goal is to give a tropical formula for the local height of $X$ with respect to a line bundle endowed with piecewise smooth metric. 
In fact, almost all metrics occurring in diophantine geometry are of this kind. 
We were inspired by Jell's global tropicalizations of smooth $(p,q)$-forms, see \ref{subsec: tropical toric varieties}.  
Since the metrics are not smooth and since we have to work with extended tropicalizations $\psitrop\colon \Xan \to N_\Sigma$, we need to work with $\delta$-forms on tropical toric varieties $N_\Sigma$ which we will introduce and study in Section \ref{section-delta-forms}. 

Green functions are 
main players in Arakelov theory. 
To compute local heights tropically, we need a $*$-product of piecewise smooth tropical Green functions on $N_\Sigma$. 
This is 
achieved in Section \ref{sec:a tropical star-product} based on the fundamental tropical Poincar\'e--Lelong formula given in Section \ref{Section: tropical Green functions and the PL-equation}.

In Section \ref{section: Non-archimedean Arkelov theory}, we introduce the $*$-product of piecewise smooth Green functions on $\Xan$ which can be used to define the local heights mentioned above. 
We then show that finitely many piecewise smooth Green functions and also any $\delta$-form on $\Xan$ can be given by using a single extended global tropicalization $\psitrop\colon \Xan \to N_\Sigma$. 
Since the $*$-products on the tropical and on the non-archimedean side are compatible, this leads to the desired tropical formula for local heights in Section \ref{section MA-measures and local heights}. 

In the following, we will describe our results more detailed. 

\subsection{Delta-forms on tropical toric varieties} \label{subsec: delta-forms on tropical toric varieties}

In Sections \ref{polyhedra} and  \ref{forms-and-currents}, we extend  the notions of polyhedra and piecewise smooth forms to any open subset $U$ of the tropical toric variety $N_\Sigma$. 
Polyhedra on $N_\Sigma$ behave rather pathologically unless they
are constant towards the boundary. 
A \emph{polyhedral current} on $U$ is a current of the form 
\begin{equation}\label{def-polyhedral-current}
T= \sum_\Delta \alpha_\Delta \wedge \delta_{[\Delta,\mu_\Delta]\in I} \in D(U)
\end{equation}
where $I$ is a locally finite family of weighted polyhedra $[\Delta,\mu_\Delta]$ in $U$ and where $\alpha_\Delta \in A(U)$. 
If all the polyhedra $\Delta \in I$ are constant towards the boundary, then we call the polyhedral current also \emph{constant towards the boundary}.

\begin{theointro}   \label{introthm delta-forms}
There is a unique  sheaf $B=\bigoplus_{p,q}B^{p,q}$ on $N_\Sigma$ of differential $\R$-algebras with respect to differentials $d',d''$ with the following properties for any open subset $U$ of $N_\Sigma$:
\begin{enumerate}
\item 
Every $\beta \in B^{p,q}(U)$ is a polyhedral current on $U$ which is constant towards the boundary and $d'\beta,d''\beta$ agree with the differentials of currents.
\item 
The bigraded differential $\R$-algebra $B(U \cap \R^n)$ agrees with Mihatsch's definition of $\delta$-forms recalled in \S \ref{subsec: delta-forms on euclidean space}.
\end{enumerate}
\end{theointro}

This is proven in Theorem \ref{product on delta-forms}.  
The elements of $B(U)$ are called \emph{$\delta$-forms on  $U$}. 
We show in Proposition \ref{pull-back of delta-forms} that there is a pull-back of $\delta$-forms with respect to equivariant morphisms of tropical toric varieties that is compatible with products. 
There are integrals for compactly supported 
$\delta$-forms of top-degree.  
In Theorem \ref{theorem: Stokes for delta-forms}, we show the formula of Stokes for $\delta$-forms over any polyhedron in $N_\Sigma$.

An interesting question is whether a $\delta$-form $T$ on $U$ has a locally finite polyhedral complex of definition $\Ccal$ in $U$  which is constant towards the boundary. This means that we are asking for a presentation \eqref{def-polyhedral-current} with $\Delta$ ranging over $\Ccal$. 
This question is related to the combinatorial problems mentioned in \S \ref{subsec: subdivisions on tropical toric varieties} below and touched in Subsection \ref{sec:part-result-subd}. 
Based on the results from this subsection, we give a positive answer to our question in Proposition \ref{special-presentation-delta-forms} if the fan $\Sigma$ is \emph{simplicial}.

A large advantage when working with $\delta$-forms is that it is enough to work with \emph{formal $\delta$-currents} 
defined as linear functionals on the space $B_c(N_\Sigma)$ of compactly supported $\delta$-forms.
We work here with formal currents without continuity assumptions as we will not need any regularization process in this paper. 
We denote by $E^{p,q}(N_\Sigma)$ the space of formal $\delta$-currents acting on $B_c^{n-p,n-q}(N_\Sigma)$ where $n$ is the dimension of the tropical toric variety $N_\Sigma$. 

\subsection{Subdivisions on tropical toric varieties} \label{subsec: subdivisions on tropical toric varieties}

Let $U$ be an open subset of the tropical toric variety $N_\Sigma$. 
In Subsection \ref{sec:part-result-subd}, we look at the following combinatorial problem. 

\medskip 

\noindent
\textbf{Question B.} Is there a locally finite polyhedral decomposition $\Ccal$ of $U$? This means a locally finite polyhedral complex $\Ccal$ of polyhedra in $U$ which are constant towards the boundary and such that the support of $\Ccal$ is $U$.

\medskip

\stepcounter{theointro}

We cannot answer this question in general, but we have the following result.

\begin{theointro}
Let $\Sigma $ be a simplicial fan, let $U$ be an open subset of $N_\Sigma$, and let $\Tcal=(T_j)_{j \in J}$ be a locally finite family of polyhedra in $U$ that are constant towards the boundary.
Then there exists a locally finite simplicial complex $\Ccal$ of $U$ of polyhedra in $U$ which are constant towards the boundary such that 
 for all $\Delta \in \Ccal$ and all $\sigma \in \Sigma $,  every polyhedron $T\in  \Tcal$ is a locally finite union of simplices in $\Ccal$.
\end{theointro}

A slightly more general result will be proven in Theorem \ref{thm:1}. It implies a positive answer of Question B if  $\Sigma$ is \emph{simplicial}.

\subsection{Green functions on tropical toric varieties} \label{Green functions on tropical toric varieties}
Green functions are the main players in Arakelov theory. 
In \S \ref{subsection: tropical Green functions}, we introduce an analogue of Green functions on $N_\Sigma$. 
A \emph{tropical Green function} on $N_\Sigma$ is a function $g\colon \R^n \to \R$ such that for every $x \in N_\Sigma$, there is an open neighbourhood $\Omega$, a linear functional $m(y)=m_1y_1+ \dots + m_ny_n$ on $\R^n$ with $m_i\in \Z$ and a continuous function $h$ on $\Omega$ such that $g=m +h$ on $\Omega\cap \R^n$. 
Then the associated \emph{tropical toric Cartier divisor} $D$ is defined by $(\Omega,m)$, where $\Omega$ ranges over an open covering of sets as above and $m$ is the corresponding linear functional. 
In \S \ref{subsection: tropical Cartier divisors}, we recall the  notion of tropical Cartier divisors on $N_\Sigma$ introduced in Meyer's thesis \cite{meyer-thesis}. 
Relevant for us is the special case of tropical toric Cartier divisors. 
We mainly consider Green functions \emph{of piecewise smooth type} which means that the functions $h$ above are piecewise smooth. 
For such a Green function, the first Chern $\delta$-form $c_1(D,g) \in B^{1,1}(N_\Sigma)$ is given on $\Omega$ by $d'd''h$. 

For a tropical cycle $C$ in $\R^n$ with closure $\overline C$ in $N_\Sigma$, we have an associated formal $\delta$-current $\delta_{\overline C}$ of integration and we will show in Construction \ref{wedge product of function with tropical cycle} that $g \wedge \delta_{\overline C}$ makes sense as a formal $\delta$-current on $N_\Sigma$. 
We have the \emph{tropical Poincar\'e--Lelong formula}:

\begin{theointro} \label{introthm: tropical Poincare-Lelong}
Let $g$ be a Green functions for the tropical toric Cartier divisor $D$ of piecewise smooth type and let $\delta_{\overline C}$ be the current of integration for a tropical cycle $C$ in $\R^n$. 
Then we have 
\[
d'd'' g \wedge \delta_{\overline C} = c_1(D,g) - \delta_{D \cdot \overline{C}} \in E(N_\Sigma)
\]
where $D \cdot \overline{C}$ is the tropical intersection product studied in \S \ref{subsection: intersection product}.
\end{theointro}

\subsection{The star-product of tropical Green functions} \label{star-product of tropical Green functions}
Let $C$ be a tropical cycle in $\R^n$ of codimension $p$ and let $T \in E^{p-1,p-1}(N_\Sigma)$. For a Green function $g$ with associated tropical toric Cartier divisor $D$, we define the $*$-product by
$$(D,g) * (\overline C, T) \coloneqq (D \cdot \overline C, g \wedge \delta_{\overline C}+ c_1(D,g)\wedge T).$$
Here, the tropical intersection product is supported in the boundary $N_\Sigma \setminus \R^n$ and so we are forced to consider also  classical tropical cycles supported in the boundary. They are constructed from the closures of classical tropical cycles of boundary orbits $N(\sigma)$ for $\sigma \in \Sigma$. In Section \ref{sec:a tropical star-product}, we generalize the $*$-product to such tropical cycles and show that it is symmetric and satisfies a projection formula.

\subsection{Non-archimedean Arakelov theory and extended tropicalizations} \label{subsec: non-archimedean Arakelov theory and extended tropicalizations}

Let $X$ be an algebraic variety over a non-archimedean field $K$. 
Then there are Green functions of piecewise smooth type $g$ for Cartier divisors $D$ 
on $X$ and we will introduce a $*$-product similarly as in \ref{star-product of tropical Green functions}. 
They induce a first Chern $\delta$-form $c_1(D,g) \in B^{1,1}(X^\an)$ such that the Poincar\'e--Lelong formula holds. 
For Green functions of smooth type, this was done by Chambert-Loir and Ducros \cite{chambert-loir-ducros} and the generalization to the piecewise smooth case is due to Mihatsch \cite{mihatsch2021a}. 
The $*$-product for Green functions was defined in \cite{gubler-kuenne2017} in a more special setting and extends easily to Green functions of piecewise smooth type as we show in \S \ref{subsec: star-product with ps Green functions}.

Now let $\psi\colon X \to X_\Sigma$ be a morphism to a toric variety $X_\Sigma$ associated to the fan $\Sigma$ in $\R^n=N_\R$. 
Then the canonical extended tropicalization $\trop\colon X_\Sigma^\an \to N_\Sigma$ yields an extended tropicalization map $\psitrop\colon \Xan\to N_\Sigma$. 
We relate the $\delta$-forms on the tropical and non-archimedean side by showing that there is a pull-back
\[
\psi_\trop^*\colon B(N_\Sigma) \longrightarrow B(\Xan)
\]
which is a homomorphism of bigraded differential $\R$-algebras that 
is compatible with integration. 
Then we 
can state our main global tropicalization result:

\begin{theointro} \label{global tropicalization result}
Let $g_0,\dots,g_d$ be Green functions of piecewise smooth type for Cartier divisors $D_0,\dots, D_d$ on the projective variety $X$ and let $\beta_0,\dots, \beta_m$ be $\delta$-forms on $\Xan$. 
Then there is a closed embedding $\psi\colon X \to X_\Sigma$ into a projective toric variety $X_\Sigma$ with the following properties:
\begin{enumerate}
\item 
For $i=0, \dots, d$, there is a toric Cartier divisor $E_i$ on $X_\Sigma$ with associated tropical toric Cartier divisor $D_{i,\trop}$  such that $\psi^*(E_i)=D_i$.
\item 
For $i=0, \dots, d$, there is a tropical Green function $g_{i,\trop}$ of piecewise smooth type for $D_{i,\trop}$ such that $g_i=g_{i,\trop} \circ \psi_\trop$.
\item 
For $j=1,\dots, m$, there is $\beta_{j,\trop} \in B(N_\Sigma)$ with $\psi_\trop^*(\beta_{j,\trop})$. 
\end{enumerate}
\end{theointro}

This follows from Propositions \ref{global tropicalization for quasi-projective} and \ref{tropicalization of Green functions} by passing to the product of the corresponding closed embeddings. 
This shows also that we may use a multiprojective space for $X_\Sigma$. 

\subsection{Local heights of projective varieties} \label{subsec: local heights of projective varieties}

Let $X$ be a $d$-dimensional projective variety over the non-archimedean field $K$ and let $D_0,\dots, D_d$ be properly intersecting Cartier divisors on $X$. 
For $j=0,\dots, d$, let $g_i$ be a Green function for $D_i$ of piecewise smooth type and let $\widehat{D_i} \coloneqq (D_i,g_i)$. Then we define the \emph{local height of $X$} by 
\[ 
\lambda_{\widehat{D_0}, \dots, \widehat{D_d}}(X) \coloneqq \bigl((D_0,g_0) * \dots  *(D_d,g_d)* (X,0)\bigr)(1).
\]
It follows from Proposition \ref{properties local heights} that these local heights agree with the usual local heights in Arakelov theory. 
Using Theorem \ref{global tropicalization result}, we deduce in Theorem \ref{computation of local heights tropically} the following tropical formula for local heights.

\begin{theointro} \label{introthm: computation of local heights tropically}
Let $X$ be a projective $d$-dimensional variety over $K$. 
Given  $\widehat{D_0}, \dots, \widehat{D_d}$ as above, 
there exists a closed embedding $\psi\colon X \to X_\Sigma$ into a projective toric variety $X_\Sigma$ such that $\psitrop$ tropicalizes $\widehat{D_i}=(D_i,g_i)$ by a pair $(D_{i,\trop},g_{i,\trop})$ as in Theorem \ref{global tropicalization result}. 
For the tropical cycle $C=\psi_{\trop,*}(X)$, we have
\[
\lambda_{\widehat{D_0}, \dots, \widehat{D_d}}(X)=\bigl((D_{0,\trop},g_{0,\trop})* \dots * (D_{d,\trop},g_{d,\trop})*(C,0)\bigr)(1).
\]
\end{theointro}
This is indeed a \emph{tropical formula} as the right hand side is
given as the integral of a $*$-products of tropical Green functions
that can be computed completely on the tropical toric variety
$N_\Sigma$.

\subsection{Notation and Conventions}\label{subsec: notation}

The set of natural numbers $\N$ includes $0$. 
The set-theoretic inclusion $A \subset B$ allows $A=B$. 
We set $\R_{\geq 0} \coloneqq \{r \in \R \mid r \geq 0\}$ and $\Rsup \coloneqq \R \cup \{ \infty \}$. 
A \emph{lattice} is a free $\Z$-module of finite rank. 
A \emph{variety} over a field $F$ is an integral scheme which is of finite type and separated over $\Spec\, F$.

A \emph{non-archimedean field} is
a field $K$  complete with respect to
a given ultrametric absolute value $|\phantom{a}|\colon K\to \R_{\geq 0}$.
If $Y$ is a variety over $K$, then $\Yan$ denotes the analytification of $Y$ as a Berkovich analytic space. 

A \emph{fan} is a finite polyhedral complex $\Sigma$ in $N_\R$ consisting of strictly convex polyhedral rational cones. For a cone $\sigma$ in the fan $\Sigma$, write $N(\sigma)\coloneqq N_\R / \langle\sigma\rangle_\R$.  
We have canonical projection maps $\pi_{\sigma} \colon N_{\R}
\to N(\sigma)$ and $\pi_{\sigma, \tau} \colon N(\tau)
\to N(\sigma)$ for $\tau \prec \sigma$.

We denote by $X_\Sigma$ the corresponding toric variety with dense torus $\torus$ and by $N_\Sigma$ the corresponding partial compactification of $N_\R$. We recall that
\begin{align*}
	N_\Sigma \coloneqq  \coprod \limits_{\sigma \in \Sigma } N(\sigma)
\end{align*}
as a set. For the topology on $N_\Sigma$ and other details, we refer to \cite{burgos-gubler-jell-kuennemann1}. We call $N_\Sigma$  the \emph{tropical toric variety associated to $\Sigma$}.

For a fan $\Sigma$, we call \emph{star of $\sigma$} and denote $\Star_\Sigma(\sigma)$ the set of cones $\tau \in \Sigma$ with $\sigma \prec \tau$. 
We note that $\Sigma(\sigma) \coloneqq \{\pi_\sigma(\tau)\mid \tau \in \Star_\Sigma(\sigma)\}$ is a fan in $N(\sigma)$ and there is a canonical bijective correspondence between
$\Star_\Sigma(\sigma)$ and $\Sigma(\sigma)$ given by $\pi_\sigma$. 
Note that in the literature on toric varieties, often $\Sigma(\sigma)$ is called the star of $\sigma$ as it is the fan of the orbit closure associated to $\sigma$.

For tropical toric varieties, we will consider the following morphisms. 
	Let $N_{\Sigma'}'$ be another tropical toric variety with underlying lattice $N'$ of rank $n'$ and 
	let $L\colon N'_\R \to N_\R$ be a linear map. We call a map $F\colon N'_{\Sigma'} \to N_\Sigma$ \emph{$L$-equivariant} if it is continuous and $L$-equivariant in the sense that $F(x'+y')=F(x')+L(y')$ for every $x'\in N_{\Sigma '}$ and
	$y'\in N_\R'$.
Equivalently, we can consider an affine map $E\colon N_\R' \to N(\sigma)$ for some $\sigma \in \Sigma$ with $E(x'+y')=E(x')+\pi_\sigma(L(y'))$ for every $x'\in N'$ and
$y'\in N_\R'$ such that for
each cone $\tau '\in \Sigma'$ there is a cone $\tau \in \Sigma
(\sigma )$ such that  $\pi _{\sigma }(L(\tau '))\subset \tau
$. The equivalence is given by noting that $F_{N_\R'}$ is induced by such an $E$ and conversely, $F$ is the unique extension of $E$ to a continuous map $N_{\Sigma'}'\to N_\Sigma$.

\subsection{Acknowledgements}
We would like to thank Andreas Gross for helpful remarks.

\section{Polyhedra on tropical toric varieties}\label{polyhedra}
 
We start this section with a subsection on polyhedra on  tropical toric varieties. 
We will see that these polyhedra behave rather pathologically unless they are constant towards the boundary. 
Since polyhedra play a crucial role in this paper, we advice the reader to have a careful look at the first subsection. 
In the second subsection, we will introduce polyhedral complexes. 
Then we look at the problem of subdivisions mentioned in \S \ref{subsec: subdivisions on tropical toric varieties}. 
For those less interested in such combinatorial problems, this subsection can be skipped in a first reading.
 
In Section \ref{polyhedra}, we fix the following notation: Let $N$ be a free abelian group of rank $n$, $M=\mathrm{Hom}_\Z(N,\Z)$ its dual and denote by $N_\R$  and $M_\R$ the scalar extensions to $\R$. 
We consider a fan $\Sigma$ with associated \emph{tropical toric variety $N_\Sigma$}, see \S \ref{subsec: notation}.

\subsection{Polyhedra}
\label{subsection: polyhedra on tropical toric varieties}

A \emph{polyhedron $\Delta$} in $N_\R$ is by definition an intersection of finitely many closed half-spaces $\{v \in N_\R \mid m_i(v) \geq c_i\}$ with $m_i \in  M_\R$ and $c_i \in \R$. 
If we can choose all $m_i \in M$ and all $c_i$ in a subgroup $\Gamma$ of $\R$, then we call $\Delta$ a $(\Z,\Gamma)$-polyhedron.
	The empty set and $N_\R$ are polyhedra, although the empty set is special and sometimes has to be treated separately. 
A bounded polyhedron in $N_\R$ is called a \emph{polytope}. 
Given a non-empty polyhedron $\Delta$ in $N_\R$, we denote by $\A_\Delta$ the affine subspace of $N_\R$ generated by $\Delta$ and by $\L_\Delta$ the linear subspace of $N_\R$ underlying $\A_\Delta$.  
We define $\dim\Delta\coloneqq \dim_\R\L_\Delta$ and $\dim\Delta\coloneqq -1$ if $\Delta =\emptyset$.

If $\Delta \not =\emptyset$, then we denote by
\[
\mathrm{rec}(\Delta)\coloneqq \{v\in N_\R\,|\,\Delta+v\subset \Delta\}
\]
the \emph{recession cone} of $\Delta$.
We also denote by
$\relint(\Delta )$ the \emph{relative interior} of $\Delta $, that is, the interior of $\Delta $ as a subset of $\A_{\Delta }$ when $\Delta$ is non-empty and the empty set otherwise.  

A \emph{polyhedron} $\Delta$ in $N_\Sigma$ is  defined as the topological closure in $N_\Sigma$ of a polyhedron $\Delta(\sigma)$ in $N(\sigma)$ for some $\sigma\in \Sigma$.
If $\Delta $ is non-empty, then the cone $\sigma$ and the polyhedron $\Delta(\sigma)$ are uniquely determined by $\Delta$. 
In this situation, we say that $\Delta$ has \emph{sedentarity} $\mathrm{sed}(\Delta)\coloneqq\sigma$ and call $\Delta(\sigma)$ the \emph{finite part of $\Delta$}. 
We define $\dim\Delta=\dim \Delta(\sigma)$.
Given a polyhedron $\Delta$ in $N_\Sigma$, we write
$\Delta(\tau)\coloneqq N(\tau)\cap \Delta$ for each $\tau\in \Sigma$
and get
  \[
\Delta=\bigcup\limits_{\genfrac{}{}{0pt}{}{\tau\in\Sigma}
{\mathrm{sed}(\Delta)\prec\tau}}\Delta(\tau).
\]
A \emph{face} of a polyhedron $\Delta$ in $N_\Sigma$ is defined as the
topological closure in $N_\Sigma$ of a face of a polyhedron
$\Delta(\tau)$ for some $\tau\in \Sigma$.

Let $\Delta$ be a polyhedron in $N_\Sigma$ of sedentarity $\sigma \in \Sigma$. 
The following lemma, describes the closure $\Delta$ of the finite part
$\Delta(\sigma)$ precisely. 
We pick any $\tau \in \Sigma$ which has $\sigma$ as a face and we
consider again the canonical projection
$\pi_{\tau,\sigma } \colon N(\sigma) \to N(\tau)$.  We write $\tau '$
for the image of $\tau $ in $N(\sigma )$.  Recall that any non-empty
polyhedron in an euclidean space is the sum of its recession cone and
of a polytope \cite[Theorem 19.1]{rockafellar1970}.  Let
$C=\rec(\Delta (\sigma ))$, so $\Delta (\sigma )=K+C$ for some
polytope $K$ in $N(\sigma)$.

\begin{lemma} \label{lemm:1}
In the above notation, the polyhedron $\Delta$  has the following properties:
\begin{enumerate}
\item \label{item:1} 
The subset $\Delta (\tau )$ is non-empty if and only if $C\cap \relint(\tau ')$ is non-empty.
\item \label{item:2} 
If $C\cap \relint(\tau ')\not = \emptyset$, then $\Delta (\tau)=\pi _{\tau ,\sigma }(\Delta (\sigma ))$. 
Therefore, we have
\begin{displaymath}
\pi _{\tau ,\sigma }^{-1}(\Delta (\tau ))=\Delta (\sigma )+\L_{\tau '}.
\end{displaymath}
\end{enumerate}
\end{lemma}
\begin{proof}
If $\Delta$ has sedentarity $(0)$, this is a result of Osserman and Rabinoff \cite[Lemma 3.9]{osserman-rabinoff2013}. 
The general case follows from that replacing $N_\Sigma$ by the tropical toric variety given by the closure of $N(\sigma)$.
\end{proof}

\begin{cor}\label{cor:2} 
Let $\sigma \prec \tau \prec \rho $ be cones of $\Sigma $ and $\Delta $ a polyhedron of $N_{\Sigma}$ of sedentarity $\sigma $. 
Then $\overline{\Delta (\tau )}(\rho )$  is either empty or agrees with $\Delta (\rho )$. 
\end{cor}

\begin{proof}
Denote by $\tau '$ and $\rho' $ the images of $\tau $ and $\rho $ in $\Sigma (\sigma )$ and by $\rho ''$ the image of $\rho $ in $\Sigma (\tau )$ (see \ref{subsec: notation}.
Assume that $\overline{\Delta (\tau )}(\rho )\not = \emptyset$, then $\Delta (\tau )\not = \emptyset$ and $\Delta (\rho )\not = \emptyset$. 
Therefore, applying  Lemma \ref{lemm:1} three times, we deduce
\begin{displaymath}
\overline{\Delta (\tau )}(\rho )=\pi _{\rho ,\tau }(\pi _{\tau,\sigma }(\Delta (\sigma )))=\pi _{\rho ,\sigma }(\Delta (\sigma)) = \Delta (\rho )   
  \end{displaymath}
proving the claim. 
\end{proof}

\begin{lemma}\label{lemm:2}
   In the real vector space $N_\R$, we consider  a strictly convex
  polyhedral cone $C$,  a compact set $\Omega$,  a polyhedron $K$
  and $v\in \relint(C)$. Then there exists a real number $\lambda
  \ge 0$ such that
  \begin{displaymath}
    (K+\L_{C})\cap (\Omega +\lambda v)\subset K+C.
  \end{displaymath}
\end{lemma}
\begin{proof} For any real number $\lambda \ge 0$, 
  if $p\in (K+\L_{C})\cap (\Omega +\lambda v)$ then
  $p\in \big((K+\L_{C})\cap \Omega\big)+\lambda v$. Therefore, after
  replacing $\Omega$ by $ (K+\L_{C})\cap \Omega$, we can assume that
  $\Omega \subset K+\L_{C}$. If we prove that the function 
  $\varphi\colon K+\L_{C}\to  \R_{\ge 0}\cup \{\infty\}$ given by
  \begin{displaymath}
    p\mapsto \inf\{ \lambda \in \R_{\ge 0}\mid p+\lambda v\in K+C\}
  \end{displaymath}
 is 
  finite and continuous, then compactness of $\Omega $ would yield
  the result for $\lambda \coloneqq \max\{\varphi(y)\mid y \in \Omega\}$. Since $v\in \relint (C)$, for every vector $w\in \L_{C}$
  there is a $\lambda \in \R_{\ge 0}$ such that $w+\lambda v \in
  C$. Therefore the function $\varphi$ is finite. 

  On $(K+\L_{C})\times \R$ we consider the subset of
  points
  \begin{displaymath}
    E_{0}=\{(p,\lambda )\in (K+\L_{C})\times \R\mid p+\lambda v \in K+C\}
  \end{displaymath}
  This is a closed convex set because it is the preimage of a closed
  convex set by an affine function. Let us consider the closed convex subset $E=E_{0}\cap ((K+\L_{C})\times
  \R_{\ge 0})$.  By construction, $E$
  is the epigraph of the function $\varphi$. Since $E$ is closed,
  $\varphi$ is a closed convex function in the sense of \cite[\S
  7]{rockafellar1970}. Since $K$ is a polyhedron, by \cite[Theorem
  10.2]{rockafellar1970} the function $\varphi$ is continuous
  finishing the proof of the lemma. 
\end{proof}

For a cone $\sigma \in \Sigma$, recall from \S \ref{subsec: notation} that $\Sigma(\sigma)$ is the fan in $N(\sigma)$ given as the image of $\Sigma$ with respect to the morphism $\pi_\sigma\colon N_\R \to N(\sigma)$. 
The following lemma is similar to \cite[Remark 3.3]{osserman-rabinoff2013}.

\begin{lemma}\label{lemm:5}
Let $\Delta $ be a non-empty polyhedron in $N_{\Sigma }$ of sedentarity $\sigma $. 
Then $\Delta $ is compact if and only if $\rec(\Delta(\sigma))\subset |\Sigma (\sigma)|$. 
\end{lemma}

\begin{proof}
By \cite[III. Theorem 2.8]{Ewald:convexity} any fan is a subfan of a complete fan. 
Let $\Sigma '$ be a complete fan in $N(\sigma )$ with $\Sigma(\sigma ) \subset \Sigma '$. 
Let $\Delta' $ be the closure of $\Delta $ in $N(\sigma )_{\Sigma '}$. 
Then $\Delta '$ is compact. 
Since
\begin{displaymath}
|\Sigma '|\setminus |\Sigma(\sigma ) |=\bigcup_{\tau \in \Sigma '\setminus \Sigma(\sigma ) }\relint(\tau ),  
\end{displaymath}
we deduce  that $\rec(\Delta(\sigma ) )\subset |\Sigma(\sigma ) |$ if and only if, for  any cone $\tau \in \Sigma '\setminus \Sigma(\sigma ) $, the condition $\rec(\Delta(\sigma ) )\cap \relint(\tau )=\emptyset$ is satisfied. 
By Lemma \ref{lemm:1}~\ref{item:1}, this condition is equivalent to $\Delta'(\tau )=\emptyset$ for all  $\tau \in \Sigma '\setminus \Sigma(\sigma )$, which in turn is equivalent to $\Delta '\subset N(\sigma )_{\Sigma (\sigma )}$. 
So equivalent to the condition $\Delta =\Delta '$ and we conclude that $\Delta $ is compact if and only if $\rec(\Delta(\sigma))\subset |\Sigma (\sigma )|$.
\end{proof}

\begin{definition} \label{definition: constant towards the boundary for polyhedra}
Let $\Delta $ be a polyhedron in $N_{\Sigma }$ of sedentarity $\sigma $. 
We say that $\Delta$ is \emph{constant towards the boundary} if for every cone $\tau$ in $\Sigma$ such that  $\tau \succ \sigma$ with projection $\pi_{\tau,\sigma } \colon N(\sigma) \to N(\tau)$ and every $p \in N(\tau)_\R$, we have 
\[
\pi_{\tau,\sigma}^{-1}(\Delta(\tau))\cap \Omega
=\Delta(\sigma)\cap \Omega
\] 
for some open neighbourhood $\Omega$ of $p$ in $N_\Sigma$. 
\end{definition}

\begin{prop}\label{recession-cone-boundary-condition}
  Let $\Delta$ be a polyhedron of sedentarity $\sigma$. Then $\Delta $ 
  is constant towards
  the boundary if and only if
  $\rec(\Delta(\sigma))\cap |\Sigma(\sigma )|$
   is a union of cones from
  $\Sigma(\sigma)$.
\end{prop}

\begin{proof}
By definition, $\Delta $ is the closure on $N_{\Sigma }$ of $\Delta(\sigma ) \subset N(\sigma )$. 
Moreover $\Delta(\sigma )=K+C$, where $K$ is a compact polyhedron and $C=\rec(\Delta (\sigma ))$ is a finitely generated cone \cite[Theorem 19.1]{rockafellar1970}. 
Assume that $\Delta $ is constant towards the boundary. 
Let $\tau$ be a cone in $\Sigma$ with $\sigma \prec \tau $. 
We first assume that $\Delta(\tau )\not =\emptyset$. 
We pick any $p\in \Delta(\tau)$.
Let $\pi _{\tau ,\sigma}\colon N(\sigma)\to N(\tau)$ be the projection and $\tau '=\pi _{\sigma }(\tau )$.
Since $\Delta $ is constant towards the boundary, there exists an open neighbourhood $\Omega$ of $p$ such that
\[
\pi_{\tau,\sigma}^{-1}(\Delta(\tau))\cap \Omega
=\Delta(\sigma)\cap \Omega.
\] 
By definition of the topology of $N_{\Sigma }$ (see \cite[Remark 3.1.2]{burgos-gubler-jell-kuennemann1}), we can assume that
$\Omega \cap N(\sigma )$ is of the form $\Omega \cap N(\sigma)=q+\Omega_0+\tau' $, where $\Omega_0$ is a neighbourhood of $0$ in a subspace $\L_{\tau' }^{\perp}$ of $N(\sigma )$ such that $\L_{\tau' }\oplus \L_{\tau' }^{\perp}=N(\sigma )$ and $q\in N(\sigma )$ is a point with $\pi _{\tau,\sigma}(q)=p$. 
Then using that $\Delta $ is constant towards the boundary, we deduce
\begin{displaymath}
q+\tau' \subset \pi_{\tau,\sigma}^{-1}(\Delta(\tau))\cap \Omega
=\Delta(\sigma )\cap \Omega
\end{displaymath}
and hence $q+\tau' \subset \Delta$. 
In particular, $q\in \Delta $ and $\tau' \subset \rec(\Delta(\sigma))$. 
If $\Delta(\tau )= \emptyset$, then Lemma \ref{lemm:1}~\ref{item:1} yields $\relint(\tau')\cap\rec(\Delta(\sigma ))=\emptyset$.
Note that if $\tau$ runs through $\Star_\Sigma(\sigma)$, then $\tau'$ runs through the fan $\Sigma(\sigma)$ of $N(\sigma)$.  
We conclude that for every cone $\tau '\in \Sigma(\sigma)$ either
$\tau '\subset \rec(\Delta(\sigma ))$ or
$\relint(\tau' )\cap \rec(\Delta(\sigma) )=\emptyset$.  
This, together with the properties of a fan, implies that $\rec(\Delta(\sigma ) )\cap |\Sigma(\sigma )|$ is a union of cones of $\Sigma(\sigma )$.

We now prove the converse. 
Assume that $\rec(\Delta(\sigma ) )\cap |\Sigma(\sigma )|$ is a	union of cones of $\Sigma(\sigma )$.
Let  $\tau \in \Sigma $ with $\sigma \prec \tau $. 
We consider again the cone $\tau' =\pi_\sigma(\tau)$ of $\Sigma(\sigma)$.
If $\relint(\tau' )\cap \rec(\Delta(\sigma ) )=\emptyset$, then Lemma \ref{lemm:1}~\ref{item:1} yields $\Delta(\tau )=\emptyset$. 
Therefore $\Delta$ satisfies trivially the condition of being constant towards the boundary for all points of $N(\tau )$. 
On the other hand, if $\relint(\tau' )\cap\rec(\Delta(\sigma ))\neq\emptyset$, then our assumption that $\rec(\Delta(\sigma))\cap |\Sigma(\sigma )|$ is a union of cones from $\Sigma(\sigma)$ yields that $\tau' \subset \rec(\Delta(\sigma ))$.
Let $p\in \Delta(\tau )$. 
Then, by Lemma \ref{lemm:1}~\ref{item:2}, we can find a point $q\in \Delta (\sigma )$ with $\pi_{\tau,\sigma  }(q)=p$. 
Let $\Omega_0$ be a compact neighbourhood of $0$ on a complementary subspace $\L_{\tau '}^{\perp}\subset N(\sigma )$ as above and  $v\in \relint(\tau ')$. 
By Lemma \ref{lemm:1} \ref{item:2} and Lemma \ref{lemm:2} we can find a $\lambda>0$ such that
\begin{align*}
\pi ^{-1}_{\tau ,\sigma }(\Delta (\tau ))\cap (q+\Omega_0 +\lambda v)
&=(\Delta (\sigma )+\L_{\tau '})\cap (q+\Omega_0 +\lambda v)\\
&\subset \Delta(\sigma ) +\tau '=\Delta (\sigma ).
\end{align*}
By the description of the topology of $N_\Sigma$ given in \cite[Remark 3.1.2]{burgos-gubler-jell-kuennemann1}, there is a neighbourhood $\Omega$ of $p$ in $N_\Sigma$ such that $\Omega \cap N(\sigma)=q+\lambda v +\Omega_0 +\tau '$. 
By our choice of $\lambda$, we obtain
\begin{displaymath}
\pi ^{-1}_{\tau ,\sigma }(\Delta (\tau ))\cap\Omega  
=\pi ^{-1}_{\tau ,\sigma }(\Delta (\tau ))\cap (q+\Omega_0 +\lambda	v+\tau ')\subset \Delta (\sigma )\cap \Omega.
\end{displaymath}
The inclusion $\Delta (\sigma )\cap \Omega\subset \pi^{-1}_{\tau
  ,\sigma }(\Delta (\tau ))\cap\Omega  $ is clear, so we deduce that
$\Delta $ also satisfies the condition of being constant towards the
boundary for the cone $\tau$. 
\end{proof}

  \begin{rem}\label{rem:2} Proposition
    \ref{recession-cone-boundary-condition} means that a 
    polyhedron $\Delta $ in $N_{\Sigma }$ of sedentarity $\sigma $ is
    constant towards the 
    boundary if and only if the fan $\Sigma(\sigma ) $ and the
    polyhedron $\Delta (\sigma )$ in $N(\sigma )$ are compatible in
    the sense of \cite{osserman-rabinoff2013}.
  \end{rem}

As a consequence of the previous proposition, being constant towards the boundary is stable under taking faces.

  \begin{cor}
    \label{polyhedraonconstant towards the boundary and faces}
    Let $\Delta$ be a polyhedron in $N_\Sigma$ of sedentarity $\sigma$.
    If $\Delta$ is constant towards the boundary, and $F$ is a face of
    $\Delta (\sigma )$, then $\overline F$ is a 
    polyhedron of sedentarity $\sigma $ which is constant towards the
    boundary. 
  \end{cor}

\begin{proof}
  Since $F$ is a face of $\Delta (\sigma )$, the cone $\rec(F)$ is
  a face of $\rec(\Delta (\sigma ))$. By Proposition
  \ref{recession-cone-boundary-condition} $\rec(\Delta (\sigma ))$
  is a union of cones of $\Sigma $. Therefore $\rec(F)$ is a union of faces of
  cones of $\Sigma $. 
  Again by Proposition
  \ref{recession-cone-boundary-condition},
  $\overline F$ is constant towards the boundary.
\end{proof}

\begin{cor} \label{polyhedraonconstant towards the boundary and orbit}
Let $\Delta$ be a polyhedron in $N_\Sigma$ of sedentarity $\sigma$ for
a face $\sigma$ of a cone $\tau \in \Sigma$. 
If $\Delta$ is constant towards the boundary, then $\Delta(\tau)$ is
a polyhedron of sedentarity $\tau$ which is constant towards the
boundary. 
\end{cor}
  
\begin{proof}
Let $\Delta $ be a polyhedron of sedentarity $\sigma $ that is
constant towards the boundary and $\tau \in \Sigma $ with $\sigma
\prec \tau $.  
If $\Delta (\tau )$ is the empty set, then it is constant towards the boundary. 
So we can assume that $\Delta(\tau )$ is non-empty. 
Let $\tau '$ be the image of $\tau $ in $N(\sigma )$. 
By Lemma \ref{lemm:1}, $\rec(\Delta(\sigma ) )\cap\relint(\tau ')\not = \emptyset$. 
Since $\Delta $ is constant towards the boundary, by Proposition
\ref{recession-cone-boundary-condition}, $\rec(\Delta(\sigma ))\cap|\Sigma
(\sigma )|$ is a union of cones of $\Sigma (\sigma )$. This implies
that $\tau '\subset \rec(\Delta (\sigma ))$. 

We know that $\Delta (\tau )=\pi _{\tau ,\sigma }(\Delta (\sigma ))$
and  $\rec(\Delta (\tau ))=\pi _{\tau ,\sigma }(\rec(\Delta(\sigma )))$. 
Moreover, $\Sigma (\tau )=\pi _{\tau ,\sigma }(\Star_{\Sigma(\sigma )}(\tau '))$. 
By assumption, $\rec(\Delta(\sigma )) \cap |\Sigma (\sigma)|$ is a union of cones of $\Sigma (\sigma )$. 
Therefore $\rec(\Delta(\sigma ))\cap |\Star_{\Sigma (\sigma )}(\tau
')|$ is a union of cones of $\Star_{\Sigma (\sigma )}(\tau ')$.
Each cone of $\Sigma (\tau )$ is the image by $\pi _{\tau ,\sigma }$
of a cone of $\Star_{\Sigma(\sigma )}(\tau ')$. 
Hence, if we prove that
\begin{equation}\label{eq:3}
\pi _{\tau ,\sigma }(\rec(\Delta(\sigma ) ))\cap \pi _{\tau ,\sigma }(|\Star_{\Sigma (\sigma )}(\tau ')|)
=\pi _{\tau ,\sigma }(\rec(\Delta(\sigma ) )\cap |\Star_{\Sigma (\sigma )}(\tau ')|),
\end{equation}
then $\rec(\Delta (\tau ))\cap |\Sigma (\tau )|$ will be a union  of cones of $\Sigma (\tau )$. 
Thus we are reduced to prove \eqref{eq:3}.

The inclusion `$\supset$' is clear, so we only need to prove the inclusion `$\subset$'. 
Let $p\in \pi _{\tau ,\sigma }(\rec(\Delta (\sigma )))\cap \pi _{\tau ,\sigma}(|\Star_{\Sigma (\sigma )}(\tau ')|)$. 
Then there exist points $q_{1}\in \rec(\Delta (\sigma ))$ and $q_{2}\in |\Star_{\Sigma (\sigma)}(\tau ')|$ with $p=\pi _{\tau ,\sigma }(q_{1})=\pi _{\tau,\sigma }(q_{2})$. 
Therefore $q_{1}-q_{2}\in \L_{\tau'}$. 
Since $\L_{\tau '}=\tau '-\tau '$, there exist $v_{1},v_{2}\in \tau '$ with $q_{1}+v_{1}=q_{2}+v_{2}$. 
Call this point $q$. 
Since $\tau' \subset \rec(\Delta(\sigma ))$, 
we deduce that $q=q_1+v_1\in\rec(\Delta (\sigma ))$. 
Since each cone of $\Star_{\Sigma (\sigma)}(\tau ')$ contains $\tau '$, we deduce that $q=q_2+v_2\in |\Star_{\Sigma (\sigma )}(\tau ')|$. 
Clearly $p=\pi _{\tau,\sigma }(q)$, so $p\in \pi _{\tau ,\sigma }(\rec(\Delta (\sigma ))\cap|\Star_{\Sigma (\sigma )}(\tau ')|)$. 
Hence equation \eqref{eq:3} is satisfied.
\end{proof}

\begin{cor} \label{face also constant towards the boundary}
Let $\Delta $ be a polyhedron of $N_\Sigma$ and $F$ a face of $\Delta $. 
If $\Delta $ is constant towards the boundary, then the same is true for $F$.
\end{cor}

\begin{proof}
This is a combination of Corollaries \ref{polyhedraonconstant
    towards the boundary and faces} and \ref{polyhedraonconstant
    towards the boundary and orbit}.
\end{proof}

The polyhedra on $N_\Sigma$ which are constant towards the boundary satisfy many desirable properties, while as we will see latter, polyhedra that are not constant towards the boundary are somehow pathological.

\begin{cor}\label{cor:4} Let $\Delta $ be a polyhedron on $N_{\Sigma }$ of sedentarity $\sigma $ and constant towards the boundary. 
Let $\tau \in \Sigma$ with $\sigma \prec \tau $ and $\tau '$ the image of $\tau $ in $\Sigma (\sigma )$. 
Then the map that sends $F$ to $\overline F(\tau)$  is a bijection between the set of faces $F$ of $\Delta (\sigma)$ with $\rec(F)\cap \relint(\tau' )\not = \emptyset$ and the set of faces of $\Delta (\tau )$.  
\end{cor}

\begin{proof}
  We start by showing that, if $F$ is a face of $\Delta(\sigma)$ such that
  $\rec(F)\cap \relint(\tau' )\not = \emptyset$, then $\overline
  F(\tau )$ is a face of $\Delta (\tau )$. If $F=\Delta $ or
  $F=\emptyset$ there is nothing to prove. Let $h$ be an  affine
  function on $N(\sigma)$ such that $h|_{F}=0$ and $h|_{\Delta \setminus F}>0$, and
  let $H$ be the set of zeros of $h$. Since $F\subset H$, we have
  $\rec(H)\cap \relint(\tau' )\not = \emptyset$. By Lemma \ref{lemm:1}, we get $H(\tau )\not =
  \emptyset$. 
  
  We claim $\overline F(\tau )=\overline H(\tau )\cap \Delta (\tau
  )$. The inclusion $\subset $ is clear.  Let $x\in \overline H(\tau )\cap \Delta (\tau
  )$. By Lemma \ref{lemm:1}~\ref{item:2}, there is a point $y\in \Delta(\sigma ) $
  with $\pi _{\tau ,\sigma }(y)=x$ and a vector $v\in \L_{\tau '}$
  with $y+v\in H$. Let $w\in \rec(F)\cap \relint(\tau' )$. Since
  $\rec(F)\subset \rec(H)$, for any $\lambda >0$, $y+v+\lambda w\in
  H$. Since $\Delta $ is constant towards the boundary and
  $\rec(\Delta (\sigma ))\cap \relint(\tau ')\not = \emptyset$, by Proposition
  \ref{recession-cone-boundary-condition}, $\tau
  '\subset \rec(\Delta (\sigma ))$. For $\lambda \gg 0$, $v+\lambda w\in \tau
  '$, thus $y+v+\lambda w\in \Delta(\sigma ) $.  We conclude that for $\lambda \gg 0$, we have
  $y+v+\lambda w\in H\cap \Delta (\sigma )=F$. By Lemma  \ref{lemm:1}, we get $x\in \overline F(\tau )$ proving the claim.
  
From the previous claim, we will deduce  that $\overline F(\tau)$ is a face of $\Delta(\tau)$. 
We recall from Lemma \ref{lemm:1} that $\overline H(\tau)=\pi_{\tau,\sigma}(H)$. 
Using that $H$ is an affine hyperplane, we have either $\overline H(\tau)= N(\tau)$ or $\overline H(\tau)$ is an affine hyperplane of $N(\tau)$. 
In the first case, we get $\overline F(\tau )=\overline H(\tau )\cap \Delta (\tau)=\Delta(\tau)$ which is a face of $\Delta(\tau)$. 
In the second case, the affine hyperplane $\overline H(\tau)$ is given by $b=0$ for an affine function $b$ on $N(\tau)$. 
Clearly, we may assume that $a= b \circ \pi_{\tau,\circ}$. 
Then $\Delta(\tau)= \pi_{\tau,\sigma}(\Delta(\sigma))$ is contained in the half-space $b \geq 0$ of $N(\tau)$. 
We deduce from $F(\tau )=\overline H(\tau )\cap \Delta (\tau)$ that $F(\tau)$ is a face of $\Delta(\tau)$. 

Let now $G$ be a non-empty face of $\Delta (\tau )$. 
We have to show that there exists a unique face $F$ of $\Delta (\sigma )$ with $\rec(F)\cap \relint(\tau ')\not = \emptyset$ such that $\overline F(\tau )=G$.
Since $\Delta $ is constant towards the boundary, there exists an open neighborhood $\Omega$ of $p$ in $N_\Sigma$, such that 
\begin{equation}\label{eq:5}
\pi_{\tau,\sigma}^{-1}(\Delta(\tau))\cap \Omega
=\Delta(\sigma)\cap \Omega.
\end{equation}
Since $G$ is a face of $\Delta (\tau)$, the equality
\eqref{eq:5} implies that there is a unique face $F$ of $\Delta
(\sigma )$ such that  $    \pi_{\tau,\sigma}^{-1}(G(\tau))\cap \Omega
=F\cap \Omega$. 
By Lemma \ref{lemm:1}, the face $F$ satisfies $\overline F(\tau )\cap \Omega=G\cap \Omega $ and is unique satisfying this property. 
Using $G\not = \emptyset $, we know that  $\rec(F)\cap \relint (\tau ')\not=\emptyset$.  
Since $G$ is path-connected, the unicity of $F$ implies that $F$ is
independent of the choice of the point $p$ and that $\overline{F}(\tau
)=G$. 
\end{proof}

\begin{rem} \label{equivalent condition for faces}
Note that, when $F$ is constant towards the boundary, the
condition $\rec(F)\cap \relint(\tau' )\not = \emptyset$ in Corollary \ref{cor:4} is equivalent to the condition $\tau' \subset \rec(F)$. 
We only have to show that the first condition implies the second. 
As $F$ is constant towards the boundary by Corollary \ref{face also constant towards the boundary}, we have that $|\Sigma(\sigma)| \cap \rec(F)$ is a union of cones $\sigma_i$ from $\Sigma(\sigma)$ by Proposition \ref{recession-cone-boundary-condition}. 
If $\rec(F)\cap \relint(\tau' )\not = \emptyset$, then $\relint(\tau')$ intersects one of the $\sigma_i$. As $\tau' \in \Sigma(\sigma)$ as well, we conclude that $\tau'$ is a face of this $\sigma_i$ thereby proving the claim.
\end{rem}

\begin{cor}\label{cor:3} Let $\Delta $ be a polyhedron on $N
  _{\Sigma }$ which is constant towards the boundary, let $F_1$ be a face of $\Delta$ and let $F_2 \subset F_1$. Then $F_2$ is a face of $F_1$ if and only if
  $F_2$ is a face of $\Delta $. 
\end{cor}

\begin{proof}
Let $\tau$ be the sedentarity of $F_1$, then $\sigma \prec \tau$, $F_1(\tau)$ is dense in $F_1$ and $F_{1}(\tau)$ is a face of $\Delta(\tau )$. 
	
Suppose first that $F_2$ is a face of $F_1$. 
Let $\rho \in \Sigma$ be the sedentarity of $F_2$, then $\tau\prec \rho $ and $F_2(\rho)$ is a face of $F_1(\rho)$.  
We have to show that $F_{2}(\rho)$ is a face of $\Delta(\rho)$. 
Since for ordinary polyhedra in real vector spaces a face of a face is a face, it is enough to show that  ${F_{1}}(\rho )$ is a face of $\Delta (\rho )$, but this is part of Corollary \ref{cor:4}. 
	
Conversely, assume that $F_2$ is a face of $\Delta$. 
Let $\rho \in \Sigma$ be the sedentarity of $F_2$. 
Then $F_2(\rho)$ is dense in $F_2$. 
Using $F_2 \subset F_1$, we conclude that $\tau \prec \rho$. 
By assumption, $F_2(\rho)$ is a face of $\Delta(\rho)$ and as above, we have that $F_1(\rho)$ is a face of $\Delta(\rho)$. 
Since $F_2(\rho) \subset F_1(\rho)$, we conclude that $F_2(\rho)$ is a face of $F_1(\rho)$ which is clear for ordinary polyhedra again. 
By definition, we get that $F_2$ is a face of $F_1$. 
\end{proof}

\begin{cor} \label{cor:transitivity} 
Let $\Delta $ be a polyhedron in $N_{\Sigma }$ of sedentarity $\sigma \in \Sigma $ that is constant towards the boundary. 
Let $\tau ,\rho \in \Sigma $ be cones such that $\sigma \prec\tau \prec \rho $. 
Then
\begin{displaymath}
\overline{\Delta (\tau )}(\rho )=\Delta (\rho ).
\end{displaymath}
\end{cor}

\begin{proof} 
By Corollary \ref{cor:2}, we only need to prove that if $\Delta (\rho )\not = \emptyset$, then $\overline{\Delta(\tau )}(\rho )\not = \emptyset$. 
So assume that $\Delta (\rho )\not = \emptyset$. 
As we have seen in the proof of Corollary \ref{polyhedraonconstant towards the boundary and orbit}, we have $\rho' \subset \rec(\Delta (\sigma ))$ for the cone $\rho'$ in $N(\sigma)$ induced by $\rho$. 
For the cone $\tau'$ of $N(\sigma)$ induced by $\tau$, we conclude that $\tau'\subset \rec(\Delta (\sigma ))$. 
By Lemma \ref{lemm:1}, we get $\Delta (\tau  )\not = \emptyset$ and $\pi_{\tau,\sigma}(\Delta(\sigma))=\Delta(\tau)$. 
Since $\rec(\Delta (\tau ))=\pi _{\tau ,\sigma }(\rec(\Delta(\sigma )))$, the cone $\rho''$ of $N(\tau)$ induced by $\rho$ satisfies $\rho ''\subset \rec(\Delta (\tau ))$.  
By Lemma \ref{lemm:1}, we get $\overline{\Delta(\tau )}(\rho )\not = \emptyset$.
\end{proof}

The following result was shown by Osserman and Rabinoff \cite[Corollary 3.11]{osserman-rabinoff2013} when both polyhedra $\Delta,\Delta'$ have sedentarity $(0)$.

\begin{cor} \label{cor:intersection}
Let $\Delta, \Delta'$ be polyhedra in $N_\Sigma$ which are constant towards the boundary and of sedentarity $\sigma\in \Sigma$ and $\sigma'\in \Sigma$, respectively. 
Then we have:
\begin{enumerate}
\item \label{item:disjoint} 
If there is no $\tau \in \Sigma$ with  $\sigma \prec \tau $ and $\sigma'\prec \tau $, then $\Delta \cap \Delta'=\emptyset$.
\item \label{item:9} 
Assume that there is  $\tau \in \Sigma$ with  $\sigma \prec \tau $ and $\sigma'\prec \tau $. 
Let $\rho$ be the minimal such $\tau$. 
Then $\Delta \cap \Delta '=\overline{\Delta (\rho)\cap \Delta'(\rho )}$. 
In particular, $\Delta \cap \Delta'$ is a polyhedron of sedentarity $\rho $ which is constant towards the boundary. 
\end{enumerate}
\end{cor}

\begin{proof}

	To prove \ref{item:disjoint}, we note that if the cone $\tau $ 
	does not exists, then $\overline{N(\sigma )}\cap
	\overline{N(\sigma ')}=\emptyset $, hence $\Delta \cap \Delta
	'=\emptyset$. 
	
	To prove  \ref{item:9}, 
we  assume that such a cone $\tau$ exists and we use $\rho$ for the minimal one.  In this case
	$\Delta \cap \Delta '\subset \overline{N(\rho)}$. 
Note that $\Delta \cap \Delta '\supset \overline{\Delta (\rho)\cap \Delta
	'(\rho )}$ is obvious, so we have to prove $\Delta \cap \Delta '\subset\overline{\Delta (\rho)\cap \Delta
	'(\rho )}$. We may assume $\Delta \cap \Delta'\neq \emptyset$.
Pick any $\nu \in \Sigma$ with $\Delta \cap \Delta' \cap N(\nu) \neq \emptyset$. Let $x\in \Delta \cap \Delta '\cap N(\nu)$.
	Since $\Delta $ and $\Delta'$ are constant towards the boundary, there is a neighborhood
	$\Omega $ of $x$ in $N_\Sigma$ such that
	\begin{displaymath}
		\pi_{\nu,\rho}^{-1}({\Delta}(\nu))\cap \Omega
		=\Delta(\rho)\cap \Omega, \quad
		\pi_{\nu,\rho}^{-1}({\Delta'}(\nu ))\cap \Omega
		=\Delta'(\rho)\cap \Omega.
	\end{displaymath}
	Since $x$ belongs to the closure of $\pi_{\nu,\rho}^{-1}(x)$, 
	we
	conclude that $x\in \overline{\Delta (\rho)\cap \Delta' (\rho)}(\tau )$.
	This shows that $\Delta \cap \Delta '=\overline{\Delta (\rho)\cap
		\Delta '(\rho)}$ and proves \ref{item:9}.
\end{proof}

\begin{ex}\label{counterexamples}
We next give a series of examples showing that the condition of being constant towards the boundary is necessary in the previous results.
\begin{enumerate}
\item \label{item:11} 
Consider the fan given by the single cone $\tau$ in $\R^{3}$ generated by the vectors $(1,1,1)$ and $(1,1,-1)$ and its faces. 
Let $\Delta = \R_{\geq 0}^3 $ 
and $F=\{(0,0,z)\mid z\ge 0\}$ which is a face of $\Delta$. 
Using the coordinates $x,y,z$ on $\R^3$, we identify $N(\tau)$ with $\R$  where the projection $\pi _{\tau }$ is given by $\pi_{\tau }(x,y,z)=y-x$. 
In this case, Lemma \ref{lemm:1} shows that
\begin{displaymath}
\overline{\Delta} (\tau )=N(\tau ),\quad\text{and}\quad \overline{F}(\tau )=\{0\}.
\end{displaymath}
Note that $\{0\}$ is not a face of $\R$. 
Thus the map in Corollary \ref{cor:4} is not well defined if the polyhedron is not constant towards the boundary. 
This also shows that a face of a face is not necessarily a face if the polyhedron is not constant towards the boundary, because $\{0\}\subset N(\tau)$ is a face of $\overline{F}$ but it is not a face of $\overline{\Delta }$.  
\item \label{item ii} 
Now let $x,y$ be the coordinates of $\R^2$. 
Consider the fan given by the first quadrant of $\R^{2}$, the cones  $\sigma =\{0\} $, $\tau =\{x=0\}$ and $\rho $ the  maximal cone and the polyhedron $\Delta$ given by the closure of $\{(x,y)\in \R^2\mid x=y\ge 0\}$. 
Then
\begin{displaymath}
\Delta (\rho )=N(\rho )\not =\emptyset,\quad\text{but}\quad\Delta (\tau )=\emptyset.
\end{displaymath}
    This shows that the condition of being constant towards the
    boundary is needed in Corollary \ref{cor:transitivity}.
  \item Consider the same fan as in \ref{item ii} and the polyhedra $\Delta
    =\{(\lambda ,2\lambda )\mid \lambda \ge 0\}$ and $\Delta'
    =\{(2\lambda ,\lambda )\mid \lambda \ge 0\}$. Then
    \begin{displaymath}
      \overline{\Delta} \cap \overline{\Delta '}=\{(0,0)\}\cup N(\rho ).
    \end{displaymath}
    showing that $\Delta \cap \Delta '$ is not a polyhedron and
    does not agree with $\overline{\Delta \cap \Delta '}$. 
  \end{enumerate}
\end{ex}

\subsection{Simplicial decompositions}
\label{sec:part-result-subd}

In this section we study the question of whether a locally finite family of polyhedra in a tropical toric variety can be decomposed into a polyhedral complex. 
We will give a positive answer when the polyhedra are constant towards the boundary and the fan is simplicial. 

We start by defining polyhedral complexes in tropical toric varieties. 

\begin{definition} \label{definition: polyhedral complex} 
Let $U$ be an open subset of the tropical toric variety $N_{\Sigma }$. 
A \emph{polyhedral complex $\Ccal$ in $U$} 
is a locally finite set of polyhedra contained in $U$ such that 
\begin{enumerate}
\item
if $\Delta\in\Ccal$, then all faces of $\Delta$ belong to $\Ccal$, and
\item
given $\Delta_1,\Delta_2\in\Ccal$ the intersection
$\Delta_1\cap\Delta_2$ is either empty or a common face of $\Delta_1$
and $\Delta_2$. 
\end{enumerate}
A polyhedral complex $\Ccal$ in 
$U$ is called \emph{constant
  towards the boundary if all polyhedra in $\Ccal$ are constant
  towards the boundary.}
\end{definition}

We will see in Proposition \ref{prop:1} that we can construct polyhedral complexes in $N_{\Sigma}$ from polyhedral complexes in $N_{\R}$ that are constant towards the boundary and locally finite in $N_{\Sigma }$.

Recall from the Minkowski--Weyl Theorem \cite[Theorem
19.1]{rockafellar1970} that every polyhedron $T$ in $N_\R$ can be
written as 
\begin{displaymath}
T=\Conv(p_{0},\dots,p_{r})+\Cone(v_{1},\dots,v_{s})
\end{displaymath}
where $\Conv(p_{0},\dots,p_{r})$ is the convex hull  {of the points}
 $p_{0},\dots,p_{r}\in N_\R$ and where
$\Cone(v_{1},\dots,v_{s})$ is the cone generated by 
{the vectors} $v_{1},\dots,v_{s} \in N_\R$. 
The recession cone of $TY$ is then given as $\rec(T)=\Cone(v_{1},\dots,v_{s})$.

\begin{definition} Let $\Sigma $ be a fan and $\sigma \in \Sigma$. 
A polyhedron $T\subset N(\sigma )$ is called \emph{simplicial} if it can be written as 
\begin{displaymath}
T=\Conv(p_{0},\dots,p_{r})+\Cone(v_{1},\dots,v_{s})
\end{displaymath}
with $r\ge 0$ and $\dim T=r+s$ for some $p_0,\dots,p_r,v_1,\dots,v_s \in N(\sigma)$. 
Note that this last condition is equivalent to the condition that the vectors $v_j\in N(\sigma)$, $j=1,\dots,s$ and $p_i-p_0$, $i=1,\dots, r$ are linearly independent.
In particular, the points $p_{i}\in N(\sigma )$ are in general position. 
Additionally, we make the convention that the empty set $\emptyset$ is simplicial.
A polyhedral complex $\Pi$ in $N_\R$ is called \emph{simplicial} if
all polyhedra in $\Pi $ are simplicial.  
This definition applies in particular to fans. 
\end{definition}

Clearly, if a polyhedron $T$ is simplicial, then each face of $T$ is simplicial. 
With some extra conditions, the same is true for the faces of the closure $\overline T$ in a compactification $N_{\Sigma}$.

\begin{lemma}\label{lemm:4} 
Let $\Sigma $ be a simplicial fan and $T$ a simplicial polyhedron of
$N_\R$ such that $\rec(T)\in \Sigma $.  
Let $\overline T$ be the closure of $T$ in $N_{\Sigma }$. Then, for
any $\sigma \in \Sigma $, the polyhedron $\overline T(\sigma )$ is
simplicial. 
\end{lemma}

\begin{proof}
Let $\sigma\in \Sigma$ and write
$T=\Conv(p_{0},\dots,p_{r})+\Cone(v_{1},\dots,v_{s})$ with
$r\ge 0$ and $\dim T=r+s$. 
Assume that the intersection $\overline T\cap N(\sigma )$ is nonempty. 
By Lemma \ref{lemm:1}(i), we have $\relint(\sigma )\cap \rec(T)\not =\emptyset$. 
Since $\rec(T)\in \Sigma $, we deduce that $\sigma $ is a face of
$\rec(T)=\Cone(v_{1},\dots,v_{s})$.  
Hence there is a subset $I\subset \{1,\dots,s\}$ such that $\sigma =\Cone(v_{i},i\in I)$. 
Write $I^{C}=\{1,\dots,s\}\setminus I$ and let $\pi_\sigma \colon
N_{\R}\to N(\sigma )$ be the canonical projection.   
Then Lemma \ref{lemm:1}(ii) yields the first equality in
\begin{displaymath}
\overline T(\sigma) = \pi_\sigma(T)= \Conv(\pi_\sigma
(p_{0}),\dots,\pi_\sigma (p_{r}))+\Cone(\pi_\sigma
(v_{j}),j\in I^{C}),
\end{displaymath}
and the last term is simplicial as the vectors $v_1,\dots, v_s$,
$p_1-p_0, \dots, p_r-p_0$ are linearly independent and $(v_i)_{i \in
  I}$ generates the kernel of $\pi_\sigma$. 
\end{proof}
\begin{lemma} \label{lemm:3} 
Let $\Sigma $ be a simplicial fan on $N_{\R}$.
Let $\Pi $ be a polyhedral complex on $N_{\R}$ such that $\rec(T)\in \Sigma$ for all $T\in \Pi $ and such that $\Pi$ is locally finite in $N_\Sigma$.  
Then there exists a simplicial polyhedral refinement $\Pi' $ of $\Pi $ such that 
\begin{enumerate}
\item
$\Pi'$ is locally finite in $N_\Sigma$,
\item 
each polyhedron of $\Pi $ that is already simplicial belongs to $\Pi'$, 
\item 
for every $T'\in \Pi '$, the condition $\rec(T')\in \Sigma $ still holds. 
\end{enumerate}
\end{lemma}

\begin{proof}
The proof goes by induction on the dimension of $\Pi $. 
If the dimension is zero, then there is nothing to prove. 
If the dimension $k$ is non-zero, then by the induction hypothesis, we
may assume that there is a simplicial subdivision $\Pi '_{k-1}$ of the
$(k-1)$-skeleton $\Pi_{k-1}$ of $\Pi  $ such that all the simplicial
polyhedra of $\Pi $ of dimension less than $k$ are in $\Pi '_{k-1}$
and such that $\rec(T')\in \Sigma $ for every polyhedron $T'$ of
$\Pi'_{k-1}$.  
Let $T$ be a $k$-dimensional polyhedron of $\Pi $.  Recall that the
\emph{canonical decomposition} $\{T\}$ of the polyhedron $T$ is its
decomposition into faces.  

If $T$ is simplicial, then all of its faces are simplicial. 
Hence none of the faces of $T$ are subdivided in $\Pi'_{k-1}$ and the canonical decomposition $\{T\}$ of $T$ is \emph{compatible} with $\Pi '_{k-1}$ in the sense that $\Pi'_{k-1}\cup \{T\}$ is again a polyhedral complex.

If $T$ is not simplicial, then let
\begin{displaymath}
T=\Conv(p_{0},\dots,p_{r})+\Cone(v_{1},\dots,v_{s})
\end{displaymath}
be a  presentation of $T$. 
Note that the cone $\rec(T)=\Cone(v_{1},\dots,v_{s})\in \Sigma$ is
simplicial by assumption.  
We conclude that $T$ contains no line and hence we may assume that
$p_0,\dots, p_r$ are the extreme points of $T$ and
$v_1,\dots,v_s$ are the extreme directions of $T$, see \cite[Theorem
18.5]{rockafellar1970}.  
Such a presentation is called \emph{minimal}. 
Since $T$ is not simplicial, we must have $r \geq 1$. 
Note also that the relative interior of the polyhedron $T$ is the sum
of the relative {interiors of the polytope
  $\Conv(p_{0},\dots,p_{r})$ and of the cone
  $\Cone(v_{1},\dots,v_{s})$}.
Consider the cone
\begin{displaymath}
\widetilde T =
\Cone((p_{0},1),\dots,(p_{r},1),(v_{1},0),\dots,(v_{s},0))\subset N_{\R}\times \R_{\ge 0}.
\end{displaymath}
This cone satisfies
\begin{displaymath}
\widetilde T\cap \left(N_{\R}\times \{1\}\right)=T\times\{1\}.
\end{displaymath}
The subdivision $\Pi '_{k-1}$ induces a simplicial subdivision of the
boundary of $T$ that we denote $\Pi _{\partial T}$.  
This subdivision $\Pi _{\partial T}$
induces a conical simplicial subdivision $\Pi _{\partial \widetilde T}$ of the boundary of $\widetilde T$, whose restriction to $\widetilde T\cap \left(N_{\R}\times\{0\}\right)=\rec(T)\times \{0\}$ agrees with the canonical subdivision $\{\rec(T)\}$ of $\rec(T)$. 

Consider the vector
\begin{displaymath}
v= \frac{1}{r+1}\sum_{i=0}^{r}(p_{i},1)+\sum_{j=1}^{s}(v_{j},0)\in N_{\R}\times \R_{\ge 0}
\end{displaymath}
and the conical subdivision $\Pi _{\widetilde T}$ of $\widetilde T$ containing all cones of $\Pi _{\partial \widetilde T}$ and all cones of the form $\Cone(\Delta,v)$ for $\Delta \in \Pi _{\partial \widetilde T}$. 
{Since the coefficients in the definition of $v$ are all strictly
positive, $v$ belongs to the relative interior of $\widetilde T$.}
Since all cones of $\Pi_{\partial \widetilde T}$ are simplicial, 
we deduce that all cones of $\Pi _{\widetilde T}$ 
are simplicial. 
By intersecting with $N_{\R}\times \{1\}$, we obtain a simplicial decomposition $\Pi _{T}$ of $T$ that does not change any face of $T$ that was already simplicial and that is compatible with the decomposition $\Pi _{\partial T}$ of the boundary of $T$.  
Moreover, since the induced decomposition on $\widetilde T\cap
\left(N_{\R}\times \{0\}\right)$ is still the canonical decomposition
of $\rec(T)$, {we deduce that} for all $\Delta \in \Pi _{T}$ that
$\rec(\Delta )$ is a face of $\rec(T)$, so it belongs to $\Sigma $. 

Since the decomposition of $T$ is compatible with the previous
decomposition of the boundary of $T$, the decompositions for the
different $k$-dimensional polyhedra of $\Pi $ are compatible and glue
together to give a subdivision of $\Pi $ with the desired properties. 
\end{proof}

\begin{definition} 
If $\Sigma $ is simplicial, then we call a polyhedron $\Delta $ in  $N_{\Sigma }$ of sedentarity $\sigma $ \emph{simplicial} if $\Delta (\sigma )$ is simplicial and
$\rec(\Delta (\sigma ))\in \Sigma (\sigma )$. A \emph{simplicial complex} is a polyhedral complex consisting of simplicial polyhedra. 
\end{definition}

Note that, thanks to Lemma \ref{lemm:4}, the two conditions $\Delta
(\sigma )$ simplicial and $\rec(\Delta (\sigma ))\in \Sigma (\sigma )$
imply that, for every $\tau \in \Sigma $, $\Delta (\tau )$ is  simplicial.

\begin{definition} 
If $\Sigma $ is simplicial, then we call a polyhedron $\Delta $ in  $N_{\Sigma }$ of sedentarity $\sigma$ \emph{simplicial} if $\Delta (\sigma )$ is simplicial and $\rec(\Delta(\sigma))\in\Sigma (\sigma )$. 
A \emph{simplicial complex} is a polyhedral complex consisting of simplicial polyhedra. 
\end{definition}

Note that, thanks to Lemma \ref{lemm:4},  the conditions  $\rec(\Delta (\sigma ))\in \Sigma (\sigma )$ and $\Delta(\sigma )$ simplicial imply that, for every $\tau \in \Sigma $, $\Delta (\tau )$ is either
empty or simplicial.

We next see how to construct polyhedral complexes in $N_\Sigma$.

\begin{prop}\label{prop:1} 
Let $\Sigma $ be a fan and $\Pi$ a complete polyhedral complex  on
$N_\R$ such that $\overline T$ is constant towards the
boundary for each $T \in \Pi$ and such that $\Pi$ is locally
finite in $N_\Sigma$.  
Let $\Ccal$ be the set of all faces of all closures in $N_{\Sigma }$ of polyhedra of $\Pi $. 
Then $\Ccal$ is a polyhedral complex in $N_{\Sigma} $.  
If moreover $\Sigma $ and $\Pi$ are simplicial and for each $T\in \Pi $, $\rec(T)\in \Sigma $, then $\Ccal$ is simplicial. 
\end{prop}

\begin{proof}
By Corollary \ref{cor:3}, $\Ccal$ satisfies the first condition of Definition \ref{definition: polyhedral complex}. 
So we only need to check the second condition.  
Let $\Delta _{1}$ and $\Delta _{2}$ be elements of $\Ccal$ of sedentarity $\sigma _{1}$ and $\sigma _{2}$ respectively. 
By Corollary \ref{cor:4} and Remark \ref{equivalent condition for faces}, there are unique polyhedra $T _{1}$ and $T _{2}$ in $\Pi $ with $\sigma_{i}\subset \rec(T_{i})$, such that
\begin{displaymath}
\Delta _{i}(\sigma _{i})=\overline{T _{i}}(\sigma _{i}),\text{ so }
\Delta _{i}=\overline{\overline{T _{i}}(\sigma _{i})} \qquad i=1,2.
\end{displaymath}
Let $\tau $ be the minimal cone such that $\sigma _{1}\prec \tau $ and $\sigma _{2}\prec \tau $. 
If such a cone does not exist, then $\Delta _{1}\cap \Delta _{2}=\emptyset$ by Corollary \ref{cor:intersection}~\ref{item:disjoint}. 
Assume that such a cone exists.
By Corollary \ref{cor:intersection}~\ref{item:9} and Corollary
    \ref{cor:transitivity}
\begin{displaymath}
      \Delta _{1}\cap \Delta _{2}=\overline{\Delta _{1}(\tau )\cap \Delta _{2}(\tau )}=
      \overline{\overline{\overline{T _{1}}(\sigma _{1})}(\tau
        )\cap
        \overline{\overline{T _{2}}(\sigma _{2})}(\tau )}=
      \overline{\overline{T _{1}}(\tau
        )}\cap
      \overline{\overline{T _{2}}(\tau )}.
\end{displaymath}
By Corollary \ref{cor:3}, $\Ccal$ satisfies the first condition of  Definition \ref{definition: polyhedral complex}. 
So we only need to check the second condition.  
Let $\Delta _{1}$ and $\Delta _{2}$ be elements of $\Ccal$ of sedentarity $\sigma _{1}$ and $\sigma _{2}$ respectively. 
By Corollary \ref{cor:4}, there are unique polyhedra $T _{1}$ and $T _{2}$ in $\Pi $ with $\sigma_{i}\subset \rec(T_{i})$, such that
\begin{displaymath}
\Delta _{i}(\sigma _{i})=\overline{T _{i}}(\sigma _{i}),\text{ so }\Delta _{i}=\overline{\overline{T _{i}}(\sigma _{i})} \qquad i=1,2.
\end{displaymath}
Let $\tau $ be the minimal cone such that $\sigma _{1}\prec \tau $ and $\sigma _{2}\prec \tau $. 
If such a cone does not exist, then $\Delta _{1}\cap \Delta _{2}=\emptyset$ by Corollary \ref{cor:intersection}~\ref{item:disjoint}. 
Assume that such a cone exists. 
By Corollary \ref{cor:intersection}~\ref{item:9} and Corollary
 \ref{cor:transitivity},
\begin{displaymath}
\Delta _{1}\cap \Delta _{2}
=\overline{\Delta _{1}(\tau )\cap \Delta _{2}(\tau )}
=\overline{\overline{\overline{T _{1}}(\sigma _{1})}(\tau)\cap\overline{\overline{T _{2}}(\sigma _{2})}(\tau )}
=\overline{\overline{T _{1}}(\tau)}\cap\overline{\overline{T _{2}}(\tau )}.
\end{displaymath}
is a polyhedron of sedentarity $\tau $.
Applying again Corollary \ref{cor:intersection}~\ref{item:9} to the polyhedra $\overline{T_{i}}$, $i=1,2$ that are both of sedentarity $0$ and to the polyhedra $\overline{\overline{T_{i}}(\tau )}$, of sedentarity $\tau $, we obtain,  
\begin{displaymath}\label{eq:8}
\overline{\overline{T_{1}}(\tau)}\cap\overline{\overline{T_{2}}(\tau )}
=\overline{\overline{T_{1}}(\tau )\cap \overline{T_{2}}(\tau )}
=\overline{\overline{T_{1}\cap T_{2}}(\tau )}.
\end{displaymath}
Since $T_{1}\cap T_{2}$ is a face of $T_{i}$, by Corollary
\ref{cor:4}, $\Delta _{1}\cap \Delta _{2}$ is a face of
$\overline{\overline{T_{1}\cap T_{2}}(\tau )}$ and hence of
$\overline{T_{i}(\tau )}$, $i=1,2$, by transitivity of faces given in
Corollary \ref{cor:3}. Since $\overline{T _{i}(\tau )}$ is a face of
$\Delta _{i}=\overline{T _{i}(\sigma _{i})}$, $i=1,2$, Corollary
\ref{cor:3} implies that $\Delta _{1}\cap \Delta _{2}$ is a common
face of $\Delta _{1}$ and of $\Delta _{2}$.

    If $\Sigma $ and $\Pi $ are simplicial, and each $T\in \Pi $
    satisfies that $\rec(T)\in \Sigma $, Lemma
    \ref{lemm:4} implies that $\Ccal$ is also a simplicial complex. 
  \end{proof}

\begin{thm} \label{thm:1}
Let $\Sigma $ be a simplicial fan,  $U\subset N_{\Sigma }$ an open subset, and $\Tcal=(\Lambda _j)_{j \in J}$
a locally finite family of polyhedra in $U$ that are constant towards the boundary.
Then there exists a simplicial  complex $\Ccal$ on $U$, such that $\rec(\Delta(\sigma))\in\Sigma (\sigma )$ for all $\Delta \in \Ccal$ and all $\sigma \in\Sigma $, and such that each polyhedron $\Lambda \in  \Tcal$ is a locally finite union of polyhedra in $\Ccal$. 
\end{thm}

\begin{proof}
By a result of Coles and Friedenberg \cite[Theorem 1.1]{ColesFriedenberg23:localfincomp}, there exists a
complete polyhedral complex $\Pi$ of $N_\R$ such that $\rec(T)
\in \Sigma$ for all $T \in \Pi$ and such that $\Pi$ is locally
finite in $N_\Sigma$.  
Using Lemma \ref{lemm:3}, we may assume that $\Pi$ is simplicial and
that $\rec(T)\in \Sigma$ for all $T \in \Pi $.  
The latter implies that every polyhedron $T\in \Pi$ is pointed,
i.e. it admits a face that is a vertex.
By Proposition \ref{prop:1}, the set $\Ccal_{0}$  of all
  closures of polyhedra of $\Pi $ on $N_{\Sigma }$ is a simplicial polyhedral
complex in $N_{\Sigma }$.

In a first reduction step, we show that we can assume that all polyhedra $\Lambda \in \Tcal$ 
satisfy the condition 
\begin{equation}\label{eq:2}
\rec(\Lambda (\sigma ))\in \Sigma (\sigma),\ \forall \sigma \in \Sigma. 
\end{equation}
Note that condition \eqref{eq:2} is a bit stronger than being constant towards the boundary, see Proposition \ref{recession-cone-boundary-condition}. 
Moreover, \eqref{eq:2} implies that the polyhedron $\Lambda (\sigma)$
is also \emph{pointed}.

We now consider all the intersections between
polyhedra $\Lambda \in \Tcal$  and polyhedra $\Delta \in \Ccal_{0}$.  
These intersections define a new family of polyhedra $\Tcal'$ in $U$.
The family $\Tcal'$ is locally finite in $U$ and such that each
$\Lambda \in \Tcal$ is a locally finite union of elements of  $\Tcal'$.  
Moreover,  we claim that each polyhedron $\Lambda ' \in \Tcal'$
satisfies condition \eqref{eq:2}.  
Indeed, we have $\Lambda '=\Lambda  \cap \Delta $ for
some $\Lambda \in \Tcal$ and $\Delta \in \Ccal$. 
Assume that $\Lambda ' (\sigma)\neq \emptyset$ and let $\rho $ be
  the sedentarity of $\Lambda '$.  Then Lemma \ref{lemm:1} and Corollary \ref{cor:intersection}~\ref{item:9} yield
\begin{equation}\label{recession-formula}
  \rec(\Lambda ' (\sigma))= \pi_{\sigma,\rho}(\rec(\Lambda '(\rho)))
  = \pi_{\sigma,\rho}(\rec(\Delta(\rho ))\cap \rec(\Lambda (\rho ))).
\end{equation}
Since $\Lambda $ is constant towards the boundary, Corollary
\ref{polyhedraonconstant towards the boundary and orbit} and Proposition
\ref{recession-cone-boundary-condition} yield that $\rec(\Lambda (\rho
))\cap |\Sigma(\rho )|$ is a union of cones in $\Sigma(\rho )$.  
Using that we have $\rec(\Delta(\rho )) \in \Sigma(\rho )$ by
assumption, we conclude that $\rec(\Delta(\rho ))\cap \rec(\Lambda  (\rho ))\in
\Sigma(\rho )$.  
Hence $\rec(\Lambda '(\sigma))\in \Sigma(\sigma)$.
This proves the first reduction step.

Note that, after the first reduction step, we assume that every
polyhedron in $\Tcal$ satisfies the condition \eqref{eq:2}. By Lemma
\ref{lemm:5} this also implies that each polyhedron $\Lambda \in \Tcal$ is compact.  

By Proposition \ref{prop:1}, any 
  complete polyhedral complex $\Pi '$ in $N_{\R}$ that is simplicial and
  satisfies $\rec(T)\in \Sigma $ for all $T \in \Pi '$ can
  be canonically extended to a simplicial polyhedral complex in
  $N_{\Sigma }$. Thus it is natural to restrict our attention to
  $N_{\R}$ and extend the polyhedral complex to $N_{\Sigma }$ at a
  latter stage. With this in mind, the next reduction step is to show that we can assume that all
$\Lambda \in \Tcal$ are polyhedra of sedentarity $0$.  
Indeed, for each $\Lambda \in \Tcal $ whose sedentarity is not zero,
we claim that one can choose a polyhedron $\Lambda '\subset U$ of sedentarity
$0$ which is constant towards the boundary, such that $\Lambda $ is a face of
$\Lambda '$  and such that the family $\Tcal'$ of these polyhedra $\Lambda '$ is
still locally finite.  
The last claim can be established as follows.  
We number the cones of $\Sigma $ that are different from the cone
$\{0\}$ and we change the family $\Tcal$ one cone at each time to
obtain $\Tcal'$.   
Let $\sigma =\sigma_{i} \in \Sigma $ be a cone. 
We can assume that $\Tcal$ does not contain any polyhedra
of sedentarity $\sigma _{k}$ with $k<i$.  
Let $\Lambda \in \Tcal$ be a polyhedron of sedentarity $\sigma $.  
By \eqref{eq:2}, we have $\tau'=\rec(\Lambda (\sigma))\in \Sigma (\sigma )$.  
Let $\tau \in \Sigma $ be the cone that gives rise to $\tau '$. 
Then $\sigma \prec \tau $.  
Let $\Sigma _{\tau }$ be the fan consisting of $\tau $ and its faces. 
Then $N_{\Sigma_{\tau }}$ is an open subset of $N_{\Sigma }$. 
Moreover, since the fan is simplicial, there is a homeomorphism
\begin{displaymath}
N_{\Sigma_{\tau }}\simeq (\Rsup)^{r}\times \R^{d-r}
\end{displaymath}
which is an $\R$-linear isomorphism on each stratum such that we can identify 
\begin{displaymath}
\overline{N(\sigma )}=\{(x_{1},\dots,x_{d})\in N_{\Sigma_{\tau}}\mid x_{i}=\infty,\ i=1,\dots,s\}
\end{displaymath}
where $s\le r$ and $\Rsup=\R\cup \{\infty\}$. 
Since $\Lambda \subset \overline{N(\sigma )}$ is compact, there exists
$M>0$ such that the thickening of $\Lambda $  
\begin{displaymath}
\Th(\Lambda )=\left \{(x_{1},\dots,x_{d})\in N_{\Sigma_{\tau}}\middle |
\begin{aligned}
      &x_{i}\ge M,\ i=1,\dots,s\\
      &(\infty,\dots,\infty,x_{s+1},\dots,x_{d})\in \Lambda 
\end{aligned}
\right\}.
\end{displaymath}
is contained in $U$. Note that, since $\Lambda $ is constant towards the
boundary,  $\Th(\Lambda )$ is also constant towards the boundary. 
Even stronger, the polyhedron $\Th(\Lambda )$ satisfies still condition \eqref{eq:2} which is clear by the explicit description above of the thickening.

Since the family of polyhedra $\Lambda \in\Tcal$  of sedentarity $\sigma $ is
locally finite in $U$, the same is true for the family $\Tcal'$
consisting of the polyhedra $\Th(\Lambda )$.  
Moreover, the  family $\Tcal''$ consisting of the polyhedra of $\Tcal$
whose sedentarity is not $\sigma $ is also locally finite.  
Thus the family $\Tcal'\cup \Tcal''$ is locally finite, does not
contain any polyhedron  of sedentarity $\sigma _{k}$ for $k\le i$ and
each polyhedron of $\Tcal$ is a face of a polyhedron in this family.  
Note that if we can show the theorem for the familiy $\Tcal' \cup
\Tcal''$, then an elementary set-theoretic argument together with
Corollary \ref{cor:3} shows that also every face of $\Tcal' \cup
\Tcal''$ is a locally finite union of polyhedra in $\Ccal$ and hence
we may replace $\Tcal$ by $\Tcal' \cup \Tcal''$.
Since the number of cones of $\Sigma $ is finite, applying this
process to each cone,  we can assume that all the polyhedra of $\Tcal$
have sedentarity $0$.   
This finishes the second reduction step.

Following our strategy of first concentrate on $N_{\R}$, we want
  to construct a polyhedral complex $\Pi '$ in $N_{\R}$ such that, for
each $\Lambda \in \Tcal$, $\Lambda (0)$ is a finite union of polyhedra in $\Pi '$
and $\rec(T)\in \Sigma $ for all $T \in \Pi '$.
We will construct $\Pi '$ inductively.

The family  $\Tcal=(\Lambda _j)_{j\in J}$ is at most countable as it
is locally finite and $U$ is locally compact with a countable basis
\cite[Remark 3.1.2]{burgos-gubler-jell-kuennemann1}. 
Hence we may asssume that $J=\N_{\geq 1}$.
The case of a finite $J$ can easily be treated along the same lines as
the countable case.
We still denote by $\Pi$ the complete simplicial polyhedral complex of
$N_\R$ chosen above such that $\rec(T) \in \Sigma$ for all
$T \in \Pi$ and such that $\Pi$ is locally finite in $N_\Sigma$. 
We start the inductive construction of a sequence $(\Pi _{k})_{k\in
  \N}$ of polyhedral complexes of $N_\R$ with $\Pi _{0}\coloneqq \Pi$.  
Assume that we have constructed a simplicial refinement $\Pi _{k}$  of
$\Pi $, satisfying that $\rec(T)\in \Sigma $ for all $T \in
\Pi _{k}$ and  such that each $\Lambda _{i}(0)$ with  $i\le k$ is a
finite union of polyhedra in $\Pi _{k}$.  
We want to find a simplicial refinement $\Pi _{k+1}$ of $\Pi _{k}$
with $\rec(T)\in \Sigma$ for all $T \in \Pi_{k+1}$ such that
$\Lambda _{k+1}(0)$ is a finite union of polyhedra in
$\Pi _{k+1}$ and such that only polyhedra of $\Pi _{k}$ whose closures 
meet  $\Lambda _{k+1}$ are subdivided. 
Note that, by Corollary
  \ref{cor:intersection}, the closure of $T$ meets $\Lambda _{k+1}$ if and only if 
$T$ meets $\Lambda _{k+1}(0)$.

Let $\sigma \coloneqq \rec(\Lambda _{k+1}(0))$. 
Since $\Sigma $ is simplicial, we can find a piecewise linear function
$f_{\sigma }$ on $\Sigma $ that has the value zero on $\sigma $ and is
negative outside $\sigma$.
Let $T\in \Pi _{k}$ with $T \cap \Lambda  _{k+1}(0)\not =
  \emptyset$. Note that $T\cap \Lambda  _{k+1}(0)=T\cap \Lambda  _{k+1}$. 
We want to find a decomposition of $T$ in which
$T \cap \Lambda _{k+1}$ 
is a polyhedron and  each face of $T $ 
that does not meet with $\Lambda _{k+1}$ remains untouched.

Let  $\tau  \coloneqq \rec(T)\in \Sigma $. 
Since $\rec(T)$ and $\rec(\Lambda _{k+1}(0))$ belong to $\Sigma $ and $T\cap \Lambda  _{k+1}\not =\emptyset$, we have that 
\[
\rec(T\cap \Lambda _{k+1})=\rec(T) \cap \rec(\Lambda _{k+1}(0)) \in \Sigma.
\] 
Thus, the minimal  presentations of $T$ and of $T \cap \Lambda _{k+1}$ are 
\begin{align*}
T
&=\Conv(p_{0},\dots,p_{r})+\Cone(v_{1},\dots,v_{s})\\
T\cap  \Lambda  _{k+1}&= \Conv(q_{0},\dots,q_{r'})+\Cone(v_{1},\dots,v_{s'})
\end{align*}
with $s'\le s$. 
Note that $\tau = \rec(T)=\Cone(v_{1},\dots,v_{s}) \in \Sigma$
and that $T$ is simplicial, hence $p_1-p_0, \dots,
p_r-p_0,v_1,\dots,v_s$ are $\R$-linearly independent.

Let $\varphi\colon T \to \R$ be the minimal concave function such that
\begin{displaymath}
\varphi(q_{i})\ge 0 \,\,(i=1,\dots,r'), \qquad \varphi(p_{j})\ge -1
\ (j=0,\dots,r)
  \end{displaymath}
and whose slope at each vector $v_{i}$ is bigger or equal than the
slope of $f_{\sigma }|_{\tau }$ at the same vector. 

First we check that $T \cap \Lambda _{k+1} = \{x\in T \mid \varphi(x)=0\}$. 
Indeed, the constant function $0$ satisfies the above constraints, so $\varphi\le 0$. 
This implies in particular that $\varphi(q_{i})=0$ for $i=0,\dots r'$
and that the slope of $\varphi$ at the vectors $v_{j}$ for
$j=1,\dots,s'$ is also $0$. 
Therefore for all $x\in T \cap \Lambda _{k+1}$, we deduce that
$\varphi(x)=0$. Let now $x\in
T \setminus \Lambda _{k+1}$. 
Then there is an affine function $g$ such that $g(x)<0$ and {$g(y)\ge
  0$}  for $y\in T\cap \Lambda _{k+1}$.  
Then there is a number $\varepsilon >0$ such that the function
$\min(0,\varepsilon g)$  
satisfies the previous constraints. 
Hence $\varphi(x)\le \varepsilon g(x)<0$.

Since $\varphi$ is concave and piecewise affine, it induces a polyhedral decomposition of $T$ as follows.
Let $H_{\varphi}\subset N_{\R}\times \R$ be the hypograph of $\varphi$:
  \begin{displaymath}
    H_{\varphi}=\{(x,t)\in N_{\R}\times \R\mid x\in T,\ t\le \varphi(x)\}.
  \end{displaymath}
This is a polyhedron in $N_{\R}\times \R$.
Denote by $\pi \colon H_{\varphi}\to N_{\R}$ the projection. 
Then the subdivision of $T$ associated with $\varphi$ consist of the
sets $\pi (F)$ for $F$ a proper face of $H_{\varphi}$.

We next check that  every polyhedron of the decomposition of
  $T$ associated with $\varphi$ has recession cone in $\Sigma
  $. 
By construction of the function $\varphi$, we have that
\begin{displaymath}
  H_{\varphi}=K+H_{f_{\sigma |_{\tau }}},
\end{displaymath}
where
\begin{displaymath}
  K=\Conv((q_{i},0),\  i=0,\dots,r', (p_{j},-1),\
  j=0,\dots,r)
\end{displaymath}
and $H_{f_{\sigma }|_{\tau }}$ is the hypograph of $f_{\sigma }|_{\tau
}$. 
Since the function $f_{\sigma }|_{\tau }$ is a linear function on $\tau$, we deduce that
\begin{displaymath}
  H_{f_{\sigma }|_{\tau }}=\Cone((0,-1),(v_{i},f_{\sigma }(v_{i})),i=1,\dots,s)
\end{displaymath}
is a simplicial cone and $\rec(H_{\varphi})=H_{f_{\sigma }|_{\tau}}$. 
Therefore, if $F$ is a proper face of $H_{\varphi}$, then $\rec(F)$ is a face of $H_{f_{\sigma }|_{\tau }}$. 
By the explicit description of $H_{f_{\sigma }|_{\tau }}$, we see that
$\rec(\pi(F))=\pi (\rec(F))$ is a face of $\tau $, hence belongs to
$\Sigma $.

Since the function $f_\sigma$ depends only on $\Lambda _{k+1}$ and is the same for all $T$, this construction is compatible with taking faces in the following sense. 
If $T'$ is a face of $T$, then the
decomposition of $T'$ obtained by the previous method is just
the restriction of the decomposition of $T$ to this face because
the function $\varphi'\colon T'\to \R$ obtained as  before is
the restriction of $\varphi$ to $T'$.

Moreover, if a face $T'$ of $T$ does not meet $\Lambda _{k+1}$, then
the function $\varphi$ restricted to $T'$ is the minimal concave
function that has value $\ge -1$ at the vertices of $T'$  and slope
bigger or equal than the slope of $f_{\sigma }$ at every edge of
$\rec(T')$.  
Since $T'$ is simplicial, there is an affine function that
satisfies these conditions, so $\varphi$ restricted to $T'$ is
affine and $T'$ is not subdivided (this is a
  point where we use that the polyhedral complex $\Pi _{k}$
is simplicial in an essential way).  

We have obtained a refinement $\Pi'_{k+1}$ of $\Pi _{k}$ with
$\rec(T')\in \Sigma$ for all $T'\in\Pi'_{k+1}$ that only
subdivides the polyhedra that meet
$\Lambda _{k+1}$ and such that $\Lambda _{k+1}(0)$ is a finite union of polyhedra of
$\Pi'_{k+1}$. 
Let $\Pi _{k+1}$ be a simplicial refinement of $\Pi '_{k+1}$ obtained
by applying Lemma \ref{lemm:3}. 

Once we have our sequence of polyhedral complexes $\{\Pi
  _{k}\}$ in $N_{\R}$, let $\{\Ccal_{k}\}$  be the induced sequence of
polyhedral complexes in $N_{\Sigma }$ obtained using Proposition \ref{prop:1}.
Finally we construct a simplicial polyhedral complex $\Ccal$ on
$U\subset N_\Sigma$ so that each polyhedron in $\Tcal$ is a  finite
union of polyhedra of $\Ccal$ and $\rec(\Delta(\sigma
))\in\Sigma (\sigma )$ for all $\Delta \in \Ccal$ and all $\sigma \in
\Sigma $.

Let $\Lambda _j$ be any polyhedron from $\Tcal$. Since $\Lambda _j$ is compact and
the family $\Tcal$ is locally finite, there is $k_0 \geq j$ such that
$\Lambda _k \cap \Lambda _j = \emptyset$ for all $k > k_0$.  
Let $k\ge k_{0}$ and let $\Ccal|_{T_{j}}$ be the polyhedral
  complex consisting of all the polyhedra of $\Ccal_{k}$ contained in
  $\Lambda _{j}$. 
Since the simplicial subdivisions $\Pi_k$ of $\Pi_{k_0}$ do not modify
the subdivision $\Pi_{k_0}|_{\Lambda _j}$ of $\Lambda _j$ and a polyhedron
  of $\Ccal_{k}$ meets $\Lambda _{j}$ if and only if it is a
  face of a polyhedron of $\Pi_{k}$ that meets $\Lambda _{j}$, we see that the
 polyhedral complex $\Ccal|_{\Lambda _{j}}$ in $\Lambda _{j}$ is
independent of the choice of $k$. 
Therefore, for any $i,j$, the  polyhedral complexes  $\Ccal|_{\Lambda _i}$
  and $\Ccal|_{\Lambda _j}$ agree on $\Lambda  _i \cap \Lambda _j$ as
  the involved polyhedra 
  belong to the  polyhedral complex $\Ccal_k$ for
  $k$ sufficiently large.

Therefore, there is a locally finite simplicial complex $\Ccal$ in $U$
with support $\bigcup_{i \in \N_{\geq 0}} \Lambda _{i}$ such that
every $\Lambda  \in \Tcal$ is a  finite union of polyhedra of $\Ccal$.
By construction, for every polyhedron $\Delta\in \Ccal$
of sedentarity $(0)$, the recession cone $\rec(\Delta(0))$ belongs to $\Sigma$. As
in the first reduction step, it follows that $\rec(\Delta(\sigma
))\in\Sigma (\sigma )$ for all $\Delta \in \Ccal$ and all $\sigma \in
\Sigma $.
\end{proof}

\begin{cor}\label{corollary-refinement-simplicial-decomposition}
Let $\Sigma $ be a simplicial fan. 
Let $U\subset N_{\Sigma }$ be an open subset. 
Then there is a locally finite simplicial decomposition of $U$ such
that, for each $\sigma \in \Sigma $,  the recession cones of the
polyhedra of sedentarity $\sigma$  belong to $\Sigma (\sigma )$. 
\end{cor}

\begin{proof}
By \cite[Theorem 1.1]{ColesFriedenberg23:localfincomp} and
Lemma \ref{lemm:3}  such a decomposition exists for the whole space $N_{\Sigma }$. 
It is clear that every $x \in U$ has a polyhedral neighbourhood in $U$
which is compact and constant towards the boundary. 
We use now that $U$ is $\sigma$-compact and locally compact \cite[Remark 3.1.2]{burgos-gubler-jell-kuennemann1}. 
We conclude that there is an increasing sequence $(V_k)_{k \in \N}$ of open subsets in $U$ such that $V_k$ is relatively compact in $V_{k+1}$ for every $k \in \N$. 
For convenience, we set $V_k \coloneqq \emptyset$ for $k \in \Z_{<0}$. 
We note that for any $x \in U$ there is a constant towards the
  boundary polyhedron $\Lambda _x$ in $N_\Sigma $ that is a compact neighbourhood of $x$ in $U$.
We pick $k \in \N$. 
For any point $x$ in the compact set $\overline{V_k} \setminus
V_{k-1}$, by shrinking, we may assume that the support of
$\Lambda _{x}$ is contained in the open set  
\[
V_{k+1} \setminus \overline{V_{k-2}}\supset \overline{V_k} \setminus V_{k-1}.
\]
Finitely many such neighbourhoods $\Lambda _{x}$ are enough to cover
$\overline{V_k} \setminus V_{k-1}$.   
The polyhedra from these finitely many $\Lambda _{x}$, for varying $k
\in \N$, form by 
construction a locally finite family $\Tcal$ of polyhedra which are
constant towards the boundary.  
Since the union of all $V_k$ is $U$, the union of all polyhedra in  $\Tcal$ is also $U$.
Applying Theorem \ref{thm:1} to this collection $\Tcal$, we obtain the corollary.
 \end{proof}

\section{Forms and currents}\label{forms-and-currents}
We recall the construction of Lagerberg forms \cite{lagerberg-2012} on tropical toric varieties \cite{jell-shaw-smacka2015,burgos-gubler-jell-kuennemann1}. 
We introduce and investigate the sheaf of polyhedral currents on a tropical toric variety.
We discuss piecewise smooth forms and the integration theory of smooth forms and polyhedral currents.
 
In Section \ref{forms-and-currents}, we fix the following notation: Let $N$ be a free abelian group of rank $n$, $M=\mathrm{Hom}_\Z(N,\Z)$ its dual and denote by $N_\R$  and $M_\R$ the scalar extensions to $\R$. 
We consider a fan $\Sigma$ with associated \emph{tropical toric variety $N_\Sigma$}, see \S \ref{subsec: notation}.

\subsection{Polyhedral currents} \label{subsection: polyhedral currents}
The aim of this section is to give the definition of polyhedral
currents. We start by recalling that 
Lagerberg \cite{lagerberg-2012} introduced an analogue of
$(p,q)$-forms for real vector spaces. These were extended by Jell, Shaw
and Smacka \cite{jell-shaw-smacka2015} to tropical toric varieties. We
follow here the presentation in \cite{burgos-gubler-jell-kuennemann1}
which serves as a general reference for all the used notions. 

On $N_\R$ we have the sheaves of Lagerberg forms $A^{p,q}$ and
Lagerberg currents $D^{p,q}$. 
Their construction does not require the integral structure on $N_\R$ given by $N$.
Furthermore these sheaves are invariant under affine transformations.
Hence the sheaves $A^{p,q}$ and $D^{p,q}$ are defined on all real
affine spaces of finite dimension.

The sheaves $A^{p,q}$ and $D^{p,q}$ have been extended to the partial
compactification $N_\Sigma$ \cite{jell-shaw-smacka2015}, 
\cite[Section 3]{burgos-gubler-jell-kuennemann1}.
There are differentials $d'$ and $d''$ and an integral
\[
\int_{U}\colon A^{n,n}_c(U)\longrightarrow \R
\]
for open subsets $U$ of $N_\Sigma$.

We next introduce forms on a polyhedron
$\Delta$ of $N_\Sigma$ of sedentarity $\sigma$. Let $\Omega$ be an
open subset of a polyhedron $\Delta$ in $N_\Sigma$ of sedentarity
$\sigma=\textrm{sed}(\Delta)$.
We choose an open subset $\widetilde\Omega$ of $N_\Sigma$ with $\Omega=\Delta\cap\widetilde\Omega$.
We denote by $C^\infty(\Omega)$ the space of all functions of the form
$f|_\Omega\colon\Omega\to \R$ for some $f\in
A^{0,0}(\widetilde\Omega)$ and 
we define $A^{p,q}(\Omega)$ as the image of the natural restriction map
\[
A^{p,q}(\widetilde\Omega)\longrightarrow
A^{p,q}\bigl(N(\sigma)\cap\widetilde\Omega\bigr)\longrightarrow
A^{p,q}\bigl(\Omega\cap\textrm{relint}(\Delta(\sigma))\bigr).
\]
A partition of unity argument shows that this definition does not depend on the choice of $\widetilde\Omega$.
Observe that partitions of unity exist by \cite[Proposition 3.2.12]{burgos-gubler-jell-kuennemann1}.

\begin{rem}\label{support-remark}
By construction we have $A^{p,q}(\Omega)=0$ if
$\dim\Delta<\max\{p,q\}$.  An element $\alpha\in A^{p,q}(\Omega)$ is
determined by a form
$\alpha_0\in
A^{p,q}\bigl(\Omega\cap\textrm{relint}(\Delta(\sigma))\bigr)$, but not
every form $\alpha _{0}$ as before extends to a form on $\Omega $.
The \emph{support} of $\alpha$ is by definition the closure in $\Omega
$ of the support of $\alpha_0$.  We denote by
$A_c^{p,q}(\Omega)$ the subspace of forms with compact support.  
\end{rem}

\begin{lem} \label{support lemma for top-forms}
Let $\Delta$  be a $d$-dimensional polyhedron of $N_\Sigma$ of sedentarity $\sigma \in \Sigma$ and let $\alpha\in A^{d,d}(\Omega)$ for an open subset $\Omega$ of $\Delta$. 
Then the support of $\alpha$ is contained in $\Omega \cap N(\sigma)$. 
The same holds for forms $\alpha$ of bidegree $(d-1,d)$ or $(d,d-1)$.
\end{lem}

\begin{proof}
As seen above, there is an open subset $\widetilde\Omega$ of $N_\Sigma$ with $\Omega = \Delta \cap \widetilde\Omega$ and $\widetilde \alpha \in A^{d,d}(\widetilde\Omega)$ with $\alpha =\widetilde\alpha|_\Omega$. 
Let $p \in \Omega \setminus N(\sigma)$. 
We have to show that $p$ has a neighbourhood $U$ in $\Omega$ such that $\alpha|_U=0$. 
There is a unique $\tau \in \Sigma$ with $p \in N(\tau)$. 
As $\Delta$ is of sedentarity $\sigma$, it follows that $\sigma$ is a face of $\tau$. 
	
Recall that a basis of compact polyhedral neigbhourhoods of $p$ is given as follows: 
We choose a complementary subspace $\L_\tau^\perp$ to $\L_\tau$ in $N_\R$. 
Consider a polytopal neighbourhood $P$ of $0$ and a point $q \in N_\R$ with $\pi_\tau(q)=p$ for the projection $\pi_\tau\colon N_\R \to N(\tau)$. 
Then the closure  of $q+P$ in $N_\Sigma$ is a polytope $Q$ in $N_\Sigma$ given as
\[
Q=\overline{q+P}= \coprod_{\rho \prec \tau} \pi_\rho(q+P+\tau).\]
These closures for varying $q$ and $P$ form a basis of neighbourhoods for $p$, see \cite[Remark 3.1.2]{burgos-gubler-jell-kuennemann1}.  
	
Let $\pi_{\tau\sigma}\colon N(\sigma) \to N(\tau)$ be the projection. 
It follows from Lemma \ref{lemm:1} that $\rec(\Delta) \cap \relint(\tau)\neq \emptyset$ and that $\Delta(\tau)=\pi_{\tau\sigma}(\Delta(\sigma))$ has dimension $<d$. 
Note that $\pi_\tau$ extends to a morphism 
\[\overline{\pi_\tau} \colon \coprod_{\rho \prec \tau} N(\rho) \longrightarrow N(\tau)
\]
given on $N(\rho)$ by the projection $\pi_{\tau\rho}$. 
Setting $Q(\rho) \coloneqq Q \cap N(\rho)= \pi_\rho(q+P+\tau)$ for any $\rho \prec \tau$, we have $Q \subset (\overline{\pi_\tau})^{-1}(Q(\tau))$. 
Since $\widetilde\alpha$ is constant towards the boundary, we can choose $q$ sufficiently close to $p$ and the polytopal neighbourhood $P$ of $0$ sufficiently small  such that $Q \subset \widetilde \Omega$ and 
\[
\widetilde\alpha_Q=(\overline{\pi_{\tau}})^*(\widetilde\alpha(\tau))|_Q.
\]
It is clear that $U \coloneqq Q \cap \Delta$ is a neighbourhood of $p$ in $\Omega$. 
It follows from the above display that
\[
\alpha|_U= \widetilde\alpha|_Q|_{Q \cap \Delta}=(\overline{\pi_{\tau}})^*(\widetilde\alpha(\tau))|_{Q \cap \Delta}.
\]
We claim that this is zero which will finish the proof. 
It is enough to show that the displayed form is zero on the finite part $Q(\sigma) \cap \Delta(\sigma)$. 
There this form agrees with 
\[
\pi_{\tau \sigma}^*(\widetilde\alpha(\tau))|_{Q(\sigma)\cap \Delta(\sigma)}=(\pi_{\tau \sigma}|_{Q(\sigma)\cap \Delta(\sigma)})^*(\widetilde\alpha(\tau)|_{Q(\tau)\cap\Delta(\tau)}).
\]
using $\pi_{\tau\sigma}(Q(\sigma)\cap \Delta(\sigma)) \subset Q(\tau) \cap \Delta(\tau)$ and the usual functorial identities for pull-back. 
Now the right-hand side of the display is zero for degree reasons as $\dim( Q(\tau) \cap \Delta(\tau)) \leq \dim(\Delta(\tau))<d$. 
This proves that $\supp(\alpha)$ is contained in $\Omega \cap N(\sigma)$.
\end{proof}

\begin{art} \label{definition weighted polyhedra}
A \emph{weight} of a polyhedron $\Delta$ is an element of $(\det
\L_\Delta\setminus\{0\})/\{\pm 1\}$.
If $\dim \Delta=0$, then $\det L_\Delta=\R$ canonically and $(\det
\L_\Delta\setminus\{0\})/\{\pm 1\}=(0,\infty)$. If $\dim
  \Delta>0$, then still $\det L_\Delta\simeq \R$ and $(\det
\L_\Delta\setminus\{0\})/\{\pm 1\}\simeq (0,\infty)$, but the
isomorphism is non-canonical in general.
A \emph{weighted polyhedron $[\Delta,\mu]$} is a polyhedron $\Delta$ together with a weight $\mu$.
A \emph{weighted polyhedral complex $(\Ccal,\mu)$} is given by a polyhedral complex
$\Ccal$ together with a family $(\mu_\Delta)_{\Delta\in \Ccal}$ where $\mu_\Delta$ is a weight for the polyhedron $\Delta$.
\end{art}

\begin{ex} \label{standard weight}
		Using that $N_\R$ comes with a lattice structure given by $N$, we have a canonical weight on $N_\R$ denoted by $\mu_{N}$. It is given by the image of a generator of $\det(N)$ in  $(\det(N_\R)\setminus\{0\})/\{\pm 1\}$.

\end{ex}

\begin{art} \label{current associated to weighted polyhedron} Let $[\Delta,\mu]$ be a weighted $d$-dimensional polyhedron of sedentarity $\sigma$ contained in the open subset $U$ of $N_\Sigma$. 
The goal is to define an associated current $\delta_{[\Delta,\mu]}$ on $U$.  
Let $\eta\in A^{d,d}_c(U)$. 
Since $\eta$  
is of bidegree $(d,d)$, we have seen in Lemma \ref{support lemma for top-forms}  that the support of $\eta|_{\Delta}$ is contained in $\Delta(\sigma)$. 
We conclude that $\eta|_{\Delta(\sigma)}$ is a compactly supported Lagerberg form of bidegree $(d,d)$ on $\Delta(\sigma)$ and hence
\[
\delta_{[\Delta,\mu]}(\eta)=\int_{[\Delta(\sigma),\mu]} \eta
\]
is well-defined by \cite[Example 2.2]{mihatsch2021}. 
One can check that this defines a current $\delta_{[\Delta,\mu]}\in D^{n-d,n-d}(U)$. 
More generally, for any form $\alpha \in A^{p,q}(\Delta)$, we get a current $\alpha \wedge  \delta_{[\Delta,\mu]}$ in $D^{n-d+p,n-d+q}(U)$ defined by 
\begin{displaymath}
\alpha \wedge
\delta_{[\Delta,\mu]}(\eta)=\int_{[\Delta(\sigma),\mu]} \alpha
\wedge \eta,
\quad \quad (\eta \in A_c^{d-p,d-q}(U)).
\end{displaymath}
Sometimes, we will simplify our notation by writing
$\alpha \wedge [\Delta,\mu]$ for the current
$\alpha \wedge \delta_{[\Delta,\mu]}$.
\end{art}

\begin{definition} \label{definition of polyhedral currents}
Let $U$ be an open subset of $N_\Sigma$. A \emph{polyhedral current} is a linear functional $T$ on $A_c(U)$ given by a locally finite family $([\Delta,\mu_\Delta])_{\Delta \in I}$ of weighted polyhedra in $U$ and smooth forms $\alpha_\Delta \in A(\Delta)$ for every  $\Delta\in I$  such that  
        \begin{equation}
          \label{eq:4}
          T = \sum_{\Delta \in I} \alpha_\Delta \wedge \delta_{[\Delta,\mu_\Delta]}.
        \end{equation}
It follows from \ref{current associated to weighted polyhedron} that $T$ is a current on $U$. We say that $T$ has tridegree $(p,q,r)$ if all $\alpha_\Delta$ have bidegree $(p,q)$ and if all $\Delta \in I$ have codimension $r$. 
We call $T$ \emph{of sedentarity} $\sigma \in \Sigma$ if all $\Delta \in I$ have sedentarity $\sigma$.
The polyhedral current $T$ is called \emph{constant towards the
  boundary} if we can choose a representation of $T$ as above with all
polyhedra $\Delta \in I$ constant towards the boundary.
The space of all polyhedral currents on $U$ is an $A^{\ast}(U)$
  module. If $T$ has a representation as in \eqref{eq:4} and $\beta $
  is a differential form, then
  \begin{displaymath}
    \beta \wedge T=\sum_{\Delta\in I }\beta|_{\Delta } \wedge
    \alpha_{\Delta } \wedge \delta _{[\Delta,\mu_\Delta]}. 
  \end{displaymath} 
\end{definition}

\begin{rem} \label{polyhedral complex of definition for polyhedral currents}
Let $T$ be a polyhedral current on an open subset $U$ of $N_\Sigma$
that is constant towards the boundary. 
If $\Sigma$ is simplicial, then $T$ is given by a weighted polyhedral complex
$(\Ccal,\mu)$ that is constant towards the boundary with $|\Ccal|=U$ such that
\begin{displaymath}
  T = \sum_{\Delta \in \Ccal} \alpha_\Delta \wedge \delta_{[\Delta,\mu_\Delta]}.
\end{displaymath}
This follows from  Theorem \ref{thm:1} and Corollary
\ref{corollary-refinement-simplicial-decomposition} by setting all
weights to $0$ for faces which do not occur in the definition above. 
\end{rem}

\begin{prop} \label{polyhedral currents form sheaf}
The polyhedral currents on $N_\Sigma$ form a sheaf.
\end{prop}
	
\begin{proof}	
Let $(U_i)_{i \in I}$ be an open covering of an open subset $U$ of $N_\Sigma$ and let $T_i=\sum_{\Delta \in I_i} \alpha_{i\Delta} \wedge \delta_{[\Delta,\mu_{i\Delta}]}$ be a polyhedral current on $U_i$. 
We have to prove that the {unique} current $T$ {on $U$} with $T|_{U_i}=T_i$ for all $i\in I$ is polyhedral. 
Since $N_\Sigma$ is metrizable {\cite[Remark 3.1.2]{burgos-gubler-jell-kuennemann1}}, it is clear that $U$ is paracompact. 
{By \cite[Proposition 3.2.12]{burgos-gubler-jell-kuennemann1}, we can} choose a smooth partition of unity $(\phi_i)_{i\in I}$ such that $\supp(\phi) \subset U_i$ for every $i \in I$. 
Then we have 
\[
T=\sum_{i \in I} \sum_{\Delta \in I_i} \phi_i\alpha_{i\Delta}\wedge \delta_{[\Delta,\mu_{i\Delta}]}
\]
proving the claim.
\end{proof}

\begin{rem} \label{subdivision}
Given a polyhedron $\Delta$ with sedentarity $\sigma$, a
\emph{subdivision} is induced by a subdivision of $\Delta(\sigma)$
inducing subdivisions of all polyhedra $\Delta(\tau)$ for all $\sigma
\prec \tau$. A subdivision of a polyhedral complex $\Ccal$ is given by
subdivisions of all polyhedra in $\Ccal$ which fit on faces. Clearly,
a polyhedral current is invariant  under passing to subdivisions of the
underlying polyhedra. 
\end{rem}

\begin{rem}[Weights in short exact sequences] \label{weighted short exact sequences}
For a short exact sequence
\[
0 \longrightarrow V_1 \longrightarrow V_2 \longrightarrow V_3 \longrightarrow 0
\]
of finite dimensional $\R$-vector spaces, we have a canonical isomorphism $\det V_2 = \det V_1 \otimes \det V_3$ and hence two of three weights $\mu_i$ for $V_i$ determine the third by 
\begin{equation} \label{weight relation}
\mu_2 = \mu_1 \otimes \mu_3
\end{equation}
where we usually write $\mu_1 \wedge \mu_3$ for the right hand side.
\end{rem}
	
We next recall Mihatsch's definition \cite[\S 2.3, after Definition 2.5]{mihatsch2021} of the pull-back of  currents in the affine setting.

\begin{rem}[{Pull-back of currents}] \label{pull-back of currents}
Let  $E\colon N' \to N$ be a \emph{surjective} affine linear map of abelian groups of finite rank. 
We set $L \coloneqq E-E(0)$, $k \coloneqq \dim(\ker(L))$ and we consider the short exact sequence
\begin{equation}  
0\longrightarrow \ker(L)_\R\longrightarrow N'_\R\longrightarrow N_\R \longrightarrow 0.
\end{equation}
Using the standard weights $\mu_{N'}$ on $N'_\R$ and $\mu_N$ on $N_\R$ induced by the lattice structure as in Example \ref{standard weight}, our above considerations in  \ref{weighted short exact sequences} yield a canonical weight $\mu_{L}$ on $\ker(L)_\R$. 
It is easy to see that $\mu_L$ agrees with the image of a generator of $\det(\ker(L))$ in $(\det(\ker(L_\R))\setminus \{0\})/\{\pm 1\}$ and hence $\mu_L$ is given by the integral structure.

Now let $U$ be an open subset of $N_\R$ and let $U'$ be an open subset of $E^{-1}(U)$. 
Then integration along the fibre yields a unique map $E_*\colon A_c^{p,q}(U') \to A_c^{p-k,q-k}(U)$ where $E_*(\alpha')$ is characterized by the projection formula
\[
\int_{[N_\R,\mu_N]} E_*(\alpha') \wedge \eta= \int_{[N'_\R, \mu_{N'}]} \alpha' \wedge E^*(\eta)
\]
for all $\eta \in A_c^{n-p+k,n-q+k}(N_\R)$. It induces a dual map $E^*\colon D^{p,q}(U) \to D^{p,q}(U')$ of Lagerberg currents characterized by 
\[
(E^*(T))(\alpha')= T(E_*(\alpha')) \, \quad (T \in D^{p,q}(U), \, \alpha' \in A_c^{n+k-p,n+k-q}(U')).
\]
We have the properties $d'E^*T=E^*d'T$ and $d''E^*T=E^*d''T$. 
For details, we refer to \cite[\S 2.3]{mihatsch2021}.
\end{rem}

We next recall Mihatsch's definition \cite[\S 2.3, after Definition 2.5]{mihatsch2021} of the pull-back of polyhedral currents in the affine setting.

\begin{rem}[Pull-back of polyhedral currents] \label{pull-back of polyhedral currents by Mihatsch}
Let $E\colon N' \to N$ be still a surjective affine linear map. 
For a weighted polyhedron $[\Delta,\mu_\Delta]$ of $N_\R$, we will  define a weight $\mu_{E^{-1}(\Delta)}$ on the polyhedron $E^{-1}(\Delta)$ of $N_\R'$.
We consider the associated linear map $L \coloneqq E-E(0)$ and 
the short exact sequence
\begin{equation} \label{short exact sequence of preimage} 
0 \longrightarrow \ker(L)_\R \longrightarrow \L_{E^{-1}(\Delta)}  \longrightarrow \L_\Delta \longrightarrow 0
\end{equation}
of $\R$-vector spaces. 
Applying \ref{weighted short exact sequences} to the short exact sequence \eqref{short exact sequence of preimage}, we get a canonical weight $\mu_{E^{-1}(\Delta)}=\mu_L \wedge \mu_\Delta$ on $\L_{E^{-1}(\Delta)}$ using the weight $\mu_L$ on $\ker(L)_\R$ from \ref{pull-back of currents}.

Let $T = \sum_{\Delta \in I} \alpha_\Delta \wedge \delta_{[\Delta,\mu_\Delta]}$ be any polyhedral current on an open subset $U$ of $N_\R$. Mihatsch showed in \cite[\S 2.3]{mihatsch2021} that $E^*(T)$ is a polyhedral current on $E^{-1}(U')$ given by 
\begin{equation}
\label{eq:pull_back_current}
E^*(T)=\sum_{\Delta \in I} E^*(\alpha_\Delta) \wedge \delta_{[E^{-1}(\Delta),\mu_{E^{-1}(\Delta)}]}.
\end{equation}
For any open subset $U'$ of $E^{-1}(U)$, we deduce by restriction from the above that $E^*(T)$ induces a polyhedral current in $D(U')$.
\end{rem}

\begin{prop} \label{polyhedral current and equivalence for pull-back}
Let $E \colon N' \to N$ be a surjective affine map of free abelian groups of finite rank and  let $U'$ be an open subset of $N_\R'$. 
Then $U \coloneqq E(U')$ is an open subset of $N_\R$. 
We consider $T \in D(U)$ and $E^*(T) \in D(U')$ as in Remark \ref{pull-back of currents}. 
Then $T$ is a  polyhedral current on $U$ if and only if $E^*(T)$ is a polyhedral current on $U'$.
\end{prop}

\begin{proof}
We have seen in Remark \ref{pull-back of polyhedral currents by
  Mihatsch} that $E^*(T)$ is polyhedral if $T$ is polyhedral.  
Conversely, we assume that $E^*(T)$ is a polyhedral current on $U'$. 
We claim that $T$ is a polyhedral current. 
This can be checked locally in a neighbourhood of $p \in U$. 
We have
\[
E^{\ast}T=\sum _{\Delta'}\alpha' _{\Delta'} \wedge \delta_{[\Delta' ,\mu _{\Delta'
}]}
\]
for a locally finite family  $([{\Delta'}, \mu_{\Delta'}])_{\Delta' \in I'}$ of weighted polyhedra in $U'$ and $\alpha' _{\Delta'} \in A(\Delta')$. 
We pick $p'$ any point in the fiber $E^{-1}(p)=p'+\ker(L)_\R$ in our previous notation. 
Using the linear map $L \coloneqq E-E(0)$, we choose a splitting $N_\R'= N_\R \times \ker(L)_\R$ of linear spaces with projections $p_1=L_\R$ and $p_2$. 
We pick a point $p'$ in the fiber $E^{-1}(p)$ and an open neighbourhood $\Omega'=\Omega \times \Omega''$ of $p'$ accordingly to our splitting.
Let $k=\dim(\ker(L)_\R)=\dim(N_\R')-\dim(N_R)$. 
Then there is $\rho \in A_c^{k,k}(\Omega'')$ with $\int_{[\ker(L)_\R, \mu_L]} \rho=1$ and hence $E_*(p_2^*(\rho))=1$. 
For any $\omega \in A_c(\Omega)$, we conclude from the projection formula
\[
T(\omega)=T(E_*(p_2^*(\rho))\wedge \omega)=T(E_*(p_2^*(\rho) \wedge E^*\omega))=(E^*T)(p_2^*(\rho)\wedge E^*\omega)).
\]
Since $(E^*T)(p_2^*(\rho)\wedge \cdot))$ is a polyhedral current on $U'$ with relatively compact support over $U$, we conclude that the push-forward of this polyhedral current is a polyhedral current by \cite[\S 2.3]{mihatsch2021}. 
We deduce from the above display that $T$ is a polyhedral current on the open neighbourhood $\Omega$ of $p$.
\end{proof}

Recall that we have defined in \ref{definition of polyhedral currents} when a polyhedral current is constant towards the boundary. Note that for a polyhedral current $T$, we have a unique decomposition $T=\sum_{\sigma \in \Sigma} T_\sigma$ with $T_\sigma$ a polyhedral current of sedentarity $\sigma$. In this way, we can focus easily on polyhedral currents of given sedentarity $\sigma$ and then in a next step on polyhedral currents of sedentarity $0$ by restricting to $\overline{N(\sigma)}$. For such currents, we can use the above pull-back along the canonical projections $\pi_{\sigma} \colon N_\R
\to N(\sigma)$ to give an alternative characterization of constant
towards the boundary.

\begin{prop} \label{constant towards the boundary for polyhedral
    currents} Let $T$ be a polyhedral current on $U$ of tridegree
  $(p,q,r)$ of sedentarity $0$ which is constant towards the
  boundary. Then for every $\sigma \in \Sigma$, there is a unique
  polyhedral current $T(\sigma)$ of tridegree $(p,q,r)$ on the open
  subset $U(\sigma)=U \cap N(\sigma)$ of the affine space $N(\sigma)$
  with $T(0)=T|_{U(0)}$ and with the property that for every
  $x \in U(\sigma)$, there is an open neighbourhood $V$ of $x$ in $U$
  such that
  \begin{equation} \label{constant towards the boundary formula for
      pc} T(0) |_{V\cap \pi_{\sigma}^{-1}(V(\sigma))} = \pi_{\sigma}^*
    (T(\sigma)|_{V(\sigma)})|_{V\cap \pi_{\sigma}^{-1}(V(\sigma))}.
  \end{equation}
  Conversely, if there is for every $\sigma \in \Sigma$ a
  polyhedral current $T(\sigma)$ on $U(\sigma)$ with the above
  property \eqref{constant towards the boundary formula for pc},
  then there is a unique polyhedral current $T$ on $U$ of
  tridegree $(p,q,r)$ of sedentarity $0$ with
  $T|_{U(0)}=T(0)$. Moreover, this polyhedral current $T$ is
  constant towards the boundary and the given $T(\sigma)$ agree
  with the polyhedral currents on $U(\sigma)$ constructed in the
  first part of the statement for $T$.
\end{prop}

\begin{proof} 
The proof starts with the case $T=\delta_{[\Delta,\mu_{\Delta }]}$ for a polyhedron $\Delta$ of sedentarity $0$ contained in $U$ which is constant towards the boundary and a weight $\mu _{\Delta }$ in $\Delta $. 
Let $\sigma \in\Sigma $. 
If $\Delta (\sigma )=\emptyset$, we put $T(\sigma )=0$ that clearly satisfies the condition \eqref{constant towards the boundary formula for pc}. 
If $\Delta (\sigma )\not = \emptyset$, since $\Delta $ is constant towards the boundary, the map $\pi _{\sigma }$ induces an affine map
\begin{displaymath}
\pi _{\sigma ,\Delta }\colon \A_{\Delta }\longrightarrow \A_{\Delta (\sigma )}.
\end{displaymath}
Let $L_{\sigma ,\Delta }=\pi _{\sigma ,\Delta }-\pi _{\sigma ,\Delta}(0)$ be the associated linear map. 
Then $\ker(L_{\sigma ,\Delta})=\ker(\pi _{\sigma })$. 
Thus there exists a unique weight $\mu_{\Delta(\sigma )}$ satisfying
\begin{equation}\label{eq:6}
\mu_{\Delta } = \mu _{\pi _{\sigma }} \wedge \mu_{\Delta (\sigma )}
\end{equation}
where $\mu_{\pi_\sigma}$ is the canonical weight on $\ker(\pi _{\sigma })$ constructed in \ref{weighted short exact sequences}. 
Since $\Delta$ is constant towards the boundary, we have $\A_{\Delta}=\pi_\sigma^{-1}(\A_{\Delta(\sigma)})$ and the construction of the weight $\mu_{\Delta(\sigma )}$ is related to \ref{pull-back of polyhedral currents by Mihatsch}.
Setting $T(\sigma)\coloneqq \delta_{[\Delta(\sigma),\mu_{\Delta (\sigma )}]}$,
we deduce 
\eqref{constant towards the boundary formula for pc} from \eqref{eq:pull_back_current} and \eqref{eq:6}.

Next, we assume that $T=\alpha_\Delta \wedge \delta_{[\Delta,\mu_{\Delta }]}$ for a polyhedron $\Delta$ as above and for $\alpha_\Delta \in A^{p,q}(\Delta)$. 
By definition, the form $\alpha_\Delta$ is the restriction of some $\alpha'_\Delta \in
A^{p,q}(U)$. 
Let $\sigma \in \Sigma $ be a cone. 
If $\Delta (\sigma )= \emptyset$, we put again $T(\sigma)=0$. 
Se we assume that $\Delta (\sigma )\not = \emptyset$.  
Since $\alpha'_\Delta$ is constant towards the boundary, there exists a differential form $\alpha'_\Delta(\sigma )$ in $U(\sigma )$ such that, there exist an open set $U'$ on $N_{\Sigma }$ with $U(\sigma )\subset U'\subset U$ such that 
\begin{equation}\label{eq:7}
\alpha '_{\Delta }|_{U'\cap N_{R}}=\pi_{\sigma } ^{\ast}\alpha'_\Delta(\sigma )|_{U'\cap N_{\R}}.
\end{equation}
We put $\alpha_{\Delta(\sigma )}=\alpha'_\Delta(\sigma )|_{\Delta(\sigma )}$ and  $T(\sigma )=\alpha_{\Delta(\sigma )}\wedge \delta_{[\Delta (\sigma ),\mu _{\Delta (\sigma )}]}$. 
The first case considered above together with \eqref{eq:7} readily shows \eqref{constant towards the boundary formula for pc} for $T$ and $T(\sigma )$. 

In general, the current $T$ is a locally finite sum of currents
$\alpha_\Delta \wedge \delta_{[\Delta,\mu _{\Delta }]}$ considered
above. Locally around $x \in U(\sigma)$, only finitely many $\Delta$
are involved and so we deduce \eqref{constant towards the boundary
  formula for pc} for $T$ by linearity. 

Conversely, assume that for every $\sigma \in \Sigma$, there is a
polyhedral current $T(\sigma)$ on $U(\sigma)$ with \eqref{constant
  towards the boundary formula for pc}. There is a locally finite
family $(\Delta_{\sigma,i})_{i \in I_\sigma}$ of polyhedra in
$U(\sigma)$ of codimension $r$, weights $\mu _{\sigma,i}$
and forms $\alpha_{\sigma,i} \in
A^{p,q}(\Delta_{\sigma,i})$ such that  
\begin{displaymath}
  T(\sigma)=\sum_{i \in I_\sigma} \alpha_{\sigma,i} \wedge
  \delta_{[\Delta_{\sigma,i},\mu _{\sigma,i}]}
\end{displaymath}
on $U(\sigma)$.
For any $x \in U$, there is a unique $\sigma \in \Sigma$ with $x \in U(\sigma)$ and we denote by $V_x$ an open neighbourhood of $x$ for which \eqref{constant towards the boundary formula for pc} holds. 
Clearly, by shrinking $V_x$, we may assume that 
\begin{equation} \label{assumption on V_x}
	V_x \subset \bigcup_{\tau \prec \sigma} \pi_{\tau,\sigma}^{-1}(V_x(\sigma))
\end{equation}
and hence \eqref{constant towards the boundary formula for pc} holds on $V_x(0)$. 
By \cite[Proposition 3.2.12]{burgos-gubler-jell-kuennemann1}, there is a smooth partition of unity $(\varphi_j)_{j \in J}$ of $U$ with compact supports subordinate to the open covering $(V_x)_{x \in U}$. 
For every $j \in J$, we pick $x \in U$ such that $\supp(\varphi_j) \subset V_x$. Let $\sigma \in \Sigma$ with $x \in U(\sigma)$. 

We claim that there is a polyhedron $\Delta_{j,i}$ of $N_\Sigma$ of codimension $r$ contained in $V_x$, constant towards the boundary, such that $\Delta_{\sigma,i}=\Delta_{j,i}\cap U(\sigma)$. To see this, we lift $\Delta_{\sigma,i}$ to a polyhedron $\rho_{\sigma,i}$ of $N_\R$ which is mapped isomorphically to $\Delta_{\sigma,i}$ by $\pi_\sigma$. Then the closure $\Delta_{j,i}$ of $\rho_{\sigma,i}+\sigma$ is a polyhedron in $N_\Sigma$. By Lemma \ref{lemm:1}, we have  $\Delta_{j,i}(\sigma) \coloneqq \Delta_{j,i} \cap N(\sigma)= \Delta_{\sigma,i}$ and we have 
\begin{equation} \label{pull-back and polyhedron}
	\pi_{\sigma,\tau}^{-1}(\Delta_{j,i}(\sigma)) \cap \Delta_{j,i}=\Delta_{j,i}(\tau)
\end{equation} 
for every $\tau \prec \sigma$. In particular, $\Delta_{j,i}$ is constant towards the boundary and the description of the topology of $N_\Sigma$ given in \cite[Remark 3.1.2]{burgos-gubler-jell-kuennemann1} shows that we may assume that 
$\Delta_{j,i}$ is contained in $V_x$ by choosing a suitable lift $\rho_{\sigma,i}$ close enough to $x$.  This proves the above claim.

By pull-back using \eqref{pull-back and polyhedron}, we obtain $\alpha_{j,i} \in A^{p,q}(\Delta_{j,i})$ which extends $\alpha_{\sigma,i}\in A^{p,q}(\Delta_{\sigma,i})$. We claim that the family $(\Delta_{j,i})_{i \in I_\sigma}$ is locally finite in $V_x$. To see this, let $y \in V_x$. By \eqref{assumption on V_x}, we have $\pi_\sigma(y)\in V_x(\sigma) \subset U(\sigma)$ and hence only finitely many $\Delta_{j,i}(\sigma)$ contain $\pi_\sigma(y)$. By Lemma \ref{lemm:1}, we have $\pi_\sigma(\Delta_{j,i})=\Delta_{j,i}(\sigma)$ and hence $y \in \Delta_{j,i}$ for only finitely many $i \in I_\sigma$. 

Recall that for $j \in J$, we have chosen $x \in U$ with
$\supp(\varphi_j) \subset V_x$ and    $\sigma \in \Sigma$ was
determined by $x \in U(\sigma)$. As $x$ and $\sigma$ depend on $j$, we
set $I_j \coloneqq I_{\sigma,i}$ from now on.

We denote by $\mu _{j,i}=\mu _{\pi _{\sigma }}\wedge \mu _{\sigma ,i}$
the weight on $\Delta _{j,i}$ induced by $\mu _{\sigma ,i}$.
Using that $\supp(\varphi_j)$ is a locally finite family for $j \in J$, we conclude that 
\begin{displaymath}
  T \coloneqq \sum_{j \in J} \sum_{i \in I_j} \varphi_j \alpha_{j.i} \wedge \delta_{[\Delta_{j,i},\mu _{j,i}]}
\end{displaymath}
is a polyhedral current of tridegree $(p,q,r)$ on $U$ which is constant towards the boundary. 
For any $\sigma \in \Sigma$, we have 
\begin{displaymath}
  \sum_{j \in J} \sum_{i \in I_j} \varphi_j \alpha_{j.i} \wedge \delta_{[\Delta_{j,i}(\sigma),\mu _{j,i}(\sigma )]}
  = \sum_{j \in J} \sum_{i \in I_\sigma} \varphi_j \alpha_{\sigma.i} \wedge \delta_{[\Delta_{\sigma,i},\mu _{\sigma ,i}]}
  =\sum_{j \in J} \varphi_j T(\sigma)=T(\sigma)
\end{displaymath}
which proves in case $\sigma=0$ that $T|_{U(0)}=T(0)$ and in general
that $T(\sigma)$ agrees with the polyhedral current on $U(\sigma)$
constructed in the first part. Note that uniqueness is clear as any
polyhedral current on $U$ of sedentarity $0$ is determined by its
restriction to $U(0)$. 
\end{proof}

\subsection{Piecewise smooth forms} \label{ps forms and functoriality}
We introduce the bigraded sheaf $\PS$ of differential $\R$-algebras
  of piecewise smooth forms on $N_\Sigma$ and associate polyhedral
  currents with piecewise smooth forms.
\begin{definition} \label{piecewise_smooth}
  For an open subset $U$ of $N_\Sigma$, a \emph{piecewise smooth form}
$\alpha$ on $U$ is given by a locally finite covering
$(\Delta_i)_{i \in I}$ of $U$ by $n$-dimensional polyhedra $\Delta_i$
in $U$ and by
Lagerberg forms $\alpha_i \in A(\Delta_i)$ such that
\begin{displaymath}
  \alpha_i|_{\Delta_i \cap \Delta_j}=\alpha_j|_{\Delta_i \cap \Delta_j}
\end{displaymath}
for all $i,j \in I$. Suppose $\beta$ is a piecewise smooth form on $U$
given by a locally finite covering $U= \bigcup_{j \in J} \Delta_j'$
with polyhedra $\Delta_j'$ in $U$ and by $\beta_j \in
A(\Delta_j')$. We identify $\alpha$ with 
$\beta$ if and only if the collections $\{\Delta_i\}_I\cup
\{\Delta_j\}_J$ and $\{\alpha _i\}_I\cup
\{\beta  _j\}_J$ define a piecewise smooth form. In other words, if
and only if 
\begin{displaymath}
\alpha_i|_{\Delta_i \cap \Delta_j'}=\beta_j|_{\Delta_i \cap \Delta_j'} \in A(\Delta_i \cap \Delta_j')
\end{displaymath}
for all $i\in I$, $j\in J$.  
We say that $\alpha$ has \emph{bidegree $(p,q)$} if $\alpha_i \in A^{p,q}(\Delta_i)$ for all $i \in I$.
\end{definition}

\begin{lemma}\label{lemm:6} 
The assignement that associates to each open set $U$ the space of piecewise smooth forms $\PS(U)$ on $U$ is a bigraded sheaf.
\end{lemma}
\begin{proof}
First we have to prove that it is a presheaf. 
That holds as we can restrict piecewise smooth forms to open subsets. 

Let $W\subset U$ be an inclusion of open subsets and $\alpha =\{(\Delta_{i},\alpha _{i})\}_{i\in I}$ a piecewise smooth form on $U$. 
For each $\Delta_{i}$ let $\{\Delta_{i,j}\}_{j\in J_{i}}$ be a locally finite covering of $\Delta_{i}\cap W$ by full dimensional polyhedra. 
Then $\{\Delta_{i,j}\}_{i\in I,j\in J_{i}}$ is a locally finite covering of $W$ by $n$-dimensional polyhedra. 
Write $\alpha _{i,j}=\alpha _{i}|_{\Delta_{i,j}}$. 
Then the restriction of $\alpha $ to $W$ is given by $\alpha =\{(\Delta_{i,j},\alpha _{i,j})\}_{i\in I,j\in J_{i}}$. 
One can check that this restriction defines a presheaf.

Next we check the conditions for being a  sheaf. 
First we see that we can glue local sections. 
Let $U$ be an open subset and $\{U_{\lambda }\}_{\lambda\in \Lambda }$ an open covering of $U$. 
For each $\lambda $ let $\alpha _{\lambda }=\{(\Delta _{\lambda,i},\alpha _{\lambda ,i})\}_{i\in I_{\lambda }}$ be a piecewise smooth form on $U_{\lambda }$ satisfying that $\alpha _{\lambda }|_{U_{\lambda }\cap U_{\mu }}=\alpha _{\mu }|_{U_{\lambda }\cap U_{\mu }}$. 
We want to define a section $\alpha $ on $U$. 

Since $U$ is paracompact, there is a locally finite subcovering $\{U_{\lambda }\}_{\lambda\in \Lambda' }$, with $\Lambda '\subset \Lambda $.  
Then the covering $ \{\Delta _{\lambda ,i}\}_{\lambda \in \Lambda' ,i\in I_{\lambda }}$ is locally finite and $\alpha =\{\alpha _{\lambda,i},\Delta_{\lambda ,i}\}_{\lambda \in \Lambda ',i\in I_{\lambda}}$ is a piecewise smooth form satisfying $\alpha |_{U_{\lambda}}=\alpha _{\lambda }$.   

  Finally we have to show that, if a piecewise smooth form is locally
  zero then it is zero.  Let $U$ be an open subset,  $\alpha
  =\{(\Delta _{i},\alpha _{i})\}_{i\in I}$ and $\{U_{\lambda }\}_{\lambda
    \in \Lambda }$ an open covering of $U$ such that $\alpha
  |_{U_{\lambda }}=0$. This means that for each $\lambda $ there is a
  covering $\{\Delta _{\lambda ,i}\}_{i\in I_{\lambda }}$ such that,
  for each $i\in I$, $\lambda \in \Lambda $ and $j\in I_{\lambda }$
  the condition $\alpha _{i}|_{\Delta _{i}\cap \Delta _{\lambda
      ,j}}=0$ holds. Since the $\Delta _{\lambda
      ,j}$ cover the whole $\Delta _{i}$ we deduce that $\alpha =0$.
  \end{proof}

\begin{lemma}\label{choose-constant-lemma}  In Definition
  \ref{piecewise_smooth} we can always choose the 
  locally finite covering to consist of polyhedra that are constant 
  towards the boundary. 
\end{lemma}

\begin{proof}
For each cone $\sigma \in \Sigma $ we will denote
\begin{displaymath}
N_\sigma =\bigcup_{\tau \prec \sigma }N(\tau )\subset N_{\Sigma }.
\end{displaymath}
This is the tropical analogue of the affine toric variety $X_\sigma$. 
Then $N(\sigma )$ is the minimal stratum of $N_\sigma $ and there is a canonical projection $\bar\pi_{\sigma } \colon N_\sigma \to N(\sigma )$.
  
Let $\alpha = \{\Delta_i,\alpha_i\}_{i\in I}$ be a piecewise smooth form. 
For every point $p\in U$, let $\sigma _p$ be the unique cone with $p\in N(\sigma _p)$. 
We claim that there is an $n$-dimensional polyhedron $K_p$ such that
\begin{enumerate}
\item 
$K_p\subset U$ and $K_p$ is a neighbourhood of $p$.
\item 
$\rec(K_p\cap N_{\R})=\sigma_p $. 
In particular $K_p$ is constant towards the boundary and $K_p(\sigma _p)=K_p\cap N(\sigma_p )$ is a compact polyhedron.
\item 
The set of $i\in I$ with $\Delta _i\cap K_p\not = \emptyset$ is finite. 
Let $I_p$ be this set.
\item \label{item:6} 
For all $i\in I_p$, the conditions $\Delta _i(\sigma_p )\not =\emptyset$ and
\begin{displaymath}
\alpha _i|_{K_p\cap \Delta _i}=\bar \pi_{\sigma_p}  ^{\ast}(\alpha _i|_{N(\sigma_p)\cap \Delta _i})|_{K_p\cap \Delta _i}
\end{displaymath}
are satisfied. 
The display makes sense because, using Lemma \ref{lemm:1}, if $\Delta _i(\sigma _p)\not =\emptyset$ and $\rec(K_p \cap N_\R)=\sigma_p $, then we have
$K_p\cap\Delta _i\subset \bar \pi_{\sigma_p } ^{-1}(N(\sigma_p )\cap \Delta _i)$.
\end{enumerate}
To prove existence, we set $\sigma = \sigma_p$ for simplicity.
We choose a complementary subspace $\L_\sigma^\perp$ to $\L_\sigma$ in $N_\R$. 
We choose a polytopal neighbourhood $P$ of $0$ and a point $q \in N_\R$ with $\pi_\sigma(q)=p$. 
Then the closure  of $q+P$ in $N_\Sigma$ is a polytope in $N_\Sigma$. 
Moreover, these closures for varying $q$ and $P$ form a basis of neighbourhoods for $p$, see \cite[Remark 3.1.2]{burgos-gubler-jell-kuennemann1}. 
We use $K_p$ for such a neighbourhood choosing $q$ sufficiently close to $p$ and $P$ sufficiently small. 
It is clear that (i) and (ii) are satisfied. 
Since the $n$-dimensional polyhedra $\Delta_i$ form a locally finite covering, it is clear that (iii) is fulfilled if we choose $K_p$ sufficiently close to $p$. 
Condition (iv) follows simply from the fact that $\alpha_i$ is constant towards the boundary again using that we can choose $K_p$ sufficiently close to $p$. 
This proves existence of $K_p$.

We write $\Delta _{p,i}=K_p\cap \pi_{\sigma_p } ^{-1}(N(\sigma _p)\cap \Delta_i)$. 
This is a  polyhedron in $N_\Sigma$ which is constant towards the boundary. 
We also write $\alpha _{p,i}= \pi_{\sigma_p } ^{\ast}(\alpha _i|_{N(\sigma)\cap \Delta _i})$.
We claim that, for every $j\in I$,
\begin{displaymath}
  \alpha _{p,i}|_{\Delta _{p,i}\cap \Delta _j}= \alpha _{j}|_{\Delta
    _{p,i}\cap \Delta _j}. 
\end{displaymath}
If $j\not \in I_p$, then $\Delta _j\cap K_p=\emptyset$ and there is
nothing to prove, so we can assume that $j\in I_p$.
The differential forms $\alpha _j |_{\Delta_j (\sigma )}$, $j\in I_p$ define a
piecewise smooth differential form $\beta $ on $K_p(\sigma _p) \subset
N(\sigma_p )$ and, by condition \ref{item:6} above,   
\begin{displaymath}
  \alpha _{p,i}|_{\Delta _{p,i}\cap \Delta _j\cap N_{\R}}=\bar \pi _{\sigma _p}^{\ast}(\beta )|_{\Delta _{p,i}\cap \Delta _j\cap N_{\R}}=
  \alpha _{j}|_{\Delta
    _{p,i}\cap \Delta _j\cap N_{ \R}}. 
\end{displaymath} 
It remains to be shown that there is a locally finite collection of the
sets $\Delta_{p,i}$ that covers $U$. 
Since $U$ as an open subset of $N_\Sigma$ is $\sigma$-compact and
locally compact, we know that $U$ is exhaustible by compact sets.
 That is, there is a sequence of compact subsets 
\begin{displaymath}
  \dots \subset K_n \subset \tint(K_{n+1}) \subset K_{n+1} \subset
  \tint(K_{n+2})\subset \dots
\end{displaymath}
such that $U=\bigcup_n K_n$. We put $K_n = \emptyset$ for all $n\le
0$. Then for any point in the compact set $p\in K_{n+1}\setminus \tint
(K_{n})$ we can choose the set $K_p$ to be contained in the open
subset $\tint(K_{n+2})\setminus K_{n-1}$. Then there is a finite number
of points $p_{n,1},\dots, p_{n,r_n}$ such that $K_{n+1}\setminus \tint(
K_{n})$ is covered by the sets $K_{p_{n,i}}$ for $i=1,\dots,r_n$. Then
the sets $K_{p_{n,i}}$, $n\ge 1$, $i\le r_n$ form a locally finite
covering of $U$. Since each $K_{p_{n,i}}$ is covered by the finite
family $\Delta _{p_{n,i},j}$, $j\in I_{p_{n,i}}$. we deduce that the
family $\Delta _{p_{n,i},j}$, $n\ge 1$, $i\le r_n$, $j\in I_{p_{n,i}}$
is a locally finite covering of $U$ by polyhedra that are constant
towards the boundary.  

To recap our proof, we can now represent the piecewise smooth form $\alpha$ by the family 
\[
(\Delta_{p_{n,i},j},\alpha_{p_{n,i},j})_{n\geq 1, i\leq r_n,j\in I_{p_{n,i}}},
\] 
as we already checked above.
\end{proof}

\begin{rem}[Piecewise smooth forms as polyhedral currents]
\label{piecewise smooth form and associated polyhedral current}
Piecewise smooth forms give rise to polyhedral currents as follows. 
Suppose that $\alpha$ is a piecewise smooth form on $U$ given as in Definition \ref{piecewise_smooth}. 
For every $x \in U$ and every relatively compact open neighbourhood $W$ of $x$ in $U$, there are only finitely many $i \in I$ with $\Delta_i \cap \overline W \neq \emptyset$. 
We pick these finitely many $n$-dimensional polyhedra $\Delta_i$ and after subdividing, we may assume that they form the $n$-dimensional part of a finite polyhedral complex $\Ccal$ in $U$.  
Let $\mu$ be the weight of $N_\R$ determined by the lattice $N$. 
We define a polyhedral current $[\alpha]_W$ on $W$ by
\begin{displaymath}
\alpha \mapsto \alpha \wedge [\Delta,\mu] \coloneqq
\sum_{\rho \in \Ccal_d} \alpha_\rho \wedge [\rho,\mu].
\end{displaymath}
Obviously, this current does not depend on the choice of the presentation $\alpha_i \in A(\Delta_i)$, $i \in I$. 
By Lemma \ref{choose-constant-lemma}, we can assume that $\Ccal$ is constant towards the boundary. 
As a current is determined locally, we get a well-defined polyhedral current $[\alpha]$ on $U$ which agrees with $[\alpha]_W$ on $W$ for every relatively compact open subset $W$ of $U$. 
We get an  $\R$-linear embedding $\PS(U) \hookrightarrow D(U)$, given by $\alpha \mapsto [\alpha]$. 
\end{rem}
 
\begin{rem}[Polyhedral derivatives for piecewise smooth forms] 
\label{polyhedral derivatives for ps forms}
There are  \emph{polyhedral derivatives} $\dpa$ and $\dpb$ for piecewise smooth forms. 
If $\alpha$ is represented by $(\Delta_i,\alpha_i)_{i \in I}$, then $\dpa \alpha$ is represented by $(\Delta_i,d'\alpha_i)_{i \in I}$ and $\dpb \alpha$ is represented by $(\Delta_i,d''\alpha_i)_{i \in I}$. 
We conclude easily that $\PS$ is a  sheaf  of bigraded differential $\R$-algebras with respect to the differentials $\dpa$ and $\dpb$.
\end{rem}

\begin{rem}[Pull-back for piecewise smooth forms] \label{functoriality of ps}
Let $N_{\Sigma'}'$ be another tropical toric variety with underlying lattice $N'$ of rank $n'$. 
Let  $E\colon N_\R \to N'_{\R}$ be an affine map with underlying $\R$-linear map $L$. 
We assume that $E$ extends to an $L$-equivariant map $F\colon N_{\Sigma} \to N_{\Sigma'}'$, see \S \ref{subsec: notation}.
For any open subset $U$ of $N_\Sigma$ and any open subset $U'$ of $N_{\Sigma'}'$ with $F(U)\subset U'$, we get a natural pull-back
\[
\PS(U') \longrightarrow \PS(U) \quad, \quad \alpha' \mapsto F^*(\alpha')  
\]
given as follows: 
Suppose that $\alpha'$ is given by a locally finite covering $(\Delta_j')_{j \in J}$ of $n'$-dimensional polyhedra $\Delta_j'$ in $U'$ 
and by $\alpha_j' \in A(\Delta_j')$. 
Then for every $j \in J$, we get an $n$-dimensional polyhedron $\Delta_j \coloneqq F^{-1}(\Delta_j')$ in the open subset $F^{-1}(U')$ of $N_\Sigma$ 
The presentation $F^*(\alpha_j')\in A(\Delta_j)$, $j \in J$, gives a piecewise smooth form $F^*(\alpha') \in \PS(F^{-1}(U'))$ which does not depend on the choice of the representation. 
Now we note that $U$ is a subset of $F^{-1}(U')$. 
Replacing $\Delta_j$ by a locally finite family of polyhedra in $U$ covering $\Delta_j \cap U$ and using the Lagerberg form $F^*(\alpha_j')$ for all of them, we get a well-defined $F^*(\alpha') \in \PS(U)$. 
\end{rem}

\subsection{The Theorem of Stokes for smooth forms} \label{subsection theorem of Stokes}
We extend Stokes' Theorem for smooth forms from $N_\R$ to $N_\Sigma$.

\begin{art}\label{contraction}
A $(p,q)$-form $\alpha$ on an open subset $U$ of $N_\R$ is by definition a $(p+q)$-linear form on $N_\R$ with values in $C^\infty(U)$ which is alternating in the first $p$ variable and alternating in the last $q$-variables. 
Using $\Lambda^p M_\R \otimes_{{\R}}\Lambda^q M_{{\R}} \hookrightarrow \Lambda^{p+q} (M_\R\times M_\R)$, we can see $\alpha$ also as an alternating $(p+q)$-form on $N_\R \times N_\R$ which we will write as $\alpha(w_1',\dots,w_p',w_1'',\dots,w_q'')$. This is useful to define the contraction $(\alpha,w)$ for any $w \in N_\R \times N_\R $ as the $(p+q-1)$-form on $N_\R \times N_\R$ given by inserting $w$ for $w_1'$. 
This is an interior derivative as in analysis and, if $w=(w_{1},w_{2}),$ it is characterized by
\begin{displaymath}
(\alpha \wedge \beta, w)= (\alpha,w_{1}) \wedge \beta +
(-1)^{\deg \alpha} \alpha \wedge (\beta,w_{{2}}) 
\end{displaymath}
for a $(p,q)$-form\ $\alpha$ and a $(p',q')$-form $\beta$, and by the identities
\begin{displaymath}
(d'\varphi,w)=\frac{\partial \varphi}{\partial w_1}, \quad
(d''\varphi,w)
=\frac{\partial \varphi}{\partial w_2}
\end{displaymath}
for a smooth function $\varphi$ on $U \times U$. 
For $v \in N_\R$, we set $v'=(v,0)$ and $v''=(0,v)$. Then we see that $(\alpha,v')$ is a $(p-1,q)$-form while $(\alpha,v'')$ is a $(p,q-1)$-form. 
\end{art}

\begin{rem} \label{orthogonal lattice vectors}
Let $[\Delta,\mu]$ be a weighted polyhedron in $N_\R$ of dimension $m$
and $\tau$ a facet (i.e.~a codimension one face) of $\Delta$ endowed
with a weight $\nu$. Then there is a unique vector $n_{\Delta,\tau}$ in $\L_\Delta/\L_\tau$ which points in the direction of $\Delta$ and satisfies $\mu = \nu \wedge
n_{\Delta,\tau}$ using the notation introduced in Remark \ref{weighted short exact sequences}.
\end{rem} 

\begin{rem}[Boundary integrals] \label{definition of boundary integrals}
Now let $\Delta$ be any polyhedron in $N_\Sigma$ of dimension $m$. Our
goal is to define boundary integrals over $\Delta$. 
We may assume that $\Delta$ has sedentarity $\{0\}$ which means that $\Delta$ is the closure of  a polyhedron contained in $N_{\R}$, otherwise we replace $N$ by $N(\sigma)$ for the sedentarity $\sigma$ of $\Delta$. 
Let $\alpha \in A_c^{m-1,m}({N_{\Sigma }})$.
By Lemma \ref{support lemma for top-forms},  $\alpha|_{\Delta }$ has (compact) support in $N_{\R} \cap \Delta$
as  $\dim \Delta =m$.  
It means that we only have to consider facets $\tau$ of $\Delta$ contained in the finite part $N_\R$. 
We define  
\[
\int_{\partial_\tau'[\Delta,\mu]} \alpha \coloneqq
-\int_{[\tau,\nu]} (\alpha,n_{\Delta,\tau}'')|_\tau
\]
which is independent of the choice of $\nu$ and of the normal vector $n_{\Delta,\tau}$. 
Then we define 
\[
\int_{\partial'[\Delta,\mu]} \alpha \coloneqq \sum_{\tau \prec \Delta} 	\int_{\partial_\tau'[\Delta,\mu]} \alpha
\]
where $\tau$ runs over all {finite} facets of $\Delta$.
Similarly, for $\beta \in A_c^{m,m-1}(N_\Sigma)$, we define 
\[
\int_{\partial_\tau''[\Delta,\mu]} \beta \coloneqq \int_{[\tau,\nu]} (\beta,n_{\Delta,\tau}')|_\tau
\]
and 
\[
\int_{\partial'[\Delta,\mu]} \beta \coloneqq \sum_{\tau \prec\Delta}\int_{\partial_\tau'[\Delta,\mu]} \beta
\]
where $\tau$ is running over all {finite} facets of $\Delta$.
We write $\partial'[\Delta,\mu]$ and $\partial''[\Delta,\mu]$ for the corresponding currents.
If $\dim(\Delta)>0$, then the currents
${\partial}'[\Delta,\mu]$ and
${\partial}''[\Delta,\mu]$ 
are never polyhedral. 

Note that the above boundary integrals depend only on the restriction of the form to $\Delta$.
Hence they are well-defined for $\alpha \in A_c^{m-1,m}(\Delta)$ and $\beta \in A_c^{m,m-1}(\Delta)$.
\end{rem}

Using the above notation, we have the \emph{theorem of Stokes} in the same way as in analysis
\begin{thm} \label{theorem of Stokes}
Let $\alpha \in A_c^{m-1,m}(\Delta)$, $\beta  \in
A_c^{m,m-1}(\Delta)$ and $[\Delta,\mu ] $ a weighted  polyhedron of dimension
$m$. Then 
\begin{equation} \label{formula of Stokes}
\int_{[\Delta,\mu]} d'\alpha = \int_{\partial'[\Delta,\mu]} \alpha, \quad  \int_{[\Delta,\mu]} d''\beta = \int_{\partial'[\Delta,\mu]} \beta 
\end{equation}
\end{thm}
\begin{proof}
Let $\sigma \in \Sigma$ be the sedentarity of $\Delta$. 
The support of $\alpha$ is contained in the finite part $\Delta(\sigma)$ by Remark \ref{definition of boundary integrals}.
Hence the claim follows from \cite[Proposition 2.7]{mihatsch2021}.
\end{proof}

\subsection{Residue formula} \label{residue formula}
Let $U$ be an open subset of $N_\Sigma$.

\begin{definition}[Polyhedral derivatives and residues of polyhedral currents] \label{polyhedral derivatives}
Let $[\Delta,\mu]$ be a weighted polyhedron contained in the open subset $U$ of $N_\Sigma$ and let $\alpha \in A^{p,q}(\Delta)$.  
The \emph{polyhedral derivatives $\dpa$ and $\dpb$} of $\alpha\wedge \delta_{[\Delta,\mu]}$ are given as
\[
\dpa(\alpha\wedge \delta_{[\Delta,\mu]})=d'(\alpha)\wedge \delta_{[\Delta,\mu]},\,\,\dpb(\alpha\wedge \delta_{[\Delta,\mu]})=d''(\alpha)\wedge \delta_{[\Delta,\mu]}.
\]
This definition extends to polyhedral currents by linearity.

The \emph{$d'$-residue} of a polyhedral current $T$ is the current defined by $d'T-\dpa T$ and the \emph{$d''$-residue} of $T$ is the current defined by $d''T-\dpb T$.
\end{definition}

\begin{rem}[Compatibilty of polyhedral derivatives]
Note that the polyhedral derivatives of piecewise smooth forms, defined in Remark \ref{polyhedral derivatives for ps forms}, are compatible with the polyhedral dervatives in Definition \ref{polyhedral derivatives} above, if we pass to the associated polyhedral current as in Remark \ref{piecewise smooth form and associated polyhedral current}.
\end{rem}

In the next Proposition, we compute residues.

\begin{prop} \label{prop: residue formula}
Let $T = \sum_{\Delta \in I} \alpha_\Delta \wedge [\Delta,\mu_\Delta]$
be a polyhedral current on $U$ 
of tridegree $(p,q,l)$given by a locally finite family
$([\Delta,\mu_\Delta])_{\Delta  \in I}$ of weighted polyhedra of
codimension $l$ in $U$ and smooth forms $\alpha_\Delta \in
A^{p,q}(\Delta)$ for $\Delta \in I$.  
Then we have 
\[
d'T = \dpa T - \partial' T
\] 
where $\partial' T \in D^{p+l+1,q+l}(U)$ is the current defined by
\[
\langle \partial'T, \eta \rangle \coloneqq \sum_{\Delta \in I}  \int_{\partial' [\Delta,\mu_\Delta]}  \alpha_\Delta \wedge \eta
\]
for $\eta \in A_c^{r-l-p-1,r-l-q}(U)$.
Similarly, we have 
\[
d''T = \dpb T - \partial'' T
\]
where $\partial'' T \in D^{p+l,q+l+1}(U)$ is the current defined by
\[
\langle\partial''T, \eta \rangle \coloneqq \sum_{\Delta \in I}  \int_{\partial''[\Delta,\mu_\Delta]}  \alpha_\Delta \wedge \eta.
\]
for $\eta \in A_c^{r-l-p,r-l-q-1}(U)$.
\end{prop}

\begin{proof}
We prove the first formula, the second one is similar. By linearity, we may assume that $T= \alpha \wedge [\Delta,\mu]$. Then 
\[
\langle d'T, \eta \rangle = (-1)^{p+q+1}\int_{[\Delta,\mu]} \alpha \wedge d'\eta
=\int_{[\Delta,\mu]} d'(\alpha) \wedge \eta  - \int_{[\Delta,\mu]} d'(\alpha \wedge \eta).
\]
Using the theorem of Stokes, we get 
\[
d'(\alpha \wedge [\Delta,\mu])= \dpa(\alpha \wedge [\Delta,\mu])- \alpha \wedge\partial'[\Delta,\mu]
\]
as claimed.
\end{proof}

\subsection{The Theorem of Stokes for polyhedral currents} \label{subsection: Stokes for polyhedral currents}
Let $U$ be an open subset of $N_\Sigma$. 
Let $\alpha$,$\beta$ and $\gamma$ be polyhedral currents on $U$ of tridegree $(p,p,n-p)$, $(p-1,p,n-p)$ and $(p,p-1,n-p)$.

\begin{prop}\label{special-types-are-delta-forms}
The currents $d'\beta$ and $d''\gamma$ are sums of polyhedral currents of tridegree $(p,p,n-p)$ and $(p-1,p-1,n-p+1)$.
\end{prop}

\begin{proof}
We can work locally and assume by linearity that $\beta =\beta_0
\wedge \delta_{[\sigma,\mu_\sigma]}$ and $\gamma =\gamma_0 \wedge
\delta_{[\sigma,\mu_\sigma]}$. 
If $\tau$ is a codimension one face of $\sigma$, then, as in
\ref{orthogonal lattice vectors} we choose a weight $\mu_\tau$ for
$\tau$ and a vector $n_{\sigma,\tau}\in N_\sigma$ such that
$\mu_\sigma=\mu_\tau\wedge n_{\sigma,\tau}$. Then, the contractions
$(\beta_0,n''_{\sigma,\tau})|_\tau$ and
$(\gamma_0,n'_{\sigma,\tau})|_\tau$ in $A^{p,p}(\tau)$ have been
defined in \ref{contraction}. 
It follows from Proposition \ref{prop: residue formula} that we have 
\begin{align}
\label{relate-d-and-dpa}
d'\beta&=(d'\beta_0)\wedge \delta_{[\sigma,\mu_\sigma]}+\sum_{\genfrac{}{}{0pt}{}{\tau\prec \sigma}{{\dim \tau=p-1}}}(\beta_0,n''_{\sigma,\tau})|_\tau\wedge \delta_{[\tau,\mu_\tau]},\\
\label{relate-d-and-dpb}
d''\gamma&=(d''\gamma_0)\wedge \delta_{[\sigma,\mu_\sigma]}-\sum_{\genfrac{}{}{0pt}{}{\tau\prec \sigma}{{\dim \tau=p-1}}}(\gamma_0,n'_{\sigma,\tau})|_\tau\wedge \delta_{[\tau,\mu_\tau]}.
\end{align}
These formulas show that $d'\beta$ and $d''\gamma$ are polyhedral of the given tridegrees.
\end{proof}

Let $\Delta$ be a polyhedron in $U\subseteq N_\Sigma$ of dimension $n$ equipped with the canonical weight $\mu$ induced by the lattice $N$.
In the following, we assume that the currents $\alpha$,$\beta$ and $\gamma$ 
have compact support.
By definition and our compact support assumption, we find a finite family $([\sigma,\mu_\sigma])_{\sigma \in I}$ of weighted polyhedra of dimension $p$ in $U$ and smooth forms $\alpha_\sigma,\beta_\sigma,\gamma_\sigma \in A_c(\sigma)$ for every  $\sigma\in I$ such that  
\begin{equation}\label{present-alpha-beta}
\alpha = \sum_{\sigma \in I} \alpha_\sigma \wedge \delta_{[\sigma,\mu_\sigma]},\,\,\,
\beta = \sum_{\sigma \in I} \beta_\sigma \wedge \delta_{[\sigma,\mu_\sigma]},\,\,\,
\gamma = \sum_{\sigma \in I} \gamma_\sigma \wedge \delta_{[\sigma,\mu_\sigma]}.
\end{equation}
We can refine the family $([\sigma,\mu_\sigma])_{\sigma \in I}$  that for each each polyhedron $\sigma\in I$ the intersection $\sigma\cap \Delta$ is either empty or a face of $\sigma$.

\begin{definition}\label{define-boundary-delta-integrals}
In order to define the integral $\int_\Delta\alpha$ and the boundary integrals $\int_{\partial'\Delta}\beta$ and $\int_{\partial''\Delta}\gamma$, we can assume by linearity that $\alpha = \alpha_0 \wedge \delta_{[\sigma,\mu_\sigma]}$, $\beta =\beta_0 \wedge \delta_{[\sigma,\mu_\sigma]}$ and $\gamma =\gamma_0 \wedge \delta_{[\sigma,\mu_\sigma]}$ for some weighted polyhedron $\sigma$ of dimension $p$ such that $\tau\coloneqq \sigma\cap \Delta$ is either empty of a face of $\sigma$.

If $\tau$ is empty or $\dim \tau<p-1$, then we define 
\begin{align}\label{empty-case}
&\int_\Delta\alpha\coloneqq 0,\,\,\,
\int_{\partial'\Delta}\beta\coloneqq 0,\,\,\,
\int_{\partial''\Delta}\gamma\coloneqq 0.
\end{align}
If $\dim \tau=p-1$, then we choose $n_{\sigma,\tau}$ and $\mu_\tau$ in a compatible way as in \ref{orthogonal lattice vectors} and define 
\begin{align}\label{codimension-one-case}
&\int_\Delta\alpha\coloneqq 0,\,\,\,
\int_{\partial'\Delta}\beta\coloneqq \int_{[\tau,\mu_\tau]}(\beta_0,n''_{\sigma,\tau})|_\tau,\,\,\,
\int_{\partial''\Delta}\gamma\coloneqq -\int_{[\tau,\mu_\tau]}(\gamma_0,n'_{\sigma,\tau})|_\tau.
\end{align}
If $\dim \tau=p$ and hence $\tau=\sigma$, then we define 
\begin{align}\label{full-dimension-case}
&\int_\Delta\alpha\coloneqq \int_{[\sigma,\mu]}\alpha_0,\,\,\,
\int_{\partial'\Delta}\beta\coloneqq 0,\,\,\,
\int_{\partial''\Delta}\gamma\coloneqq 0.
\end{align}
\end{definition}

The above definitions of the boundary integrals might look strange. 
One should notice that if $\sigma$ is contained in $\Delta$, then the derivative $d'\beta$ in the sense of currents of the polyhedral current $\beta=\beta_0 \wedge \delta_{[\sigma,\mu_\sigma]}$ includes already boundary contributions as we have seen in \eqref{relate-d-and-dpa}. 
Hence it makes sense to define the boundary integral of $\beta$ as $0$. 
On the other hand, if the intersection of $\Delta$ and $\sigma$ is a codimension $1$ face $\tau$ of $\sigma$, then there is a boundary integral needed. 
This will become clear in the proof of the Theorem of Stokes below.

\begin{thm}\label{polyhedral-stokes-formula-proposition}
Given polyhedral currents $\beta$ and $\gamma$ with compact support of tridegree $(p-1,p,n-p)$ and $(p,p-1,n-p)$ on an open subset $U\subset N_\Sigma$, we have
\begin{equation}\label{polyhedral-stokes-formula}
\int_\Delta d'\beta=\int_{\partial'\Delta}\beta,\,\,\,
\int_\Delta d''\gamma=\int_{\partial''\Delta}\gamma.
\end{equation}
\end{thm}

\begin{proof}
We prove only the first equality. 
The second equality is proved in the same way.
As before, we may assume that $\beta =\beta_0 \wedge \delta_{[\sigma,\mu_\sigma]}$ such that $\tau\coloneqq \sigma\cap \Delta$ is empty or a face of $\sigma$. 

If $\tau$ is empty or $\dim(\tau)<p-1$, then all terms in \eqref{polyhedral-stokes-formula} vanish by definition.
If $\dim(\tau)=p-1$, then using \eqref{relate-d-and-dpa},\eqref{codimension-one-case} and \eqref{full-dimension-case}, we get
\[
\int_\Delta d'\beta
=\int_{[\tau,\mu_{\tau}]}(\beta_0,n_{\sigma,\tau}'')|_{\tau}
=\int_{\partial'\Delta}\beta.
\]
If $\dim(\tau)=p$, then \eqref{relate-d-and-dpa}, \eqref{full-dimension-case}, and Stokes Theorem \eqref{theorem of Stokes} yield
\[
\int_{[\Delta,\mu]}d'\beta
=\int_{[\sigma,\mu_\sigma]}d'\beta_0-\int_{\partial'[\sigma,\mu_\sigma]}\beta_0=0
\]
and the right-hand side of the first equality in \eqref{polyhedral-stokes-formula} vanishes by definition.
\end{proof}

\section{Delta forms}\label{section-delta-forms} \label{section: delta-forms}

Let $N$ be a free abelian group of rank $n$. 
Mihatsch has introduced sheaves of $\delta$-forms on $N_\R$.
The goal of this section is to  generalize Mihatsch's construction to
a tropical toric variety $N_\Sigma$ for a rational polyhedral fan $\Sigma$ in $N_\R$.

\subsection{Delta forms on $N_\R$}\label{delta-forms-on-N_R}
{Mihatsch has introduced $\delta$-forms on $N_\R$
as polyhedral currents $T$ satisfying that $d'T$ and $d''T$ are both
polyhedral current \cite[Def. 3.1]{mihatsch2021}s. He showed
\cite[Thm. 3.3]{mihatsch2021} that this
condition is equivalent to  a 
balancing condition similar to the one of tropical geometry and that,
for a polyhedral current $T$, it is enough that one of  $d'T$ or
$d''T$ is polyhedral to be a $\delta $-form.}

For example, the current of integration of a tropical cycle and the current associated to a piecewise smooth form are $\delta$-forms. 
Mihatsch showed that the delta-forms build a bigraded differential $\R$-algebra with respect to a natural $\wedge$-product and natural differentials $d',d''$.  
Moreover, there is a natural pull-back of $\delta$-forms with respect to affine maps $f\colon N' \to N$ of free abelian groups of finite rank which agrees with the pull-back from \ref{pull-back of polyhedral currents by Mihatsch} in case of a surjective affine map.
There is also a natural integral $\int_{U}\colon B_c^{n,n}(U)\to \R$ for $\delta$-forms with compact support.

\begin{rem} \label{Mihatsch's notation}
Mihatsch's definition of $\delta$-forms is local, so we get for any open subset $U$ of $N_\R$ a bigraded $\R$-algebra $B^{\cdot,\cdot}(U)$ with differentials $d',d''$. 
In fact, the trigrading of polyhedral currents from  Definition \ref{definition of polyhedral currents} gives a trigrading
$B(U)= \oplus_{p,q,r} B^{p,q,r}(U)$ related to the bigrading by 
\[
B^{p,q}(U)=\bigoplus_{r\in \N}B^{p-r,q-r,r}(U).
\]
This gives a sheaf $B$ of bigraded differential algebras on $N_\R$. 

If $f\colon N' \to N$ is an affine map of free abelian groups of finite rank, then we get a homomorphism $f^*\colon B(U) \to B(f^{-1}(U))$ of differential bigraded $\R$-algebras which preserves the trigrading. 
If $f$ is surjective, then $f^*$ is induced by the pull-back of polyhedral currents from \ref{pull-back of polyhedral currents by Mihatsch}.
\end{rem}

\begin{lemma} \label{balancing}
Let $f \colon N' \to N$ be a surjective affine map of abelian groups of finite rank and let $U'$ be an open subset of $N_\R'$ and $U$ the open subset $f(U')$. 
For a polyhedral current $T$ on $U$ of tridegree $(p,q,r)$, the
polyhedral current $f^*(T)$ is a $\delta$-form on $U'$ if and only if
$T$ is a $\delta$-form on $U$. 
\end{lemma}

\begin{proof}
We have seen in Proposition \ref{polyhedral current and equivalence for pull-back} that a current $T$ on $U$ 
is polyhedral if and only if $f^{\ast}T$ is polyhedral on $U'$. 
The lemma follows from the fact that $f^{\ast}(d'T)=d'f^{\ast}T$ so $d'T$ is polyhedral if and only if $d'f^{\ast}T$ is polyhedral. 
\end{proof}

\subsection{Delta forms on $N_\Sigma$}\label{delta-forms-on-N_Sigma}
Let $\Sigma$ be a rational polyhedral fan in $N_\R$ with associated partial compactification $N_\Sigma$.

\begin{definition} \label{definition of delta-form}
Let $U$ be an open subset of $N_\Sigma$.
A \emph{$\delta$-form $\alpha$ of type $(p,q,r)$ on $U$} is a
polyhedral current $\alpha$  of type $(p,q,r)$ of sedentarity $0$ on
$U$ which is constant towards the boundary such that the current
$d'\alpha|_{U \cap N_R}$ is a polyhedral current on the open subset
$U \cap N_\R$ of $N_\R$. 
In other words $\alpha|_{U \cap N_R}$ is a $\delta $-forms in the sense of Mihatsch's definition recalled in \S \ref{delta-forms-on-N_R}.
We denote the space of $\delta$-forms of type $(p,q,r)$ on $U$ by $B^{p,q,r}(U)$ and the space of $\delta$-forms of type $(p,q,r)$ on $U$ with compact support in $U$ by $B_c^{p,q,r}(U)$ and put 
\[
B^{p,q}(U)=\oplus_{r\in \N}B^{p-r,q-r,r}(U),\,\,\,
B(U)=\oplus_{p,q\in \N}B^{p,q}(U).
\] 
\end{definition}

\begin{rem}\label{properties-delta-forms-compactification}
  \begin{enumerate}
  \item \label{item:7} {It follows from Proposition \ref{polyhedral currents form sheaf} that
      $U\mapsto B(U)$ defines a sheaf of trigraded real vector spaces
      on $N_\Sigma$.} 
    \item \label{item:8} {A $\delta$-form in $B^{n,n}(U)$ is constant towards
      the boundary.  Hence it vanishes near the boundary for degree
      reasons and has support in $N_\R\cap U$.  Consequently the
      integral in Remark \ref{Mihatsch's notation} induces a unique
      integral $\int_U\colon B^{n,n}_c(U)\to \R$ such that
      $\int_U\alpha=\int_{N_\R\cap U}\alpha|_{N_\R\cap U}$ for all
      $\alpha\in B_c^{n,n}(U)$.}
  \end{enumerate}  
\end{rem}

\begin{ex}\label{example_classical_tropical_cycle}
A classical tropical cycle $C$ in $N_\R$ (in the sense of Definition \ref{definition of a classical tropical cycle} below), can be seen as a polyhedral current of the form $\delta _C=\sum_{\Delta}[\Delta ,\mu_\Delta ]$ (i.e. the differential forms are all
    constant functions) 
    satisfying a balancing condition. This balancing 
    condition is equivalent to $d'\delta _C=0$ or equivalently to $d''\delta _C=0$, see \cite[Proposition 3.8]{gubler-forms}.  
  Let $C$ be a classical tropical cycle in $N_\R$ of codimension
  $r$. Let $\overline C$ be the closure in 
  $N_\Sigma$, then $\delta_{\overline C}$ is a polyhedral current. We claim that
 $d'\delta _{\overline C}=d''\delta
  _{\overline C}=0$.

Let $\eta$ be a $(n-r-1,n-r)$-smooth form on $N_{\Sigma }$ with
  compact support. 
By Lemma \ref{support lemma for top-forms}, the forms $\eta|_{C}$ and  $d'\eta|_{C}$
  have compact support contained in $N_{\R}$. Let $U\subset N_{\R}$ be an open set containing the
  support of $d'\eta|_{C}$. Then 
  \begin{displaymath}
    d'\delta _{\overline C}(\eta)=-\delta_{\overline C}(d'\eta)=
    -\delta_{\overline C}|_{U}(d'\eta)=d'\delta _{C}(\eta)=0. 
  \end{displaymath}
  This proves $d'\delta_{\overline C}=0$. A similar argument shows that $d''\delta _{\overline C}=0$.

If in addition $\overline C$ is constant towards the
  boundary, then $\delta _{\overline C}$ is a closed $\delta
  $-form. By definition if $\overline C$ is not constant towards the
  boundary, then $\delta _{\overline C}$ is a closed polyhedral
  current, but not 
  a  $\delta
  $-form.
\end{ex}

\begin{prop}\label{special-presentation-delta-forms}
If the fan $\Sigma$ is simplicial and $U\subset N_\Sigma$ is open, then every $\delta$-form $\alpha \in B^{p,q}(U)$ admits globally on $U$ a presentation 
\begin{equation}\label{def-polyhedral-current2}
T= \sum_\Delta \alpha_\Delta \wedge \delta_{[\Delta,\mu_\Delta]\in I} \in D(U)
\end{equation}
where the polyhedra $\Delta$ belong to a weighted polyhedral complex $(\Ccal,\mu)$ that is constant towards the boundary with support  $|\Ccal|=U$.
\end{prop}

\proof
This is a direct consequence of Remark \ref{polyhedral complex of definition for polyhedral currents}.
\qed

\begin{prop} \label{criterion for delta-forms}
Let $U$ be an open subset of $N_\Sigma$. Let $\alpha$ be a polyhedral current on $U$ of sedentarity $0$ which is constant towards the boundary.
Then $\alpha$ is a $\delta$-form if and only if $d'\alpha$ is a polyhedral current.
\end{prop}

\begin{proof}
Let us first assume that $d'\alpha$ is polyhedral.
Then $\alpha|_{U \cap N_\R}$ and $d'\alpha|_{U\cap N_\R}$ are both polyhedral and hence $\alpha$ is a $\delta$-form.

For the converse implication assume that $\alpha$ is a $\delta$-form.
If $\alpha$ is $\delta$-form, then each component $\alpha^{p,q,r}$ is a $\delta$-form by Definition \ref{definition of delta-form} and \cite[Theorem 3.3]{mihatsch2021}.
By linearity, we may assume that $\alpha$ has a fixed type $(p,q,r)$. 
Then there is a locally finite family $[\Delta,\mu_\Delta]$ of weighted polyhedra of codimension $l$ in $U$  
and $\alpha_\Delta \in A^{p,q}(\Delta)$ such that
$\alpha=\sum_{\Delta \in I} \alpha_\Delta \wedge [\Delta,\mu_\Delta]$.
It follows from the definitions that it is enough to prove that $\partial'T$ is a polyhedral current where $\partial'=\dpa-d'$ is the differential considered in the residue formula
\[
\langle \partial'\alpha, \eta \rangle = \sum_{\Delta \in I}  \int_{\partial' [\Delta,\mu_\Delta]}  \alpha_\Delta \wedge \eta.
\]
from Proposition \ref{prop: residue formula}.  
The definition of the boundary integrals in \ref{definition of boundary integrals} shows that only facets of sedentarity $0$ play a role. 
Therefore $\partial'\alpha$ is the canonical extension of $\partial'\alpha |_{N_{\R}}$ which is polyhedral because $\alpha|_{N_\R}$ is a $\delta $-form. 
\end{proof}

\begin{rem} \label{differentials for delta-forms}
Let $U$ be an open subset of $N_\Sigma$.
\begin{enumerate}
\item Given a $\delta$-form $\alpha$ on $U$, we conclude from
  Proposition \ref{criterion for delta-forms} that $d'\alpha$ and
  $d''\alpha$ are $\delta$-forms on $U$ as
  $d'd'\alpha = 0=d'' d'' \alpha $.  Hence
    $U\mapsto (B^{\cdot,\cdot}(U),d',d'')$ defines by Remark
    \ref{properties-delta-forms-compactification}~\ref{item:7} a sheaf of
    differential bigraded $\R$-algebras on $N_\Sigma$.
\item Polyhedral currents of sedentarity zero on $U$ of
    tridegree $(p-1,p,n-p)$ or $(p,p-1,n-p)$ are always $\delta$-forms
    by Proposition \ref{criterion for delta-forms} and Corollary
    \ref{special-types-are-delta-forms}.
\end{enumerate}
\end{rem}

\subsection{Products and pull-backs of $\delta$-forms on $N_\Sigma$.}\label{products-pullbacks-delta-forms-on-N_R}
Let $U \subset N_\Sigma$ be open. 
{For  $\sigma \in \Sigma $,} 
we write $U(\sigma) \coloneqq U \cap {N(\sigma)}$. 
For  $\omega \in A^{p,q}(U)$, we write $\omega(\sigma)=\omega|_{U(\sigma)}$. By definition, $\omega$ is constant towards the boundary which means the following. For every $\sigma \in \Sigma$ and for every $x \in U(\sigma)$, there is an open neighbourhood $V$ of $x$ in $U$ such that for
all $\tau \prec \sigma$ we have
\begin{displaymath}
	\omega(\tau) |_{V(\tau)\cap \pi_{\sigma,\tau}^{-1}(V(\sigma))} 
	= \pi_{\sigma, \tau}^* 
	(\omega(\sigma)|_{V(\sigma)})|_{V(\tau)\cap \pi_{\sigma,\tau}^{-1}(V(\sigma))} 
\end{displaymath}
where we use the canonical map $\pi_{\sigma,\tau}\colon N(\tau) \to N(\sigma)$ as usual. 
We have seen a similar characterization of polyhedral currents which are constant towards the boundary in Proposition \ref{constant towards the boundary for polyhedral currents}. 
The next result transfers this characterization to $\delta$-forms on $U$.

\begin{prop} \label{delta-forms and constant towards the boundary}
Let $\alpha$ be a $\delta$-form in $B^{p,q,r}(U)$. Then for every $\sigma \in \Sigma$, there is a unique $\delta$-form $\alpha(\sigma)\in B^{p,q,r}(U(\sigma))$ on the open subset $U(\sigma)$ of the affine space $N(\sigma)$ with $\alpha(0)=\alpha|_{U(0)}$ and with the property that for  every $x \in U(\sigma)$, there is an open neighbourhood $V$ of $x$ in $U$ such that for
all $\tau \prec \sigma$ we have
\begin{equation} \label{constant towards the boundary formula}
	\alpha(\tau) |_{V(\tau)\cap \pi_{\sigma,\tau}^{-1}(V(\sigma))} 
	= \pi_{\sigma, \tau}^* 
	(\alpha(\sigma)|_{V(\sigma)})|_{V(\tau)\cap \pi_{\sigma,\tau}^{-1}(V(\sigma))}. 
\end{equation}
Conversely, if there is for every $\sigma \in \Sigma$ a $\delta$-form $\alpha(\sigma)\in B^{p,q,r}(U(\sigma))$ with the above property \eqref{constant towards the boundary formula}, then there is a unique $\delta$-form $\alpha \in B^{p,q,r}(U)$ with $\alpha|_{U(0)}=\alpha(0)$. 
Moreover, the given $\delta$-forms $\alpha(\sigma)$ agree with the $\delta$-forms on $U(\sigma)$ constructed in the first part of the statement for $\alpha$.
\end{prop}

\begin{proof} 
Suppose first that $\alpha \in B^{p,q,r}(U)$. For $\sigma \in \Sigma$, we have the polyhedral currents $\alpha(\sigma)$ on $U(\sigma)$ from Proposition \ref{constant towards the boundary for polyhedral currents}. It is enough to show that $\alpha(\sigma)$ is a $\delta$-form. 
Since $\alpha$ is a $\delta$-form, we have that $\alpha(0)$ is a $\delta$-form on $U(0)$. 
For every $x \in U(\sigma)$, the description of the topology given in \cite[Remark 3.1.2]{burgos-gubler-jell-kuennemann1} 
shows that there is an open neighbourhood $W$ of $x$ in $U(\sigma)$ and an open subset $W'$ of $N_\R$ with $x \in \pi_\sigma(W')\subset U(\sigma)$. We apply Lemma \ref{balancing} for the open subsets $W\coloneq \pi_\sigma(W')$ and $W'$ of $N(\sigma)$ and $N_\R$, respectively. Since $\alpha(0)$ is a $\delta$-form on $W'$, we deduce from \eqref{constant towards the boundary formula} for $\tau=\{0\}$ that $\alpha(\sigma)$ is a $\delta$-form on $W$. This proves that $\alpha(\sigma)$ is a $\delta$-form on $U(\sigma)$. 

Conversely, assume that all the polyhedral currents $\alpha(\sigma)$ are $\delta$-forms. By Proposition \ref{constant towards the boundary for polyhedral currents}, the conditions \eqref{constant towards the boundary formula} yield that $\alpha$ is a polyhedral current of sedentarity $0$ which is constant towards the boundary. Since $\alpha(0)$ is a $\delta$-form on $U(0)$, we conclude that $\alpha$ is a $\delta$-form. 
\end{proof}

\begin{thm} \label{product on delta-forms}
Let $U$ be an open subset of $N_\Sigma$. Then there is a unique product $\alpha \wedge \beta$ for $\delta$-forms $\alpha,\beta$ on $U$ which agrees on the dense stratum $U(0)=N_\R \cap U$ with Mihatsch's product from Remark \ref{Mihatsch's notation}. For any $\sigma \in \Sigma$, we have 
\begin{equation}  \label{characterization of product of delta-forms}
(\alpha \wedge \beta)(\sigma) = \alpha(\sigma) \wedge \beta(\sigma)
\end{equation}
using Mihatsch's product of $\delta$-forms on $U(\sigma) = U \cap N(\sigma)$. 
In this way, the bigraded sheaf $B=\bigoplus_{p,q} B^{p,q}$ of $\delta$-forms on $N_\Sigma$ gets a bigraded differential sheaf of $\R$-algebras with respect to the natural differentials $d',d''$. 
\end{thm}

\begin{proof}
The $\delta$-forms $\alpha$ and $\beta$ satisfy \eqref{constant towards the boundary formula}. For $\sigma \in \Sigma$, we define $(\alpha \wedge \beta)(\sigma) \in B(U(\sigma))$ by \eqref{characterization of product of delta-forms} using Mihatsch's product on $U(\sigma)$. By construction, the $\delta$-forms $(\alpha \wedge \beta)(\sigma)$ satisfy \eqref{constant towards the boundary formula} and so the claim follows from Proposition \ref{delta-forms and constant towards the boundary}. The last claim carries immediately over from the affine case given in Remark \ref{Mihatsch's notation}.
\end{proof}

We want to define pull-backs of $\delta$-forms in the following setting. 
Let $N_{\Sigma'}'$ be another tropical toric variety with underlying lattice $N'$ of rank $n'$, let $L\colon N_\R' \to N_\R$ be an $\R$-linear map and let $f\colon N_\Sigma \to N_{\Sigma'}'$ be any $L$-equivariant morphism of tropical toric varieties, see \S \ref{subdivision}.
For every $\sigma \in \Sigma$, there is a unique $\sigma' \in \Sigma'$ such that $f$ restricts to an affine map	
\[
f_\sigma \colon N(\sigma) \to N'(\sigma')
\]
of strata. In particular, let $\rho \in \Sigma'$ be the cone corresponding to $\{0\}$. 

\begin{prop} \label{pull-back of delta-forms}
Using the above notation, let $U$ (resp.~$U'$)  an open subset of
$N_\Sigma$ (resp.~of $N'_{\Sigma'}$) such that $f(U)\subset U'$. For
every $\delta$-form $\alpha' \in B(U')$, there is a unique $\delta
$-form  $f^*(\alpha') \in B(U)$ with $(f^*(\alpha'))(0)=f_0^*(\alpha'(\rho))$. Moreover, for any $\sigma \in \Sigma$, we have
$$(f^*(\alpha'))(\sigma)=f_\sigma^*(\alpha'(\sigma')).$$
Finally, we note that $f^*$ induces a homomorphism $B_{N_{\Sigma'}'} \to f_*B_{N_\Sigma}$ of sheaves of bigraded differential $\R$-algebras.
\end{prop}
	
\begin{proof}
The map $L$ induces a linear map $N \to N'(\rho)$. 
The cone $\sigma$ is mapped to the cone $\sigma'$ modulo $\L_{\rho}$. 
For any face $\tau$ of $\sigma$, we get a corresponding face $\tau'$ of $\sigma'$. 
We note that the $\delta$-form $\alpha'$ on $U'$ satisfies \eqref{constant towards the boundary formula} for this face $\tau'$ of the cone $\sigma'$ which means that for every $x' \in U'(\sigma')$, there is an open neigbhourhood $V'$ of $x'$ with 
\begin{equation} \label{constant towards the boundary formula on N'}
\alpha'(\tau') |_{V'(\tau')\cap \pi_{\sigma',\tau'}^{-1}(V'(\sigma'))}= \pi_{\sigma', \tau'}^*(\alpha'(\sigma')|_{V'(\sigma')})|_{V'(\tau')\cap \pi_{\sigma',\tau'}^{-1}(V'(\sigma'))}. 
\end{equation}
We define $\alpha(\sigma)\coloneqq f_\sigma^*(\alpha'(\sigma')) \in B(U(\sigma)$. 
We use the criterion in Proposition \ref{delta-forms and constant towards the boundary} to show that they induce a unique $\delta$-form $\alpha$ on $U$. 
For $x \in U(\sigma)$, let $x' \coloneqq f(x)$ which is a point of $U'(\sigma')$. 
Using the open neighbourhood $V'$ for this $x'$ from above, we set $V \coloneqq f^{-1}(V') \cap U$ which is an open neighbourhood of $x$ in $U$ which satisfies \eqref{constant towards the boundary formula} for the face $\tau$ of $\sigma$ by using pull-back and by using $\pi_{\sigma',\tau'}\circ f_\tau = f_\sigma \circ \pi_{\sigma,\tau}$. 
It follows from Proposition \ref{delta-forms and constant towards the boundary} that $\alpha$ is a $\delta$-form on $U$ uniquely characterized by $\alpha(0)=f_\sigma^*(\alpha'(\rho))$. 
The final claim follows from the corresponding properties of pull-backs in the affine situation given in Remark \ref{Mihatsch's notation}. 	 
\end{proof}
	
\begin{prop} \label{piecewise smooth forms as delta-forms}
Let $U$ be an open subset of $N_\Sigma$. 
Then $B^{p,q,0}(U)$ is precisely the space of piecewise smooth forms on $U$ of bidegree $(p,q)$. 	
\end{prop}

\begin{proof}
By \cite[Lemma 3.7]{mihatsch2021}, this is true in the affine situation on $U(\sigma)$. 
Using Proposition \ref{delta-forms and constant towards the boundary}, this easily yields the claim as we have seen in Lemma \ref{choose-constant-lemma} that piecewise smooth forms in \S \ref{ps forms and functoriality} can be given by locally finite coverings consisting of polyhedra  which are constant towards the boundary.
\end{proof}

\subsection{Formal $\delta$-currents on $N_\Sigma$.}\label{formal-delta-currents}
Let $U$ be an open subset of the partial compactification $N_\Sigma$.

\begin{definition} \label{formal delta-current partial compactifiction}
A \emph{formal $\delta$-current} on $U$ is a linear functional on the space $B_c(U)$ of compactly supported $\delta$-forms. 
Using partitions of unity \cite[Proposition 3.2.12]{burgos-gubler-jell-kuennemann1}, we obtain a sheaf $E$ of $\R$-vector spaces in $N_\Sigma$ 
with bigrading defined by
\[
E^{p,q}(U)\coloneqq \Hom(B_c^{n-p,n-q}(U),\R).
\]
\end{definition}

As we will not need them, we will not introduce $\delta$-currents on $U$ as elements of the topological dual of the space of $\delta$-forms with compact support for a suitable topology.
Instead we will only work with formal $\delta$-currents.

\begin{rem}
There are unique $\R$-linear differentials and a product
\begin{align*}
d'\colon &E^{r,s}(U )\longrightarrow E^{r+1,s}(U),\,\,\,
d''\colon E^{r,s}(U )\longrightarrow E^{r,s+1}(U )\\
\wedge\colon &B^{p,q}(U) \otimes_\R E^{r,s}(U)\longrightarrow E^{p+r,q+s}(U)
\end{align*}
such that
\begin{align*}
d'(T)(\alpha)&=(-1)^{r+s+1}T(d'\alpha),\,\,\,
d''(T)(\beta)=(-1)^{r+s+1}T(d''\beta),\\
(\gamma\wedge T)(\rho)&=(-1)^{(p+q)(r+s)}T(\gamma\wedge\rho)
\end{align*}
for all $T\in E^{r,s}(U)$, $\gamma\in B^{p,q}(U)$, and $\alpha,\beta,\rho\in B_c(U)$ of suitable bidegree.
\end{rem}

\begin{ex}
A Radon measure $\mu$ on $U$ induces a formal $\delta$-current $T_\mu\in E^{n,n}(U)$ such that
$T_\mu(f)=\int_Uf\,d\mu$
for all $f\in B^{0,0}_c(U)\subseteq C_c^0(U)$.
\end{ex}

\begin{rem}
Using the integral $\int_U\colon B_c^{n,n}(U)\rightarrow \R$ from
Remark \ref{properties-delta-forms-compactification}~\ref{item:8}, 
we obtain a natural map
\begin{equation} \label{associated delta-current}
[\phantom{a}]\colon B^{p,q}(U)\longrightarrow E^{p,q}(U),\,\,\,
\alpha\longmapsto [\alpha] \mbox{ with } [\alpha](\beta)\coloneqq \int_U\alpha\wedge \beta.
\end{equation}
It follows from 
Proposition \ref{piecewise smooth forms as delta-forms} that 
a piecewise smooth form $\alpha\in B^{p,q,0}(U)$ 
defines a formal $\delta$-current $[\alpha]$ on $U$.
\end{rem}

\subsection{The Theorem of Stokes for $\delta$-forms} \label{subsection: Stokes for delta-forms}
Let $U$ be an open subset of $N_\Sigma$ and let $\Delta$ be a $d$-dimensional polyhedron $\Delta$ in $U$ of sedentarity $\sigma \in \Sigma$. In this subsection, we want to deduce the Theorem of Stokes for $\delta$-forms on $\Delta$. 
Recall that we have shown the Theorem of Stokes in \S \ref{subsection: Stokes for polyhedral currents} even for polyhedral currents, but only in case of a polyhedron of maximal dimension $n$. Since we can restrict $\delta$-forms to  the affine space $\A_{\Delta(\sigma)}$, we will reduce to the previous case.

In the following, we will use that Mihatsch defines a pull-back of $\delta$-forms with respect to arbitrary affine maps between euclidean spaces, see \cite[Proposition/Definition 4.2]{mihatsch2021}. In particular, if $\beta \in B(U)$ for an open subset $U$ of $N_\Sigma$, then we define $\beta|_{\A_{\Delta(\sigma)}\cap U}$ by first using the obvious restriction to $U\cap N(\sigma))$ and then by using the pull-back with respect to $\A_{\Delta(\sigma)}\cap U \to U \cap N(\sigma)$.
\begin{lem} \label{support lemma for top-delta-forms}
	Let $\Delta$  be a $d$-dimensional polyhedron in an open subset $U$ of $N_\Sigma$ of sedentarity $\sigma \in \Sigma$ and let 	$\beta\in B_c^{d,d}(U)$. Then  $\beta|_{\A_{\Delta(\sigma)}\cap U}$ has compact support in $\A_{\Delta(\sigma)}\cap U$. The same holds for $\delta$-forms $\beta$  of bidegree $(d-1,d)$ or $(d,d-1)$.
\end{lem}

\begin{proof} 
The proof is  the same as for Lagerberg forms in Lemma \ref{support lemma for top-forms}. The crucial fact is that the $\delta$-form $\beta$ is constant towards the boundary as it was the case with the Lagerberg form $\widetilde\alpha$ before. The argument there shows that for every point $p$ in the closure  of $\A_{\Delta(\sigma)}\cap U$ inside $U$ with $p \not \in N(\sigma)$, there is a neighbourhood $W$ of $p$ in $U$  such that $\beta$ restricts to $0$ on $W \cap \A_{\Delta(\sigma)}$. Note that all such points $p$ form a closed subset $C$ of $U$ and hence $C \cap \supp(\beta)$ can be covered by finitely many $W$ as above, say gathered in a family indexed by $I$. Then 
$$\A_{\Delta(\sigma)} \cap \supp(\beta) \setminus \bigcup_{W \in I} W$$
is a compact subset of $\A_{\Delta(\sigma)}\cap U$ containing the support of $\beta|_{\A_{\Delta(\sigma)}\cap U}$.
\end{proof}

\begin{ex} \label{delta-current for weighted polyhedron}
Let	$[\Delta,\mu]$ be a weighted $d$-dimensional polyhedron of sedentarity $\sigma$ contained in the open subset $U$ of $N_\Sigma$. 
We get an associated formal $\delta$-current $\delta_{[\Delta,\mu]}$ on $U$ as follows.  Let $\beta\in B^{d,d}_c(U)$. 
Since $\beta$ 
is of bidegree $(d,d)$, we have seen in Lemma \ref{support lemma for top-delta-forms} that 
$\beta|_{\A_{\Delta(\sigma)}\cap U}$ has compact support in $\A_{\Delta(\sigma)}\cap U$.
We may see this restriction as a compactly supported polyhedral current on the $d$-dimensional affine space $\A_{\Delta(\sigma)}$ endowed with the weight $\mu$ and hence
\[
\delta_{[\Delta,\mu]}(\beta)=\int_{[\Delta(\sigma),\mu]} \beta
\]
is well-defined by using Definition \ref{define-boundary-delta-integrals}. 
More generally, for any $\alpha \in B^{p,q}(\Delta)$, we get a formal $\delta$-current $\alpha \wedge  \delta_{[\Delta,\mu]}\in E^{n-d+p,n-d+q}(U)$ defined by 
\begin{displaymath}
\alpha \wedge
\delta_{[\Delta,\mu]}(\beta)=\int_{[\Delta(\sigma),\mu]} \alpha
\wedge \beta,
\quad \quad (\beta \in A_c^{d-p,d-q}(U)).
\end{displaymath}
 In particular, for a classical tropical cycle $C$ on $N_\Sigma$ of sedentarity $\sigma$ (see Definition \ref{definition of a classical tropical cycle}), we get a formal $\delta$-current $\delta_{C}$ on $N_\Sigma$ by using the canonical weight on $C$ induced by the lattice structure $N(\sigma)$.

Similarly, we may use Definition \ref{define-boundary-delta-integrals} to define boundary integrals $\int_{\partial'[\Delta,\mu]} \alpha$ for $\alpha \in B_c^{d-1,d}(\Delta)$ and $\int_{\partial''[\Delta,\mu]} \beta$ for $\beta \in B_c^{d,d-1}(\Delta)$. 
Note that the above integrals and boundary integrals depend only on the restriction of the $\delta$-forms to $\A_{\Delta(\sigma)}\cap U$.
\end{ex}

\begin{thm}\label{theorem: Stokes for delta-forms}
	Given $\beta \in B_c^{d-1,d}(U)$ and $\gamma \in B_c^{d,d-1}(U)$ for the $d$-dimensional weighted polyhedron $[\Delta, \mu]$ contained in the open subset $U$ of $N_\Sigma$, we have
	\begin{equation}\label{polyhedral-stokes-formula-delta-forms}
		\int_{[\Delta,\mu]}  d'\beta=\int_{\partial'[\Delta,\mu]}\beta,\,\,\,
		\int_{[\Delta,\mu]} d''\gamma=\int_{\partial''[\Delta, \mu]}\gamma.
	\end{equation}
\end{thm}

\begin{proof} This follows from the Theorem of Stokes for polyhedral currents in Theorem \ref{polyhedral-stokes-formula-proposition} using that we have defined the integrals for $\delta$-forms by reducing to polyhedral currents on $[\A_{\Delta(\sigma)},\mu]$. \end{proof}

\begin{cor} \label{derivatives of delta-forms and delta-currents}
	Let $U$ be an open subset of $N_\Sigma$. The map $B^{p,q}(U)\to E^{p,q}(U)$ from \ref{polyhedral-stokes-formula-proposition}, which assigns to $\alpha \in B^{p,q}(U)$ its associated formal $\delta$-current $[\alpha]$, 
    is an injective linear map compatible with the differentials $d',d''$.
\end{cor}

\begin{proof} 
	The $\delta$-form $\alpha$ is a polyhedral current and hence completely determined as a linear functional on $A_c^{n-p,n-q}(U)$. Since we also obtain this linear functional by restricting the formal $\delta$-current $[\alpha]$ from $B_c^{n-p,n-q}(U)$ to $A_c^{n-p,n-q}(U)$, we get injectivity.  For any $\beta \in B_c^{n-p-1,n-q}(U)$, Stokes formula in Theorem \ref{theorem: Stokes for delta-forms} shows
	$$(d'[\alpha])(\beta)
	= (-1)^{p+q+1}\int_{N_\Sigma} \alpha \wedge d'\beta
	= \int_{N_\Sigma} d'\alpha \wedge\beta - \int_{N_\Sigma} d'(\alpha \wedge \beta)
	=[d'\alpha](\beta)$$
as the boundary integral for the polyhedron $N_\Sigma$ is zero (there are no boundary faces in $N_\R$). A similar identity holds for $d''[\alpha]$. This proves compatibility with the differentials $d',d''$.
\end{proof}

In the next section, we deal with classical tropical cycles of arbitrary sedentarity, see Definition \ref{definition of a classical tropical cycle}.

\begin{cor}\label{delta-c-is-closed}
Given a classical tropical cycle $C$ on $N_\Sigma$ of sedentarity $\sigma$, the formal $\delta$-current $\delta_{C}$ on $N_\Sigma$ from Example \ref{delta-current for weighted polyhedron} is closed under $d'$ and $d''$.
\end{cor}

\begin{proof}
Replacing $N_\Sigma$ by the tropical toric variety $\overline{N(\sigma)}$, we may assume that $C$ has sedentarity $(0)$. 
Then we have seen in Example \ref{example_classical_tropical_cycle} that the polyhedral current $\delta_C$ is $d'$- and $d''$-closed. 
We proceed similarly as in that proof. We first note that the restriction $C(0)$ of $C$ to $N_\R$ is a $\delta$-form and hence $\delta_{C(0)}$ is $d'$- and $d''$-closed as a formal $\delta$-current on $N_\R$ by Corollary \ref{derivatives of delta-forms and delta-currents}. 
Then the same arguments as in Example \ref{example_classical_tropical_cycle}, replacing the smooth form $\eta$ by a $\delta$-form and using Lemma \ref{support lemma for top-delta-forms} instead of Lemma \ref{support lemma for top-forms}, shows the claim.
\end{proof}

\section{Tropical Green functions and the Poincar\'e--Lelong equation} \label{Section: tropical Green functions and the PL-equation}

We fix a fan $\Sigma$ on $N_\R$ for a free abelian group $N$ of finite rank $n$.

\subsection{Tropical Cartier divisors} \label{subsection: tropical Cartier divisors}

In this subsection, we recall the notion of a tropical Cartier divisor on $N_\Sigma$ from tropical geometry. 
We are mainly interested in tropical toric Cartier divisors which we will introduce in an example.

We recall first some notations used in tropical geometry.

\begin{definition} \label{definition of a classical tropical cycle}
A \emph{classical tropical cycle $C$ on $N_\R$ of dimension $d$} is given by a finite $(\Z,\R)$-polyhedral complex $\Ccal$ where the  maximal dimensional polyhedra $\Delta$ have dimension $d$ and are endowed with some constant  weights $m_\Delta \in \Z$ satisfying the balancing condition. 
We identify two tropical cycles if they agree up to subdivision of the underlying polyhedral complexes ignoring the weight $0$ parts. 
An arbitrary tropical cycle on $N_\R$ is the  sum of tropical cycles of dimension $\leq n$.

More generally, a \emph{classical $d$-dimensional tropical cycle $C$ on $N_\Sigma$ of sedentarity $\sigma\in \Sigma$} is induced by a classical $d$-dimensional tropical cycle $C(\sigma)$ on $N(\sigma)$ and is obtained by passing to the closure $\overline \Delta$ in $N_\Sigma$ for every polyhedron $\Delta$ in the polyhedral complex of definition for $C(\sigma)$. 
We view $C$  as a weighted polyhedral complex $(\Ccal,m)$  in $N_\Sigma$ up to subdivision as above. 
The polyhedral complex of definition $\Ccal$ is formed by all faces of the closures $\overline \Delta$ and the weight on a maximal $\overline \Delta$ is $m_\Delta$.

We denote by $\TZ_d(N_\Sigma)=\TZ^{n-d}(N_\Sigma)$ the direct sum of classical $d$-dimensional tropical cycle with varying sedentarity $\sigma \in \Sigma$.
The classical tropical cycles on $N_\Sigma$ are then gathered in $\TZ(N_\Sigma)\coloneqq \bigoplus_d \TZ_d(N_\Sigma)$.

We say that a classical tropical cycle $C$ in $N_\Sigma$  is \emph{constant towards the boundary} if there is a polyhedral complex of definition $\Ccal$ such that every polyhedron $\Delta \in \Ccal$ is constant towards the boundary.
\end{definition}

\begin{definition} \label{piecewise affine function}
A \emph{piecewise affine function}  on $N_\Sigma$ is a function $\phi \colon N_\R \to \R$ such that there is a finite $(\Z,\R)$-polyhedral complex $\Ccal$ on $N_\R$ with $|\Ccal|=N_\R$ and such that for every $\Delta \in \Ccal$ the restriction $\phi|_\Delta$ agrees with $m_\Delta+c_\Delta$ on $\Delta$ for some $m_\Delta \in M$ and $c_\Delta \in \R$. Note that we do not assume that $\phi$ is constant towards the boundary. 
\end{definition}

\begin{definition} \label{definition tropical Cartier divisors}
A \emph{tropical Cartier divisor $D$ on $N_\Sigma$} is given by a finite family $D= (U_i,\phi_i)_{i \in I}$ with $(U_i)_{i \in I}$ an open covering of $N_\Sigma$ and $\phi_i$ a piecewise affine function on $N_\Sigma$ such that for every $i,j \in I$, the function $\phi_i - \phi_j$ on $U_i \cap U_j \cap N_\R$ is affine and  extends to a continuous function $U_i \cap U_j \to \R$.

Two such representation $(U_i,\phi_i)_{i \in I}$ and $(U_j',\phi_j')_{j \in J}$ give the same tropical Cartier divisor if and only if for all $i  \in I$ and all $j \in J$, the function $\phi_i - \phi_j'$ is affine on $U_i \cap U_j'\cap N_\R$ and extends to a continuous function $U_i \cap U_j \to \R$.

It is clear that the tropical Cartier divisors form an abelian group.
\end{definition}

\begin{rem} \label{history for tropical Cartier divisors}
The above definition of a tropical Cartier divisor is from \cite{allermann-rau-2010} in case of $N_\R$. 
Its generalization to $N_\Sigma$ is in Meyer's thesis \cite[Definition 2.35]{meyer-thesis}. They work in a relative setting, i.e.~they define tropical Cartier divisors on a classical tropical cycle $C$. We are here only interested in the absolute case $C=N_\Sigma$ which is less technical. 
\end{rem}

\begin{remi}\label{reminder-pl-function-on-fan}
We recall that a toric Cartier divisor $D$ on the toric variety $X_\Sigma$ can be given by a \emph{piecewise linear function} $\psi\colon |\Sigma| \to \R$ such that for every $\sigma \in \Sigma$ there is $k_\sigma \in M$ with $\psi|_\sigma=k_\sigma|_\sigma$. Let $\rho$ be a ray in $\Sigma$, then the multiplicity of $D$ in the toric subvariety $X(\rho)$ of codimension $1$ is equal to $-k_\rho(\omega_\rho)$ where $\omega_\rho$ is the primitive lattice vector in $\rho$. The choice of sign is for historic reasons, see \cite{fulton-toric-varieties}. 
\end{remi}

We give now a similar construction in tropical geometry.

\begin{ex} \label{tropical toric Cartier divisiors}
Let $\psi$ be a piecewise linear function as {in Reminder \ref{reminder-pl-function-on-fan}}. 
The \emph{tropical toric Cartier divisor} $D(\psi)$ on $N_\Sigma$ associated to $\psi$ is given as follows. 
Let $U_\sigma$ be the tropical toric variety given by the fan generated by $\sigma \in \Sigma$. 
Then $D(\psi)$ is given on the open subset $U_\sigma$ of $N_\Sigma$ by $m_\sigma \coloneqq -k_\sigma \in M$. 
Since the open subsets $U_\sigma$ cover $N_\Sigma$, it is easy to see that we get a well-defined tropical Cartier divisor $D(\psi)$ which is uniquely determined by $\psi$. 

If $\psi$ is not linear, then the tropical Cartier divisor $D(\psi)$ cannot be given by a single piecewise affine function. Tropical Cartier divisors given by a single piecewise affine function are called \emph{principal}. 
\end{ex}

\begin{rem}  \label{relation to Gross}
In the following, a \emph{tropical toric Cartier divisor} $D$ on $N_\Sigma$ is the tropical toric Cartier divisor associated to some piecewise linear function $\psi$ as above. It is clear that the tropical toric Cartier divisors form an abelian group. Tropical toric Cartier divisors are the Cartier divisors considered  in \cite{gross2018} in the more general situation of weakly embedded cone complexes.
\end{rem}

\begin{art} \label{unbounded locus for tropical Cartier divisor}
The restriction of a tropical Cartier divisor $D$ to the dense stratum $N_\R$ is always well-defined, but this is not true for the other strata $N(\sigma)$. 
We define the \emph{unbounded locus $S(D)$} as the complement of the sets of points $x \in N_\Sigma$ such that there is a open neighbourhood $U$ of $x$ where $D$ is given by a bounded piecewise affine function. 
Then the restriction of $D$ to $\overline{N(\sigma)}$ makes sense if and only if $N(\sigma)$ is not contained in $S(D)$. 
Note that $S(D)$ is a closed subset of $N_\Sigma$.
	
If $D$ is a tropical toric Cartier divisor $D(\psi)$, then the unbounded locus is the union of all boundary components $\overline{N(\rho)}$ of codimension $1$ in $N_\Sigma$ where $\rho$ ranges over all rays of $\Sigma$ with $\psi(\omega_\rho) \neq 0$.
\end{art}

\begin{art} \label{pull-back of tropical Cartier divisor}
Let $N'$ be a free abelian group of finite rank $n'$ and let $L\colon N' \to N$ be a homomorphism of groups. 
We consider a fan $\Sigma'$ in $N_\R'$ and an $L$-equivariant map $F\colon N_{\Sigma'}' \to N_\Sigma$ of tropical toric varieties, see \S \ref{subsec: non-arch Green functions}.
The pull-back $F^*(D)$ of a tropical Cartier divisor $D$ on $N_\Sigma$ is well-defined if $F(N_\R')$ is not contained in the unbounded locus $S(D)$. 
If $D$ is given by the representation $(U_i,\phi_i)_{i \in I}$, then $F^*(D)$ is the tropical Cartier divisor represented by $(F^{-1}(U_i),\varphi_i \circ F)_{i \in I}$. 
To see that this is well-defined, {we note that there is a unique $\sigma \in \Sigma$ such that $F(N_\R') \subset N(\sigma)$ and} we use that the restriction of $\varphi_i$ to $N(\sigma)$ is a piecewise affine function as we have seen in \ref{unbounded locus for tropical Cartier divisor}.
\end{art}

\begin{ex} \label{pull-back of tropical toric Cartier divisor}
Let $D=D(\psi)$ be a tropical toric Cartier divisor on $N_\Sigma$ and let $F\colon N_{\Sigma'}' \to N_\Sigma$ be an $L$-equivariant morphism of tropical toric varieties as in \ref{pull-back of tropical Cartier divisor}. 
If the pull-back $F^*(D)$ is well-defined, then $F^*(D)$ is a a tropical toric Cartier divisor  on $N'_{\Sigma'}$ given by $F^*(D)=D(\phi \circ L)$.
\end{ex}

\subsection{Intersection product with tropical Cartier divisors}  \label{subsection: intersection product}

We recall the intersection product of a tropical Cartier divisor with a classical tropical cycle. 
In the classical case $N_\R$, this is due to Allermann and Rau \cite{allermann-rau-2010}. 
In Meyer's thesis \cite{meyer-thesis}, this was generalized to the tropical toric variety $N_\Sigma$. 
Note that our definitions differ from theirs as we use a tropicalization with a different sign. 
We relate then Meyer's construction to the intersection product with tropical toric Cartier divisors on cone complexes given by Gross \cite{gross2018}.

\begin{cons}[The corner locus]\label{construction intersection product}
Let $C$ be a classical  $d$-dimensional tropical cycle on $N_\Sigma$ of sedentarity $(0)$ and let $\phi$ be a piecewise affine function on $N_\Sigma$. 
Meyer defines $\phi \cdot C \in \TZ_{d-1}(N_\Sigma)$ as follows: 
	
There is a $(\Z,\R)$-polyhedral complex of definition $\KC$ in $N_\Sigma$ and weights $m$ such that $(\KC,m)$ represents $C$.
Let $\Dcal$ be a polyhedral complex of definition for the piecewise affine function $\phi$ on $N_\R$.
Then $\Dcal$ is a finite $(\Z,\R)$-polyhedral complex with $|\Dcal|=N_\R$. Replacing $\Ccal$ by a subdivision, we may assume that $\Dcal \cap \Ccal(0)= \Ccal(0)$. 
Then the tropical cycle $\psi \cdot C$ is supported in the faces $\Delta' \in \Ccal_{d-1}$ and equipped with the multiplicities $w(\Delta')$ defined below. 
	
First, we assume that $\Delta'$ has sedentarity $(0)$. 
Then 
\[
w(\Delta') \coloneqq -\sum_{\Delta \succ \Delta'} m_\Delta \phi \left(n_{\Delta,\Delta'}\right) + \phi(\sum_{\Delta \succ \Delta'} m_\Delta n_{\Delta,\Delta'})
\]
where $\Delta$ ranges over all $d$-dimensional faces of $\Ccal$ containing $\Delta'$ and where $n_{\Delta,\Delta'}$ is a primitive normal vector pointing in the direction of $\Delta$. 
	
It remains to consider the case of $\Delta'$ of sedentarity $\tau \neq (0)$. 	
Given a polyhedron $\Delta\in \KC_d$ such that $\Delta'$ is a face of $\Delta$, the intersection $\rho_\Delta\coloneqq \tau\cap \mathrm{rec}(\Delta)\subseteq N_\R$ of $\tau$ with the  recession cone $\mathrm{rec}(\Delta) \in \Sigma'$ is a ray in $\Sigma$ as $\dim\Delta'=\dim(\Delta)-1$ using Lemma \ref{lemm:1}.
Hence $\rho_\Delta$ contains a unique primitive lattice vector $\omega_{\rho_\Delta}$ of $N$ and we define the weight $w(\Delta')$ of $\phi \cdot C$ in $\Delta'$ by the formula
\[
w(\Delta') = \sum_{\Delta \succ \Delta'} m_\Delta [N(\tau)_{\Delta(\tau)}:N_\Delta(\tau)] \cdot \frac{\partial\phi}{\partial \omega_{\rho_\Delta}}
\]
where $\Delta$ runs over all $d$-dimensional faces of $\KC$ containing $\Delta'$ 
{and $\frac{\partial\phi}{\partial \omega_{\rho_\Delta}}$ denotes the derivative of the function $\phi$ in the 
direction $\omega_{\rho_\Delta}$}.

Here, we denote by $[N(\tau)_{\Delta(\tau)}:N_\Delta(\tau)]$ the following lattice index: 
For the $(\Z,\R)$-polyhedron $\Delta$ in $N_\R$, we recall that $N_\Delta \coloneqq \L_\Delta \cap N$ is a sublattice of $N$ and similarly, we define the lattice $N(\tau)_{\Delta(\tau)}$ for the polyhedron $\Delta(\tau)$ in $N(\tau)$. 
Given a lattice $\Lambda$ in $N_\R$, we define $\Lambda(\tau) \coloneqq \im(\Lambda \to N(\tau))$ which is a lattice in $N(\tau)$. 
Using that $\Delta(\tau)$ is the image of $\Delta$ in $N(\tau)$, we deduce that $N_\Delta(\tau)$ is a sublattice of $N(\tau)_{\Delta(\tau)}$ making sense of the above lattice index.

By \cite[Theorem 2.48]{meyer-thesis}, we get a well-defined classical tropical cycle 
\[
\phi \cdot C  \coloneqq \sum_{\Delta'} w(\Delta') \, \Delta' \in \TZ_{d-1}(N_\Sigma)
\] 
where $\Delta'$ ranges over $\Ccal_{d-1}$. Moreover, this intersection product is bilinear.

We write $\div(\phi)\coloneqq \phi \cdot N_\Sigma$ and call it the \emph{tropical Weil divisor associated to $\phi$}.
\end{cons}

\begin{cons}[Intersection product with a tropical Cartier divisor]
\label{intersection product with tropical Cartier divisor}
Let $D$ be a tropical Cartier divisor on $N_\Sigma$ represented by $(U_i,\phi)_{i \in I}$ and let $C \in \TZ_d(N_\Sigma)$ of sedentarity $(0)$. 
Then there is a unique $D \cdot C \in \TZ_{d-1}(N_\Sigma)$ such that  $D \cdot C$ agrees on $U_i$ with $\phi_i \cdot C$ from Construction \ref{construction intersection product} for every $i \in I$.  

More generally, let $C$ be a classical $d$-dimensional tropical cycle of any sedentarity $\sigma \in \Sigma$. 
We may view $C$ as a classical tropical cycle of sedentarity $(0)$ on the tropical toric subvariety $\overline{N(\sigma)}$. 
The \emph{intersection product $D \cdot C$ is well-defined} if $N(\sigma)$ is not contained in the unbounded locus $S(D)$. 
Then the pull-back $i^*(D)$ with respect to the  immersion $i\colon \overline{N(\sigma)} \to N_\Sigma$ is well-defined and we define $D \cdot C \coloneqq i_*(i^*(D) \cdot C)$ using the above definition.
It is clear that the intersection product $D \cdot C$ is bilinear. 
\end{cons}

\begin{rem} \label{Meyer and Gross}
In case $N_\Sigma=N_\R$, the intersection product $D \cdot C$ is the same as in \cite{allermann-rau-2010} \emph{up to sign}. 
The above generalization to arbitrary tropical toric varieties $N_\Sigma$ is done in Meyer's thesis  \cite[Definition 2.47]{meyer-thesis}. 
The sign change is explained by the fact that Meyer uses $-\trop$ for the canonical tropicalization. 
Our signs are governed to get a perfect analogy of the Poincar\'e--Lelong formula in non-archimedean geometry and tropical geometry.
	
If $D$ is a tropical toric Cartier divisor and $C$ is a classical tropical cycle of $N_\Sigma$, then $D \cdot C$ is a classical tropical cycle contained in the boundary $N_\Sigma \setminus N_\R$. 
Moreover, if we also assume that $C$ is a tropical fan subdividing $\Sigma$, then we obtain the intersection product studied by Gross \cite[Subsection 3.3]{gross2018} using \cite[Remark 3.13]{gross2018}.  
	
In fact, Allermann and Rau's intersection product on $N_\R$ and the fan case studied by Gross is enough to recover Meyer's intersection product from \ref{intersection product with tropical Cartier divisor}. 
Indeed, we have seen in Construction \ref{construction intersection product} that the multipicities in the faces $\Delta' \in \Ccal_{d-1}$ are determined by rays of $\Sigma$ and replacing $C$ by the star in $\Delta'$, meaning the affine spaces $\A_\Delta$ with $\Delta \in \Ccal_d$ containing $\Delta'$, we can reduce the definition of the weight $w(\Delta')$ to the fan situation. 
\end{rem}

\begin{prop} \label{symmetry of intersection product with tropical Cartier divisors}
Let $D, D'$ be tropical Cartier divisors on $N_\Sigma$ and let $C$ be a classical $d$-dimensional tropical cycle on $N_\Sigma$. 
Then we have
\[
D \cdot (D' \cdot C)= D' \cdot (D \cdot C)
\]
if all these intersection products are well-defined as tropical cycles on $N_\Sigma$.
\end{prop}

\begin{proof} 
This follows from \cite[Theorem 2.50]{meyer-thesis}.
\end{proof}

\begin{prop} \label{projection formula of intersection product}
Let $L\colon N' \to N$ be a homomorphism of free abelian groups of finite rank, let $F \colon N_{\Sigma'}' \to N_\Sigma$ be an $L$-equivariant morphism of tropical toric varieties, let $D$ be a tropical Cartier divisor on $N_\Sigma$ and let $C'$ be a classical tropical cycle of $N'_{\Sigma'}$. 
We assume that $F(N'_\R)$ is not contained in the unbounded locus $S(D)$. Then we have the projection formula
\[
F_*(F^*(D) \cdot C) = D \cdot F_*(C).
\]
\end{prop}

\begin{proof}
In the euclidean case $N_\Sigma = N_\R$ and $ N_{\Sigma'}'=N_\R'$, this is \cite[Proposition 7.7]{allermann-rau-2010}. The generalization to tropical toric varieties is due to Meyer \cite[Theorem 2.53]{meyer-thesis} noting that the completeness hypotheses are not needed in the proof. The projection formula at the boundary can also be deduced from  \cite[Proposition 3.16]{gross2018}.
\end{proof}

\subsection{Tropical Green functions} \label{subsection: tropical Green functions}

We will introduce tropical Green functions and we will see that there is an associated tropical toric Cartier divisor. 

\begin{definition} \label{definition: tropical Green function}
	A \emph{tropical Green function  on $N_\Sigma$} is a function $g\colon N_\R \to \R$ such that for every  $x \in N_\Sigma$ there exists an open neighbourhood $\Omega$ of $x$ in $N_\Sigma$, $m \in M$ and a continuous function $h \colon \Omega \to \R$ such that $g=m+h|_{N_\R\cap\Omega}$ holds on $N_\R\cap\Omega$. We call $g$ \emph{of smooth (resp.~piecewise smooth) type} if $h$ is smooth (resp.~piecewise smooth).
\end{definition}

\begin{rem} \label{choice of slope for tropical Green function}
	We discuss uniqueness of the \emph{asymptotic slope} $m$ in the above definition. Let $\sigma$ be the sedentarity of $x$. Then the restriction  $m|_\sigma$ is uniquely determined by $x$. Indeed, if $m'$ is another such slope, then $m-m'$ is bounded in a neighbourhood of $x$ which yields that $m|_{\sigma}=m|_{\sigma'}$.
	
	The argument shows that $m|_\sigma$ is the same in a neighbourhood of $x$ in $N(\sigma)$ and as $N(\sigma)$ is connected, we conclude that $m|_\sigma$ depends only on $\sigma$. 
\end{rem}

\begin{prop} \label{associated tropical toric Cartier divisor}
	Let $g$ be a tropical Green function on $N_\Sigma$. Then there is a unique tropical toric Cartier divisor $D$ on $N_\Sigma$ such that for  $\sigma \in \Sigma$ and $D$ given on $U_\sigma$ by $m_\sigma \in M$, we have that $g-m_\sigma$ extends to a continuous function $U_\sigma \to \R$.
\end{prop}

\begin{proof}
	By Remark \ref{choice of slope for tropical Green function}, the asymptotic slopes $m_\sigma$ for $g$ towards $N(\sigma)$ fit together to a piecewise linear function $\phi$ on $\Sigma$. Using $m_\sigma$ on $U_\sigma$, it is clear that we get the toric Cartier divisor $D$ associated to the piecewise linear function $\psi \coloneqq -\phi$ with the required property, see Example \ref{tropical toric Cartier divisiors}. Uniqueness is clear by construction.
\end{proof}

\begin{rem} \label{notation for Green functions}
We call $D$ \emph{the tropical toric Cartier divisor associated to $g$} and \emph{$g$ a Green function for $D$}. The map $g \mapsto D$ is additive.
\end{rem}

\begin{ex} \label{canonical Green function}	
If $\Sigma$ is a complete fan, then every tropical toric Cartier divisor $D$ on $N_\Sigma$ has a canonical Green function given by $g \coloneqq -\psi$ where $\psi$ is the piecewise linear function associated to $D$, see Example \ref{tropical toric Cartier divisiors}. 
Note that $g$ is a Green function  of smooth type, but $g|_{N_\R}$ is not a smooth function.
\end{ex}

\begin{ex} \label{piecewise affine functions as Green functions}
Let $\Sigma$ be complete and let $\phi$ be a piecewise affine function on $N_\Sigma$. 
We assume that $\phi$ is given by a polyhedral complex of definition $\Dcal$ such that the recession cone of every face of $\Dcal$ is subdivided by cones in $\Sigma$. 
By Proposition \ref{recession-cone-boundary-condition}, this just means that the polyhedra in $\Dcal$ are constant towards the boundary. 
Recall the definition of the \emph{recession function $\rec(\phi)$ of $\phi$} from \cite[Definition 2.6.4]{bps-asterisque}.
The recession function $\rec(\phi)$ 
is piecewise linear with respect to $\Sigma$ and 
$\phi- \rec(\phi)$ is a bounded piecewise affine function. 
We conclude that $\phi$ is a Green function for the tropical toric Cartier divisor associated to $\psi \coloneqq -\rec(\psi)$.  
\end{ex}

\begin{rem} \label{unbounded locus for tropical Green function}
Similarly as in Remark \ref{unbounded locus for tropical Cartier divisor}, we define the unbounded locus $S(g)$ of a Green function $g$ 
as the complement of the set of points which have a neighbourhood in $N_\Sigma$ where $g$ is bounded. Clearly, we have $S(g)=S(\psi)$ where $D=D(\psi)$ is the  Cartier divisor associated to $g$. 
\end{rem}

\begin{art} \label{pull-back of tropical Green functions}
Let $L\colon N' \to N$ be a homomorphism of free abelian groups of finite rank. 
If $F\colon N'_{\Sigma'} \to N_\Sigma$ is an $L$-equivariant morphism of tropical toric varieties and if $F(N')$ is not contained in $S(\psi)$, then the \emph{pull-back} $F^*(g) \coloneqq g \circ F$ is a well-defined Green function for the tropical toric Cartier divisor $F^*(D)$ from \ref{pull-back of tropical Cartier divisor}. 
If $g$ is of smooth (resp.~piecewise smooth) type, then the same holds for $F^*(g)$.
\end{art}

\subsection{The tropical Poincar\'e--Lelong equation on $N_\R$}  \label{subsection: The euclidean tropical PL-equation}

The tropical Poincar\'e--Lelong equation on $N_\R$ computes the $d'd''$-residue of a piecewise smooth function. It is based on results of Mihatsch \cite{mihatsch2021} and generalizes \cite[Theorem 0.1]{gubler-kuenne2017}.

\begin{art} \label{residue for piecewise smooth form}
Let $\gamma$ be a piecewise smooth form on $N_\R$ of bidegree $(p,q)$ and let $C=(\Ccal,m)$ be a classical $d$-dimensional tropical cycle on $N_\R$.
We want to compute the $d'$-residue of the polyhedral current $\gamma \wedge \delta_C$. 
As this can be done locally, we may assume that there is a finite polyhedral complex of definition $\Dcal$ for $\gamma$. 
	
Applying  (3.3) in the proof of \cite[Theorem 3.3]{mihatsch2021} to the polyhedral current $T \coloneqq \gamma \wedge\delta_C$ and using the balancing condition for $C$, we get the following explicit expression for the $d'$-residue of $\gamma \wedge \delta_C$:
\[
d'(\gamma \wedge \delta_C)-\dpa \gamma \wedge \delta_C = \sum_{\tau \in (\Ccal \cap \Dcal)_{d-1}} \sum_{\sigma \in (\Ccal \cap \Dcal)_d, \, \sigma \succ \tau} m_\sigma \cdot (\gamma_\sigma, n_{\sigma,\tau}'')\wedge[\tau,\mu_\tau]
\]
where we endow the $(d-1)$-dimensional faces $\tau$ of $\Ccal \cap \Dcal$ with any weights $\mu_\tau$ as in \ref{definition weighted polyhedra} and where $n_{\sigma,\tau}''$ is the normal vector defined in \ref{orthogonal lattice vectors} which was used for defining boundary integrals.
\end{art}

\begin{definition} \label{corner locus}
The \emph{corner locus} $\delta_{h \cdot C}$ for a piecewise smooth function $h$ on $N_\R$ is a polyhedral current on $N_\R$ locally  defined by 
$$\delta_{h \cdot C} \coloneqq {-}\sum_{\tau \in (\Ccal \cap \Dcal)_{d-1}} \sum_{\sigma \in (\Ccal \cap \Dcal)_d, \, \sigma \succ \tau} m_\sigma \cdot (\gamma_\sigma, n_{\sigma,\tau}'')\wedge[\tau,\mu_\tau]$$
in the notation above applied to $\gamma \coloneqq \dpb h$ .
\end{definition}

\begin{rem} \label{remarks for the corner locus}
The name corner locus is due to the fact that $\delta_{h \cdot C}$ is supported in the locus where the piecewise smooth function $h|_C$ is singular. 
If we can choose $\Ccal$ as a $(\Z,\R)$-polyhedral complex, then we get the corner locus defined in \cite[Definition 1.10]{gubler-kuenne2017} \emph{up to sign}. 
In particular, if $h$ is piecewise linear, then we get the intersection product $h \cdot C$ from Allermann and Rau \cite{allermann-rau-2010} \emph{up to sign}.
\end{rem}

We have the \emph{tropical Poincar\'e--Lelong formula for piecewise smooth functions.}

\begin{thm} \label{tropical PL-formula for piecewise smooth functions}
For a piecewise smooth function $h$ on $N_\R$ and a classical tropical cycle $C$ of $N_\R$, we have the identity
\[
d'd''(h \wedge \delta_C)= (\dpa \dpb h) \wedge \delta_C- \delta_{h \cdot C}
\]
where $(\dpa \dpb h) \wedge \delta_C$ is a product of $\delta$-forms on $N_\R$.
\end{thm}

Note that $h \wedge \delta_C$ is a $\delta$-form on $N_\R$, so the identity can be seen as an identity of (polyhedral) currents in $D(N_\R)$. 
On the other hand, it is also an identity of formal $\delta$-currents as the derivatives of the formal $\delta$-current associated to a $\delta$-form agrees with the formal $\delta$-current associated to the derivatives of the $\delta$-form by Corollary \ref{derivatives of delta-forms and delta-currents}.

\begin{proof}
There is a similar formula for the $d''$-residue  as in \ref{residue for piecewise smooth form}. Applying this to $\gamma \coloneqq h$ yields
\[
d''(h \wedge \delta_C)=  \dpb h \wedge \delta_C.
\]
Using \ref{residue for piecewise smooth form} for $\gamma \coloneqq \dpb h$ and Definition \ref{corner locus}, we have
\[
d'(\dpb h \wedge \delta_C) = -\delta_{h \cdot C}
+ (\dpa \dpb h) \wedge \delta_C.
\]
Both identities together prove the claim.
\end{proof}

\subsection{The tropical Poincar\'e--Lelong formula for piecewise affine functions}

\label{subsection: tropical PL-formula for piecewise affine functions}

We have introduced piecewise affine functions in \S \ref{subsection: tropical Cartier divisors}. In tropical geometry, they are analogues of rational functions. In this subsection, we will prove the tropical Poincar\'e--Lelong equation for them. Again, we do this in a semi-relative setting involving a classical tropical cycle $C$ on $N_\Sigma$ of dimension $d$ and of sedentarity $(0)$.

\begin{cons} \label{wedge product of function with tropical cycle}
Let $\phi$ be any piecewise smooth function on $N_\R$. 
Then we define a formal $\delta$-current $\phi \wedge \delta_C \in E^{n-d,n-d}(N_\Sigma)$ as follows. 
Let $\omega \in B_c^{d,d}(N_\Sigma)$, then $\phi\omega \wedge \delta_C \in B_c^{n,n}(N_\R)=B_c^{n,n}(N_\Sigma)$. 
Note that the fact that $\omega$ is constant towards the boundary and has compact support yields that $\omega \wedge \delta_C$ has compact support in $N_\R$ by Lemma \ref{support lemma for top-delta-forms}.
Then we define
\[
(\phi \wedge \delta_C)(\omega) \coloneqq \int_C \phi \omega \coloneqq \int_{N_\R} \phi \omega \wedge \delta_C.
\]
\end{cons}

\begin{prop} \label{projection formula for above product}
Let $L\colon N' \to N$ be a homomorphism of free abelian groups of finite rank, let $F\colon  N'_{\Sigma'} \to N_\Sigma$ be an $L$-equivariant morphism of tropical toric varieties, let $\phi$ be a piecewise smooth function on $N_\R$ and let $C'$ be a classical tropical cycle on $N'_{\Sigma'}$ of sedentarity $(0)$. 
Then we have the projection formula
\[
F_*(F^*(\phi)\wedge \delta_{C'})= \phi \wedge \delta_{F_*(C')}.
\]
\end{prop}

\begin{proof}
This follows from the projection formula in \cite[Proposition 4.2]{mihatsch2021}.
\end{proof}

 Using the intersection product with piecewise affine functions from Construction \ref{intersection product with tropical Cartier divisor}, we have the following \emph{tropical Poincar\'e--Lelong formula}.
\begin{thm} \label{PL-formula for piecewise affine}
Let $\phi$ be a piecewise affine function and let $C \in \TZ_d(N_\Sigma)$ of sedentarity $(0)$. Then we have the  identity $d'd''(\phi \wedge \delta_C)= -\delta_{\phi \cdot C}$ of formal $\delta$-currents on $N_\Sigma$. 
\end{thm}

\begin{proof}
This can be checked locally at $u \in N_\Sigma$. 
By toric completion, we know that $\Sigma$ is a subfan of a complete fan. 
Then $N_\Sigma$ is an open tropical toric subvariety of the corresponding complete tropical toric subvariety and hence we may assume that $\Sigma$ is complete. 
By subdivision and toric resolution of singularities, there is a regular simplicial fan $\Sigma'$ which subdivides $\Sigma$ and the recession fan of a polyhedral complex of definition for $\phi$. Moreover, we may assume that the recession cones of the faces of $C$ are subdivided by cones in $\Sigma'$. 
Then there is a dominant equivariant morphism $N_{\Sigma'} \to N_\Sigma$. 
Note that $C$ is a classical tropical cycle of $N_{\Sigma'}$ which is constant towards the boundary by Proposition \ref{recession-cone-boundary-condition}. 
Using the projection formulas from Proposition \ref{projection formula of intersection product} and from Proposition \ref{projection formula for above product}, we may assume that $\Sigma = \Sigma'$.

Using the above, it is enough to prove the claim in a neighbourhood of $u$ 
and for a  complete regular simplicial fan $\Sigma$. Moreover, we may assume that the recession cones of the faces of $C$ are subdivided by cones in $\Sigma$ which means that $C$ is constant towards the boundary by Proposition \ref{recession-cone-boundary-condition}. 
By Example \ref{example_classical_tropical_cycle}, we know that $\delta_C$ is a $d'$- and $d''$-closed $\delta$-form on $N_\Sigma$. 
We have to check the identity 
\begin{equation} \label{residue formula check}
\int_{C}     \phi d'd''\omega= -\int_{\phi \cdot C}\omega.
\end{equation}
for $\omega \in B_c^{d-1,d-1}(N_\Sigma)$ with compact support in a
small open neighbourhood of $u$. 
	
As $\sigma$ is a regular simplex, there is a basis $x_1,\dots,x_n$ of
$M$ such that $\sigma=\R_{\geq 0}^r \times \{0_{n-r}\}$ in
$N_\R= \R^n$ for some $r \in \{0,\dots,n\}$.  Let
$U_\sigma=\R_\infty^r\times \R^{n-r}$ be the standard open
neighbourhood of $u$ in $N_\Sigma$ and
$u=(\infty_r,u_{r+1},\dots,u_n)$ in these coordinates.  As we may also
assume that the recession fan of a polyhedral complex of definition
$\Dcal$ for $\phi$ is subdivided by $\Sigma$, we deduce from
Proposition \ref{recession-cone-boundary-condition} that all polyhedra
in $\Dcal$ are constant towards the boundary.  Using these very
explicit assumptions, we deduce easily that $\phi=m+\phi'$ in a
neighbourhood of $u$ for some $m \in M$ and a bounded piecewise affine
function $\phi'$.
	
We prove now the claim for $\phi'$.  As the polyhedra in $\Dcal$ are
constant towards the boundary, the piecewise affine function $\phi'$
is a piecewise smooth function on $N_\Sigma$ and in particular
constant towards the boundary.  It follows that the
$\phi' d'd''\omega \wedge \delta_C \in B^{n,n}(N_\Sigma)$ has compact
support in $N_\R$, see Lemma \ref{support lemma for top-delta-forms}.
In this case, $\phi\cdot C$ is a classical tropical cycle of
sedentarity $(0)$ and $\omega \wedge \delta_{\phi \cdot C}$ is also a
$\delta$-form with compact support in $N_\R$, hence we deduce the
claim from the tropical Poincar\'e--Lelong formula for piecewise
smooth functions on $N_\R$ given in Theorem \ref{tropical PL-formula
  for piecewise smooth functions}.
	
We may assume now that $\phi=m$. Note that the above argument also
proves the claim in a neighbourhood of $u$ if $u \in N_\R$, so we now
assume $u \in N_\Sigma \setminus N_\R$.  We claim that by replacing
$\omega$ by $\omega \wedge \delta_C$, we may assume that $C=N_\Sigma$.
Indeed, we have
\[
\int_{C}     \phi d'd''\omega= \int_{N_\Sigma} \phi d'd''\omega \wedge \delta_C = \int_{N_\Sigma} \phi d'd''(\omega \wedge \delta_C)
\]
using that $\delta_C$ is a $d'$- and $d''$-closed $\delta$-form.
Using that $\phi=m$, we have
$\phi \cdot C= \sum_\rho m(n_\rho) \overline{C(\rho)}$ and
$\div(\phi)=\sum_\rho m(n_\rho) \overline{N(\rho)}$ where $\rho$
ranges over the rays in $\Sigma$ and where $n_\rho$ is the primitive
lattice vector in the ray $\rho$. Then
\[
\int_{\phi \cdot C}\omega = \sum_\rho m(n_\rho)\int_{N(\rho)} \omega(\rho) \wedge \delta_{C(\rho)}= \int_{\div(\phi)} \omega \wedge \delta_C
\]
showing the reduction step to $C=N_\Sigma$.  We denote the
intersection product $\phi \cdot C$ in the case $C=N_\Sigma$ by
$\div(\phi)=\div_0(\phi) + \div_\infty(\phi)$ with the decomposition
of the tropical Weil divisor into finite and boundary part.
	
There is a unique $\sigma \in \Sigma$ such that $u \in N(\sigma)$.
Since $u \not\in N_\R$, we have
$\sigma=\R_{\geq 0}^r \times \{0_{n-r}\}$ in $N_\R= \R^n$ for some
$r \in \{1,\dots,n\}$. Then $U_\sigma =\R_\infty^r \times \R^{n-r}$ is
an open neighbourhood of $u=(\infty_{r},u_{r+1}, \dots, u_n)$.  If
$r \geq 2$, then the boundary $N(\sigma)$ at $u$ has dimension $n-r$
and $\omega$ has bidegree $(n-1,n-1)$, so $\omega$ has to be zero in a
neighbourhood of $u$ using that $\omega$ is constant towards the
boundary. Now a partition of unity argument reduces the proof of
\eqref{residue formula check} to the case $r=1$.

So we assume from now on $r=1$. For $R >0$, we will use the cut $P_R \coloneqq (-\infty,R] \times \R^{n-1}$ which has the boundary $\partial P_R= \{x_1=R\} \times \R^{n-1}$.

	Recall that for a $\delta$-form $\beta$, we have $d'\beta = \dpa \beta - \partial' \beta$ and $d''\beta = \dpb \beta - \partial'' \beta$ for the residues defined in \S\ref{residue formula}. By the Leibniz formula, we have
	$$\int_{P_R} \phi d'd''\omega = \int_{P_R}d'(\phi d'' \omega)-\int_{P_R} d'\phi \wedge d''\omega.$$
	Using the formula of Stokes for polyhedral currents in Theorem \ref{polyhedral-stokes-formula-proposition}, 
	we have 
	$$\int_{P_R} d'(\phi d''\omega)= \int_{\partial'P_R}\phi d'' \omega.$$
	The only contribution for the boundary integral comes from the boundary $\partial P_R$. Since $\omega$ is constant towards the boundary, the boundary  integral on the right will be $0$ for $R$ sufficiently large and we get 
	$$\int_{P_R} \phi d'd''\omega = -\int_{P_R} d'\phi \wedge d''\omega.$$
	A similar argument with $d''$ instead of $d'$ shows
	$$-\int_{P_R} d'\phi \wedge d''\omega= \int_{P_R} d''(d'\phi \wedge \omega)+\int_{P_R} d'd''\phi \wedge \omega$$
	and
	$$\int_{P_R} d''(d'\phi \wedge \omega)=\int_{\partial''P_R} d'\phi \wedge \omega.$$
The only non-zero contribution to the boundary integral can be in the codimension $1$ face $D_R \coloneqq \{x_1=R\}$ of $P_R$, so we have
\[
\int_{\partial''P_R} d'\phi \wedge \omega = \int_{D_R} (d'\phi \wedge \omega,n')
\]
for the inwards pointing normal vector $n=-(1,0,\dots,0)$ by definition of the boundary integral. 
We extend here the contraction of forms to delta-forms as we did in the definition of boundary integrals in Definition \ref{define-boundary-delta-integrals} by contracting just the forms involved in the corresponding polyhedral current. 
Using the properties of the contraction in $n'$, we get
\[
\int_{D_R} (d'\phi \wedge \omega,n')= -\int_{D_R} \frac{\phi}{\partial x_1}\omega= -\int_{\div_\infty(\phi)} \omega
\]
for $R$ sufficiently large.
We used in the last step that $\omega$ is constant towards the boundary and the definition of $\phi \cdot N_\Sigma = \div(\phi)$ given in Construction \ref{construction intersection product}. 
We conclude from the displayed identities above that
\[
\int_{P_R} \phi d'd''\omega = -\int_{\div_\infty(\phi)} \omega + \int_{P_R} d'd''\phi \wedge \omega
\]
for $R$ sufficiently large. 
Since $d'd''\phi = -\delta_{\div_0(\phi)}$ on $N_\R$ by the tropical
Poincar\'e--Lelong formula  in Theorem \ref{tropical PL-formula for
  piecewise smooth functions}, we deduce that
\[
\int_{P_R} \phi d'd''\omega = -\int_{\div_\infty(\phi)} \omega - \int_{P_R}  \delta_{\div_0(\phi)} \wedge \omega
\]
for $R$ sufficiently large. Letting $R \to \infty$ and using
$\div(\phi)=\div_0(\phi) + \div_\infty(\phi)$, we deduce
\eqref{residue formula check} and hence the claim. 
\end{proof}

\subsection{Tropical Poincar\'e--Lelong formula for Green functions}
\label{subsection: tropical PL-formula for Green functions}

The tropical Poincar\'e--Le\-long equation computes the residue of a Green function $g$  of piecewise smooth type on $N_\Sigma$. 
It has two parts, the first summand is the first Chern $\delta$-form of $g$, which we will introduce in the course of this subsection, and the second part is the current of integration over the associated tropical Cartier divisor. 
For latter purposes, all of this is stated more generally in a semi-relative setting for $g \wedge \delta_C$ where $C$ is a fixed classical tropical cycle on $N_\Sigma$ of dimension $d$.

\begin{definition} \label{first Chern delta form}
Let $g$ be a Green function of piecewise smooth type for the tropical toric Cartier divisor $D$. We define the \emph{first Chern $\delta$-form} $c_1(D,g) \in B^{1,1}(N_\Sigma)$ as follows: 
Suppose that $x \in N(\sigma)$ and that $D$ is given in a neighbourhood $\Omega$ of $x$ in $N_\Sigma$ by $m_\sigma \in M$, then there is a piecewise smooth function $h$ on $\Omega$ such that $g=m_\sigma+h$ on $\Omega \cap N_\R$. 
Then we set 
$$c_1(D,g)|_\Omega \coloneqq d'd'' h.$$
Clearly, this does not depend on the choice of $\Omega$ and $m_\sigma$ and hence these $\delta$-forms glue to a global $\delta$-form $c_1(D,g) \in B^{1,1}(N_\Sigma)$. 
\end{definition}

\begin{rem} \label{properties of the first Chern delta form}
The first Chern $\delta$-form $c_1(D,g)$ is linear in $(D,g)$. 
Let $L\colon N' \to N$ be a homomorphism of free abelian groups of finite rank and $F \colon N'_{\Sigma'} \to N_\Sigma$ an $L$-equivariant morphism of tropical toric varieties.
If $F(N'_\R)$ is not contained in the unbounded locus $S(D)$, then we have
\[
F^*c_1(D,g) = c_1(F^*D,F^*g).
\]
Note that the latter condition is needed to make $F^*D$ well-defined. 
Note that we can change locally $D$ by $m \in M$ such that $x$ is no longer in the unbounded locus. 
This does not change the first Chern $\delta$-forms.
\end{rem}

We are now ready to state and prove the \emph{tropical Poincar\'e--Lelong equation for Green functions of piecewise smooth type}.

\begin{thm} \label{tropical PL equation for Green functions}
Let $g$ be Green function of piecewise smooth type for the tropical toric Cartier divisor $D$ and let $C$ be a classical tropical cycle of sedentarity $(0)$  on $N_\Sigma$. 
Then we have the identity 
\[
d'd''(g \wedge \delta_C)= c_1(D,g)\wedge \delta_C - \delta_{D \cdot C}
\]
of formal $\delta$-currents on $N_\Sigma$.
\end{thm}

\begin{proof}
We can check this identity of formal $\delta$-currents locally at $x \in N_\Sigma$, say of sedentarity $\sigma \in \Sigma$. 
We have seen in Definition \ref{first Chern delta form} that there is an open neighbourhood $\Omega$ of $x$ in $N_\Sigma$ such that $D$ is given on $\Omega$ by $m_\sigma \in M$ and there is a piecewise smooth function $h\colon U_\sigma \to \R$ such that $g=m_\sigma + h$ on $\Omega$. 
By definition, we have $d'd''h=c_1(D,g)$ on $\Omega$. 
Since $\delta_C$ is $d'$- and $d''$-closed {by Corollary \ref{delta-c-is-closed}}, we conclude that 
\[
d'd''(h \wedge \delta_C) = c_1(D,g)\wedge \delta_C
\]
on $\Omega$. 
On the other hand, Theorem \ref{PL-formula for piecewise affine} shows that 
\[
d'd''(m_\sigma \wedge \delta_C)= -\delta_{D \cdot C}
\]
on $\Omega$. 
Using $g=m_\sigma + h$ on $\Omega$, we deduce the claim.
\end{proof}

\section{Arakelov theory on tropical toric varieties} \label{sec:a tropical star-product}

In this section, we denote by $\Sigma$ a rational polyhedral fan in
$N_\R$ for a free $\Z$-module $N$ of finite rank $n$ with dual
$M=\mathrm{Hom}_\Z(N,\Z)$.  
We will introduce a star-product with Green functions of piecewise
smooth type on $N_\Sigma$ similarly as the one on the complex toric
variety $X_{\Sigma}(\C)$ in Arakelov theory.

\subsection{Definition of the tropical $*$-product} \label{subsection: definition of tropical star-product}

Recall that $\TZ^p(N_\Sigma)$ denotes the space of classical tropical
cycles of codimension $p$ on $N_\Sigma$. 
Recall also that classical tropical cycles have constant weights of
any sedentarity and that there is no condition of being constant
towards the boundary.  
We denote by $E(N_\Sigma)$ the space of formal $\delta$-currents on $N_\Sigma$.

\begin{definition}\label{define-green-delta-forms}
Given a pair $(C,T)\in \TZ^p(N_\Sigma) \times E^{p-1,p-1}(N_\Sigma)$, we define
\[
\omega(C,T)\coloneqq d'd''T+\delta_C\in E^{p,p}(N_\Sigma).
\]
If the formal $\delta$-current $\omega(C,T)$ is a $\delta$-form, then we call \emph{$T$ a Green $\delta$-current for $C$}.
\end{definition}

\begin{rem}
Recall that, by Corollary \ref{derivatives of delta-forms and delta-currents}, the map from $\delta $-forms to $\delta$-currents is injective. 
Therefore, if the $\delta$-current $\omega(C,T)$ can be represented by a  $\delta$-form, this $\delta $-form is unique and we still denote it as $\omega(C,T)$. 
\end{rem}

\begin{rem} \label{closedness of tropical cycles}
We have seen in Corollary \ref{delta-c-is-closed} that for $C\in \TZ^p(N_\Sigma)$ the $\delta$-current $\delta_{C}$ is $d'$- and $d''$-closed.
Hence the same is true for the $\delta$-current $\omega(C,T)$ for any $T\in E^{p-1,p-1}(N_\Sigma)$.
If $T$ a Green $\delta$-current for $C$, then we see in particular that the $\delta$-form $\omega(C,T)$ is  $d'$- and $d''$-closed.
\end{rem}

\begin{ex}\label{green-delta-divisor}
If $D$ is a tropical {toric} Cartier divisor on $N_\Sigma$ and $g$ is a Green function of piecewise smooth type for $D$, then $\omega(D,g)=c_1(D,g)$ is a $\delta$-form by the tropical Poincar\'e--Lelong formula in Theorem  \ref{tropical PL equation for Green functions} and $g$ is a Green $\delta$-current for the \emph{tropical Weil divisor $\cyc(D)\coloneqq D\cdot N_\Sigma$ associated to $D$.} 
\end{ex}

\begin{definition} \label{proper intersection}
Let $P$ be a polyhedral subset of $N_\Sigma$ given as a finite union of polyhedra of dimension $d$. 
Let $D_0, \dots, D_q$ be tropical toric Cartier divisors on $N_\Sigma$. 
Recall from \ref{unbounded locus for tropical Cartier divisor} that $S(D_i)$ denotes the unbounded locus of $D_i$. 
We say that \emph{$D_0, \dots, D_q$ intersect $P$ properly} if for any $I \subset \{0,\dots,q\}$, we have
\[
\dim \left(\bigcap_{i \in I} S(D_{i})\right) \cap P\leq d- |I|.
\]
If $C$ is a tropical cycle of $N_\Sigma$ of dimension $d$, then we say that \emph{$D_0, \dots, D_q$ intersect $C$ properly}, if $D_0, \dots, D_q$ intersect the polyhedral set $|C|$ properly.
\end{definition}

\begin{rem} \label{single proper intersection and wedge product}
Let $D$ be a tropical toric Cartier divisor intersecting {a classical tropical cycle} $C \in \TZ_d(N_\Sigma)$ properly. 
This means that no maximal face of $C$ is contained in {the unbounded locus} $S(D)$. 
For a Green function $g$ for $D$ of piecewise smooth type, we define {the product} $g \wedge \delta_C \in E_{d,d}(N_\Sigma)$ as follows:

By linearity, we may assume that $C$ has sedentarity $\sigma$. 
Using that $D$ intersects $C$ properly, we deduce that $N(\sigma)$ is not contained in $S(D)$. 
By \ref{pull-back of tropical Green functions}, the restriction
$D_\sigma \coloneqq D|_{\overline{N(\sigma)}}$ is a well-defined
tropical toric Cartier divisor and $g_\sigma \coloneqq
g|_{\overline{N(\sigma)}}$ is a Green function for $D_\sigma$ of
piecewise smooth type on the tropical toric variety
$\overline{N(\sigma)}$.  
In particular, the function $g_\sigma$ is piecewise smooth on $N(\sigma)$ and $g_\sigma \wedge \delta_C \in E_{d,d}(\overline{N(\sigma)})$ is defined in Construction \ref{wedge product of function with tropical cycle}. 
Finally, we define $g \wedge \delta_C \in E_{d,d}(N_\Sigma)$ as the push-forward with respect to the map $\overline{N(\sigma)} \to N_\Sigma$.
\end{rem}

Observe that we do not assume in the next definition that $C$ is of
sedentarity $(0)$.
Furthermore our construction of the $*$-products
requires the use of formal $\delta$-currents, not just 
currents, as we multiply them by $\delta$-forms.

\begin{definition}\label{star product}
For a piecewise smooth Green function $g$ for $D$ on $N_\Sigma$ and a pair $(C,T)\in \TZ^p(N_\Sigma) \times E^{p-1,p-1}(N_\Sigma)$ such that $D$ intersects $C$ properly, let
\[
(D,g) * (C,T) \coloneqq (D \cdot C, g \wedge \delta_C+\omega(D,g)\wedge T)\in \TZ^{p+1}(N_\Sigma) \times E^{p,p}(N_\Sigma).
\]
Here the term $g\wedge\delta_C$ was defined above and the term $\omega(D,g)\wedge T$ is well defined because $\omega(D,g)$ is a $\delta$-form by Example \ref{green-delta-divisor}. 
\end{definition}

\begin{prop} \label{multiplicativity of omega}
In the situation of Definition \ref{star product}, we have
\[
\omega\bigl((D,g) * (C,T)\bigr) =\omega(D,g)\wedge \omega(C,T).
\]
\end{prop}

\begin{proof}
Using the definitions, we have
\[
\omega((D,g) * (C,T))=d'd''(g \wedge \delta_C+\omega(D,g)\wedge T)+\delta_{D\cdot C}.
\]
Since $\omega(D,g)$ is $d'$- and $d''$-closed by Remark \ref{closedness of tropical cycles}, the Leibniz rule yields
\[
\omega((D,g) * (C,T))=d'd''(g \wedge \delta_C)+\omega(D,g)\wedge d'd''T+\delta_{D\cdot C}.
\]
Using the
{tropical Poincar\'e--Lelong equation for Green functions of piecewise
  smooth type in Theorem \ref{tropical PL equation for Green
    functions}}, we deduce
\[
\omega((D,g) * (C,T))= \omega(D,g)\wedge \delta_C- \delta_{D\cdot C}+\omega(D,g)\wedge d'd''T+\delta_{D\cdot C}.
\]
Finally, using $d'd''T=\omega(C,T)-\delta_C$ yields the claim.
\end{proof}

\subsection{Properties of the tropical $*$-product}  \label{subsection: properties of the tropical star-product} 
We will prove the symmetry of the action induced by the above tropical $*$-product and functoriality.

\begin{prop} \label{symmetry of star-product}
Consider a pair $(C,T)\in \TZ^p(N_\Sigma) \times E^{p-1,p-1}(N_\Sigma)$ and tropical Cartier divisors $D_1,D_2$ on $N_\Sigma$ intersecting the tropical cycle $C$ properly. 
For Green functions $g_1,g_2$ {for} $D_1,D_2$ of
piecewise smooth type, we have
\[
(D_1,g_1)*(D_2,g_2)*(C,T)=(D_2,g_2)*(D_1,g_1)*(C,T) \quad \text{\rm modulo} \quad \mathrm{Im}(d'+d'')
\]	
where 
\[
\mathrm{Im}(d'+d'')=(0,d'E^{p,p+1}(N_\Sigma)+d''E^{p+1,p}(N_\Sigma))\subset \TZ^{p+2}(N_\Sigma) \times E^{p+1,p+1}(N_\Sigma).
\]
\end{prop}

\begin{proof}
More precisely, following Moriwaki's terminology \cite{moriwaki-book}, we will prove Weil's reciprocity law
\begin{equation} \label{Weil reciprocity law}
\begin{split}
(D_1,g_1)*(D_2,g_2)*(C,T)&-(D_2,g_2)*(D_1,g_1)*(C,T)\\&=
\bigl( 0, d''(g_1 d'g_2 \wedge \delta_C)+d'(g_2 d'' g_1 \wedge \delta_C)\bigr).
\end{split}
\end{equation}
By Proposition \ref{symmetry of intersection product with tropical Cartier divisors}, we have $D_{1}\cdot D_2\cdot C-D_{2}\cdot D_1\cdot C=0$. 
Therefore, there is a $\delta $-current $S$ such that 
\begin{align*}
(D_1,g_1)*(D_2,g_2)*(C,T)-(D_2,g_2)*(D_1,g_1)*(C,T)
=( 0, S).
\end{align*}
and we have to prove the identity of $\delta$-currents
\begin{displaymath}
S=d''(g_1 d'g_2 \wedge \delta_C)+d'(g_2 d'' g_1 \wedge \delta_C)
\end{displaymath}
The claim is local and we have to test against a compactly supported $\delta$-form $\eta$ of bidegree $(n-p-1,n-p-1)$. 
By the proper intersection hypothesis and using that the $\delta$-form is constant towards the boundary, the form $\eta$ is zero in a neighbourhood of $S(D_1) \cap S(D_2) \cap |C|$ and there the claim is clear. 
So we may assume that $g_1$ is a piecewise smooth function and that $D_1=0$. 
Then we have
  \begin{align*}
    S&=(g_1 \wedge \delta_{D_2 \cdot C}+\omega_1\wedge g_2 \wedge
       \delta_C+\omega_1 \wedge \omega_2 \wedge T)-(\omega_2\wedge g_1
       \wedge \delta_C+\omega_2 \wedge \omega_1 \wedge T) \\
    &=g_1 \wedge \delta_{D_2 \cdot C}+\omega_1\wedge g_2 \wedge
       \delta_C-\omega_2\wedge g_1 \wedge \delta_C
  \end{align*}
We use now that $g_1$ is a piecewise smooth function and hence
\[
\omega_1 \wedge g_2 \wedge \delta_C=d'd''g_1 \wedge g_2 \wedge \delta_C=d'(d''g_1 \wedge g_2 \wedge \delta_C) +d''g_1\wedge d'g_2 \wedge \delta_C
\]
is true as an identity of formal $\delta$-currents. 
The {tropical Poincar\'e--Lelong equation for Green functions of piecewise smooth type in Theorem \ref{tropical PL equation for Green functions}}
for $g_2$ and multiplication with the piecewise smooth $g_1$ yield
\[
\omega_2\wedge g_1 \wedge \delta_C=g_1 d'd''g_2 \wedge \delta_C+g_1 \wedge \delta_{D_2 \cdot C}.
\]
Again the Leibniz formula for $\delta$-forms shows that
\[
g_1d'd''g_2\wedge \delta_C= -g_1d''d'g_2\wedge \delta_C=-d''(g_1d'g_2)\wedge \delta_C+d''g_1 \wedge d'g_2\wedge \delta_C
\]
proving \eqref{Weil reciprocity law} by using Remark \ref{closedness of tropical cycles}.
\end{proof}

We study now functoriality of the $*$-product.

\begin{definition}(Direct image) \label{assumptions for functoriality of star product}
Let $n'$ denote the finite rank of the free abelian group $N'$ and let $L\colon N' \to N$ be a homomorphism of groups.
Let $F\colon N_{\Sigma'}'\to N_\Sigma$ be a proper $L$-equivariant
morphism of tropical toric varieties. 
For a cone $\sigma' \in \Sigma'$, there is a unique cone $\sigma \in \Sigma$ with $F(N'(\sigma'))\subset N(\sigma)$.
Let $C' \in \TZ^{p' }(N_{\Sigma'}')$  of sedentarity $\sigma'$ and $T' \in E^{p'-1,p'-1}(N_{\Sigma'}')$. 
Then we note that $F_*(C')$ is a tropical cycle in $N_\Sigma$ of sedentarity $\sigma$.
Since $F_*(C')$ has dimension $n'-p'$, this  tropical cycle has codimension $n-n'+p$. 
For $T' \in E^{p-1,p-1}(N_{\Sigma'}')$, we define  
\[
F_*(C',T') \coloneqq  (F_*(C'),F_*(T')) \in E^{n-n'+p-1,n-n'+p-1}(N_{\Sigma}).
\]
\end{definition}

\begin{prop} \label{projection formula for star product}
Under the assumptions in Definition \ref{assumptions for functoriality of star product}, we consider a tropical toric Cartier divisor $D$ on $N_\Sigma$ which intersects $F(|C'|)$ properly.  For a Green function $g$ for $D$ of piecewise smooth type, we have the projection formula
	\[
	F_*\bigl((F^*D,F^*g) * (C',T')\bigr)=(D,g)*F_*(C',T').
	\]
\end{prop}
\begin{proof}
	It follows from the assumptions and Example \ref{pull-back of tropical toric Cartier divisor} that the pull-back $F^*(D)$ is a well-defined tropical toric Cartier divisor on $N_{\Sigma'}'$.
	By Proposition \ref{projection formula of intersection product}, we get
	$$F_*(F^*D\cdot C')=D \cdot F_*C'$$
	as an identity of tropical cycles on $N_\Sigma$. By definition, we have 
	$$ (F^*D,F^*g) * (C',T') = (F^*D \cdot C', F^*g \wedge \delta_{C'}+\omega(F^*D,F^*g)\wedge T').$$
We deduce from the projection formula in Proposition \ref{projection formula for above product} that we have
\[
F_*(F^*g \wedge \delta_{C'})= g \wedge \delta_{F_*C'}.
\]
We have seen in Proposition \ref{properties of the first Chern delta form} that $\omega(F^*D,F^*g)=F^*(\omega(D,g))$ and hence 
\[
F_*(\omega(F^*D,F^*g) \wedge T') = \omega(D,g) \wedge F_*T'.
\]
Using the above identities, we get the claimed projection formula.
\end{proof}

\section{Non-archimedean Arakelov theory} \label{section: Non-archimedean Arkelov theory}

In this section, we study $\delta$-forms on the Berkovich analytification $\Xan$ of a $d$-dimensional algebraic variety $X$ over a non-archimedean field $K$. 
We use $\delta$-forms to define a $*$-product with Green functions for Cartier divisors of piecewise smooth type extending the approach in \cite[Section 11]{gubler-kuenne2017} from so-called $\delta$-metrics to all piecewise smooth metrics. 
This is possible by Mihatsch's improved approach to $\delta$-forms \cite{mihatsch2021a}.

Let $K$ be a non-archimedean field.
Let $\Xan$ denote the Berkovich analytification of a $d$-dimensional algebraic variety $X$ over $K$. 
Let $\Sigma$ be a fan in $N_\R$ for some free abelian group $N$ of finite rank $n$.
If $X_\Sigma$ is the toric variety over $K$ associated with $\Sigma$ and $\varphi\colon X\to X_\Sigma$ is a morphism, then we consider the so-called \emph{extended tropicalization}
\begin{equation}
\varphi_\trop\colon X^\an\stackrel{\varphi^\an}{\longrightarrow}
X_\Sigma^\an\stackrel{\trop}{\longrightarrow}N_\Sigma
\end{equation} 
where $\trop\colon X_\Sigma^\an\rightarrow N_\Sigma$ is the extended tropicalization map \cite[Remark 3.1.3]{burgos-gubler-jell-kuennemann1}.
Jell \cite{jell-thesis} used extended tropicalizations to represent
differential forms on $\Xan$ by differential forms on the
tropicalization $N_\Sigma$.  
Similarly, we will represent $\delta$-forms on $\Xan$ by $\delta$-forms on $N_\Sigma$. 
In this way, we can use tropical results from Section \ref{sec:a tropical star-product} for non-archimedean Arakelov theory on the non-archimedean analyitification $X^\an$.

\subsection{Mihatsch's $\delta$-forms} \label{subsection: Mihatsch's delta-forms}
In \cite{gubler-kuenne2017}, $\delta$-forms were introduced on $\Xan$
to construct a $*$-product with Green
functions for Cartier divisors associated to so-called
$\delta$-metrics which are piecewise smooth metrics such that the
first Chern current is a $\delta$-form.  
This includes all model metrics and all smooth metrics. 
In \cite[Section 4]{gubler-kuenne2017}, $\delta$-forms on $\Xan$ were introduced by using tropicalizations on open subsets of $\Yan$ for all morphisms from $Y\to X$ of $K$-varieties.
Mihatsch \cite{mihatsch2021,mihatsch2021a} generalized this approach,
allowing more $\delta$-forms on the tropical side as described in
Subsection \ref{delta-forms-on-N_R} including all piecewise smooth
forms and a more intrinsic approach on the analytic side using
Ducros's skeletons in $\Xan$ instead of morphisms $Y \to X$ to
identify $\delta$-forms obtained from tropicalizations.  
We briefly recall Mihatsch's construction as follows.

\begin{cons} \label{Mihatsch's delta forms}
Mihatsch constructs a  sheaf $B$ of $\delta$-forms on $\Xan$ (in fact on any $K$-analytic space of pure dimension $d$ with empty boundary). 
This is a sheaf of bigraded differential $\R$-algebras $B$ with differentials $d',d''$. 
A $\delta$-form $\alpha$ is locally given by a morphism $\varphi\colon V \to (\G_{\rm m}^n)^\an$ on an open subset $V$  and a $\delta$-form $\alphatrop$ on $\R^n$ such that $\alpha=\varphi_{\rm \trop}^*(\alphatrop)$ seen as a polyhedral current on the tropical skeleton $\Sigma(V,\varphi)$. 
Here, we use the \emph{tropicalization of $\varphi$} defined by $\varphitrop \coloneqq \trop \circ \varphi$ and Ducros's tropical skeleton $\Sigma(V,\varphi)$ is explained below. 
The construction of $\delta$-forms is compatible with refinements of moment maps. The trigrading of $\delta$-forms on $\R^n$ induces also a trigrading of the sheaf $B$ on $\Xan$ and 
\[
B^{p,q}(W) = \oplus_{r \in \N} B^{p-r,q-r,r}(W)
\]
for any open subset $W$ of $\Xan$ as in Definition \ref{definition of delta-form}.
There is a pull-back of $\delta$-forms with respect to morphisms of analytic spaces. 
This pull-back respects the bigrading and the trigrading. 
We refer to \cite{mihatsch2021a} for details. 
\end{cons}

\begin{rem} \label{Ducros's skeletons and a warning}
We recall that Ducros \cite[Section 5]{ducros12} introduced skeletons $\Sigma(V,\varphi)\coloneqq \varphi^{-1}(S(\mathbb G_{\rm m}^\an))$ as the preimage of the canonical skeleton $S(\mathbb G_{\rm m}^\an)$ of $\mathbb G_{\rm m}^\an$ given by the weighted Gauss points. 
He showed that $\Sigma(V,\varphi)$ has a canonical piecewise linear structure such that $\varphitrop$ is an injective affine map on the faces of $\Sigma(V,\varphi)$ to $\R^n$. 
The piecewise linear space $\Sigma(V,\varphi)$ has pure dimension $n$ (if non-empty).
Note that $\Sigma(V,\varphi)$ is a \emph{weighted} $n$-dimensional polyhedral space \cite[Section 4]{gubler-jell-rabinoff1} or equivalently there is a canonical calibration of $\Sigma(V,\varphi)$ as used by Mihatsch, see \cite[Section 7]{gubler-forms} for the equivalence between calibrations and weights. 
\end{rem}

\begin{rem} \label{algebraization}
The tropical skeleton $\Sigma(V,\varphi)$ depends only on the values $|\varphi|$ taken on the open subset $V$ of $\Xan$, see the description in terms of graded reductions of the functions $\varphi_1,\dots,\varphi_n$ at $x$ in \cite[(3.4)]{mihatsch2021a}. 
Since locally at $x \in \Xan$ with respect to the Berkovich topology, we can approximate analytic functions by algebraic functions with respect to uniform convergence, we deduce that we may always assume that the morphism $\varphi \colon V \to (\G_{\rm m}^n)^\an$ tropicalizing the $\delta$-form $\alpha$ is given by restriction of an algebraic morphism $U \to \G_{\rm m}^n$ for a Zariski open subset $U$ of $X$ and a possible smaller open neighbourhood $V$ of $x$ in $\Uan$  using the same $\alphatrop \in B(\R^n)$. 
\end{rem}

\begin{definition} 
A \emph{refinement} of a morphism $\varphi\colon V \to (\G_{\rm m}^n)^\an$ is given by an open subset $V'$ of $V$, a morphism $\varphi'\colon V' \to (\G_{\rm m}^{n'})^\an$ and a homomorphism $p\colon  \G_{\rm m}^{n'} \to  \G_{\rm m}^{n}$ such that $\varphi|_{V'} = p^\an \circ \varphi'$. 

Similarly, we say that a piecewise linear map $F' \colon P' \to \R^{n'}$ on a polyhedral subset $P'$ of a polyhedral space $P$ (as in Remark \ref{Ducros's skeletons and a warning}) refines a given piecewise linear map $F \colon P \to \R^n$ if there is a linear map
$p\colon \R^{n'} \to \R^n$ such that $F|_{P'}= p\circ F'$.
\end{definition}

\begin{rem} \label{tropical germ}
If $V$ and $\varphi$ are algebraic, then the Bieri--Groves theorem shows that $\varphitrop(V)$ is a tropical cycle in $\R^n$ of pure dimension at most $d$. 
On the other hand, if $V$ is an arbitrary open subset of $\Xan$ and $\varphi$ is any analytic morphism, then we cannot expect that this still holds. 
However, Ducros \cite[Theorem 3.4]{ducros12} proved that every $x \in V$ has an open neighbourhood $W$ such that $\varphitrop(W)$ is an open subset of a translated tropical fan in $\R^n$ of pure dimension at most $d$. 
\end{rem}

\begin{art} \label{differential forms and currents by Chambert-Loir and Ducros}
There is a bigraded sheaf $A$ of smooth forms on $\Xan$ introduced by Chambert-Loir and Ducros \cite{chambert-loir-ducros}. 
It is a sheaf of bigraded differential $\R$-algebras with differentials $d'$, $d''$
similar to the sheaf of smooth $(p,q)$-forms on complex manifolds. 
Note that $A$ is a subsheaf of $B$, so $\delta$-forms are a natural
generalization of the smooth forms capturing the fact that first Chern
currents of formal metrics associated to models of line bundles are
not necessarily given by smooth $(1,1)$-forms, but by $\delta$-forms
\cite{gubler-kuenne2017}.  

The sheaf of \emph{currents $D$} is the topological dual of $A_{c}$. We use the bigrading of $D$ such that $A$ is a bigraded differential subsheaf of 
$D$. Note also that we can realize the underlying polyhedral currents of the $\delta$-forms as currents and $B$ is also a bigraded differential subsheaf of $D$. 
\end{art}

\begin{rem} \label{warning for delta-forms}
There is an important difference between smooth forms and $\delta$-forms which we have to take into account if we generalize facts from smooth forms to $\delta$-forms. 
If $\alpha, \beta$ are smooth forms on the open subset $V$ of $\Xan$, given by smooth forms $\alphatrop$ and $\betatrop$ on $\R^n$ with respect to the same tropicalization $\varphitrop$, then it follows from \cite[Lemma 3.2.2]{chambert-loir-ducros} that $\alpha=\beta$ if and only if for all $x \in V$, we have $\alphatrop|_{{\varphi_{\trop}}(W)}=\betatrop|_{{\varphi_{\trop}}(W)}$ where $W$ is the open neighbourhood of $x$ from Remark \ref{tropical germ}.

This is no longer true for $\delta$-forms $\alpha,\beta$ as shown in \cite[Example 4.22]{gubler-kuenne2017}. 
In Mihatsch's definition, we have $\alpha=\beta$ if and only if for every polyhedral subspace $Q$ of $P=\Sigma(V,\varphi)$ of pure dimension $d$ with boundary $\partial Q$ in $P$ and every refinement $G\colon Q \to \R^{n'}$ of $F = \varphitrop|_P$, we have the identity of currents
\[
p^*\alphatrop \wedge G_*\delta_{(Q,m)}= p^*\betatrop \wedge G_*\delta_{(Q,m)}
\]
away from $Q(\partial Q)$. Here, we use the restriction of the canonical weight $m$ of $\Sigma(V,\varphi)$ to $Q$ to see $(Q,m)$ as a weighted polyhedral set of pure dimension $d$. The problem here is that there is no pull-back of $\delta$-forms with respect to morphisms of polyhedral spaces, so it is essential to give a $\delta$-form $\alpha$ on $V$ by a $\delta$-form $\alphatrop$ on $\R^n$ and \emph{not} on any polyhedral subspace of $\R^n$ containing $\varphitrop(V)$.
\end{rem}

\subsection{Mihatsch's $\delta$-forms and extended tropicalizations} \label{subsection: delta-forms and extended tropicalizations}
We use extended tropicalizations to construct $\delta$-forms on the analytification $X^\an$.

\begin{cons}[$\delta$-forms from $\A$-tropicalizations]\label{delta-forms with A-tropicalizations}
To apply the results of our previous paragraphs, we need to consider \emph{$\A$-tropicalizations}. 
Let $\psi\colon W \to (\A^n)^\an$ be a morphism on an open subset $W$ of $\Xan$. 
Then the induced \emph{extended tropicalization}  is defined by 
\[
\psitrop \coloneqq \trop \circ \psi\colon W \longrightarrow \R_\infty^n
\]
with $\R_\infty=\R\cup\{\infty\}$.
We show that every $\delta$-form $\betatrop$ on $\R_\infty^n$ induces a $\delta$-form $\beta$ on $W$ which we denote by $\psi^*_\trop(\beta)$. 
The construction of $\psi^*_\trop(\beta)$ is local on $W$, so we fix some $x \in W$. 
By reordering the coordinates, we may assume that $\psi_1(x),\dots,\psi_k(x)$ are non-zero and $\psi_{k+1}(x)=\dots=\psi_n(x)=0$. 
Hence, the first $k$ coordinates of $\psi_{\trop}(x)$ are finite while the last $n-k$ coordinates are equal to $\infty$.
Let $p \colon \R^k \times \R_\infty^{n-k} \to \R^k$ denote the first projection. 
Then there is an open neighbourhood $\Omega$ of $p \circ \psitrop(x)$ in $\R^k$, $r \in \R$ and $\alphatrop \in B(\Omega)$ such that $\betatrop=p^*(\alphatrop)$ on $\Omega(r)\coloneqq \Omega \times (r,\infty]^{n-k}$. 
The set $V \coloneqq \varphi_\trop^{-1}(\Omega(r))$ is an open
neighbourhood of $x$ in $W$ and we have a morphism $\varphi\colon W
\to (\G_{\rm m}^k)^\an$ induced by $\psi_1,\dots, \psi_k$.  
We  define $\beta|_V \coloneqq \varphi_\trop^*(\alphatrop)$. 
Clearly, this does not depend on the choice of $V$ and hence the $\delta$-forms glue to a $\delta$-form $\beta=\psi^*_\trop(\beta)$ on $W$.
\end{cons}

\begin{cons}[$\delta$-forms from extended
  tropicalizations]\label{delta-forms with T-tropicalizations} 
Construction \ref{delta-forms with A-tropicalizations} has an obvious generalization to morphisms $\psi \colon W \to X_\Sigma^\an$ for any toric variety $X_\Sigma$. 
Then the \emph{extended tropicalization} $\psitrop \coloneqq \trop \circ \psi$ is a continuous map from the open subset $W$ of $\Xan$ to the tropical toric variety $N_\Sigma$. 
We claim that every $\delta$-form $\betatrop$ on $N_\Sigma$ induces a $\delta$-form $\beta$ on $W$ denoted by $\psi_\trop^*(\betatrop)$. 
We say that \emph{$\beta$ is tropicalized by $\psitrop$ and $\betatrop$.}  
	
We proceed locally at $x \in V$. 
Let $\sigma \in \Sigma$ such that $\psitrop(x) \in N(\sigma)$. 
Let $U_\sigma$ be the affine open  subset of $X_\Sigma$ associated to the cone $\sigma$. 
Then $V \coloneqq \psi^{-1}(U_\sigma)$ is an open subset of $W$ given by the union of $\psi_\trop^{-1}({N(\tau )})$ where $\tau$ ranges over all faces of $\sigma$. 
Let $S_\sigma \coloneqq \{u \in M \mid \langle u, \sigma \rangle \geq 0\}$ be the semigroup associated to $\sigma$, then we have $U_\sigma = \Spec (K[S_\sigma])$. 
Since $\betatrop$ is constant towards the boundary, there is an open neighbourhood $\Omega$ of $\psitrop(x)$ in $U_\sigma^\an$ such that $\betatrop|_\Omega$ is determined by $\alphatrop|_{\Omega \cap N(\sigma)}$ for a $\delta$-form $\alphatrop$ on $N(\sigma)$ using pull-backs of $\alphatrop$ along the canonical maps $N(\tau) \to N(\sigma)$ for all faces $\tau$ of $\sigma$. Note that the integral structure $N \cap N(\sigma)$ has dual $M(\sigma)\coloneqq \sigma^\perp \cap M$ and hence we have a morphism from $U_\sigma$ to the split torus $T(\sigma)\coloneqq \Spec K[M(\sigma)]$. 
Let $\varphi \colon V \to T(\sigma)$ be the morphism induced by composition with $\psi$, then we define $\beta|_V \coloneqq \varphi_{\trop}^*(\alphatrop)$. 
Again this $\delta$-form does not depend on the choice of $V$ and hence we obtain a well-defined $\delta$-form $\beta$ on $W$.
\end{cons}

\begin{prop} \label{properties of tropicalized delta-forms}
Let $\psitrop$ be an extended tropicalization given by a morphism $\psi \colon W \to X_\Sigma$. 
Then Construction \ref{delta-forms with T-tropicalizations} gives a
morphism $\psi_\trop^*\colon B(N_\Sigma) \to B(W)$ of bigraded
$\R$-algebras compatible with the differentials $d',d''$. 
Moreover, this morphism is compatible with 
pull-back maps associated with 
morphisms $W' \to W$ of analytic spaces and also with 
refinements of extended tropicalization maps. 
\end{prop}

\begin{proof}
This is clear for tropicalizations induced by morphisms $\varphi\colon W \to \G_{\rm m}^n$. As the above construction reduces to the case of such morphisms, we get the claim. 
\end{proof}

\begin{rem} \label{integration and tropicalization}
	We can integrate a compactly supported $\delta$-form $\beta$
        on $\Xan$ of bidegree $(d,d)$.  Mihatsch \cite[\S
        4.3]{mihatsch2021a} notes that the support of $\beta$ is
        contained in  the $d$-dimensional part $S$ of a finite
        union of tropical skeletons which tropicalize $\beta$
        locally. Note that the set $S$ is denoted by $\Sigma $ in \emph{loc.cit}. 
	Since $\beta|_\Sigma$ is a polyhedral current on $S$, he defines $\int_\Xan \beta \coloneqq \int_S \beta|_\Sigma$ using the canonical calibrations on $S$ introduced by Chambert--Loir and Ducros. We refer to \textit{ibid.}\,for more details.
\end{rem}	

We will use integration mainly in the following situation.

\begin{ex} \label{integration and global tropicalization}
Let $\psi\colon X \to X_\Sigma$ be a closed embedding into a toric variety $X_\Sigma$ and $\beta_\trop \in
B_c^{n,n}(N_\Sigma)$.  
The Bieri--Groves theorem shows that $\psitrop(\Xan)$ is a polyhedral subset of pure dimension $d$ and of some sedentarity $\sigma \in \Sigma$. 
From tropical geometry,  we have canonical weights on $\psitrop(\Xan)$ giving it the structure of a classical tropical cycle $C$ on $N_\Sigma$ of pure dimension $d$, see e.~g.~\cite[\S 13]{gubler-guide}.
	 
	 Replacing $X_\Sigma$ by the closure of the orbit corresponding to $\sigma$, we may assume that $\sigma$ is the dense orbit $T$. We consider a compactly supported $\betatrop \in B^{d,d}(N_{\Sigma})$. We claim then that $\beta \coloneqq \psi_\trop^*(\betatrop) \in B_c^{d,d}(\Xan)$. If $X_\Sigma$ is proper, then this would follow directly from the fact that $\trop$ is a proper map. In general, we pass to a toric compactification $X_{\overline \Sigma}$ and then we use that the support of a top degree $\delta$-form  on the closure of $X$ in $X_{\overline \Sigma}$ is disjoint from any lower dimensional subset  \cite[Corollary 4.9]{mihatsch2021a} and hence is contained in $\Xan$.  
	Using (4) in \cite[\S 4.3]{mihatsch2021a}, we get
	\begin{equation} \label{integration of globally tropicalized forms}
		\int_\Xan \beta = \int_\Xan \psi_\trop^*(\betatrop) = \int_{C} \betatrop.
	\end{equation}
\end{ex} 

The following result is the \emph{transformation formula} for $\delta$-forms.

\begin{prop} \label{transformation formula over X}
Let $\varphi \colon X' \to X$ be a morphism of varieties of dimension $d$ and let $\beta \in B_c^{d,d}(X)$. 
We define $\deg(\varphi)=[K(X'):K(X)]$ if $\varphi$ is surjective and $\deg(\varphi)=0$ otherwise. 
Then $\varphi^*\beta \in B_c^{d,d}((X')^\an)$ and we have
\[
\int_{(X')^\an} \varphi^*\beta = \deg(\varphi) \int_\Xan \beta.
\]
\end{prop}

\begin{proof}
By Nagata's compactification theorem, $X'$ is an open dense subset of a variety $\overline{X'}$ such that $\varphi$ extends to a proper morphism $\overline{\varphi}\colon \overline{X'} \to X$. 
Then $\overline{\varphi}^* \beta $ is a compactly supported $(d,d)$-form whose support is disjoint from the lower dimensional algebraic set $ \overline{X'} \setminus X$ by \cite[Corollary 4.9]{mihatsch2021a} and hence $\varphi^*\beta$ is compactly supported in $\Xan$. 
We may assume that $\varphi$ is surjective, as otherwise both sides of the transformation formula vanish. 
The claim is local in $X$, so we may assume $X$ affine and that $\beta$ is tropicalized by a morphism  $\psi\colon X \to \A^n$ and by $\betatrop \in B_c^{d,d}(\R_\infty^n)$. 
Note also that $\varphi^*\beta$ is tropicalized by $\psi' \coloneqq \psi \circ \varphi$ and the same $\delta$-form $\betatrop$. 
The tropical variety associated to $\psi'$ is the classical tropical cycle $C'=\deg(\phi)C$ and it is shown in  \cite[\S 4.3]{mihatsch2021a}(4) that
\[
\int_{(X')^\an} \varphi^*\beta = \int_{C'} \betatrop.
\]
Together with \eqref{integration of globally tropicalized forms}, we get the transformation formula.
\end{proof}

\subsection{Global tropicalizations} \label{subsection: global tropicalizations}

Jell \cite{jell-thesis} used extended tropicalizations to tropicalize smooth $(p,q)$-forms globally. In this subsection, we do the same for $\delta$-forms. 

\begin{definition} \label{tropical charts}
	A \emph{tropical $\A$-chart} for $X$ is a pair $(V,\psi)$ where $\psi\colon U \to \A^n$ 
is a closed embedding from an open subset $U$ of $X$ and $V$ is an open subset of $\Uan$ given by $V \coloneqq \psi_\trop^{-1}(\Omega)$ for an open subset $\Omega$ of $\psitrop(\Uan)$.

We say that $(V',\psi')$ is a \emph{tropical $\A$-subchart} of $(V,\psi)$ if $(V',\varphi')$ is a tropical chart for $X$ such that the closed embedding $\psi'\colon U' \to \A^{n'}$ refines $\varphi$ and $V' \subset V$. 
\end{definition}

\begin{rem} \label{remarks about tropical charts}
Since $\psi$ is a closed embedding, the Bieri--Groves theorem shows that $\psitrop(\Uan)$ is a polyhedral subset of pure dimension $d$ in $\R_\infty^n$. By construction, we have $\Omega = \psitrop(\Uan)$.  

By the definition of the Berkovich topology on the analytification of an affine variety, it is clear that the tropical $\A$-charts form a basis for the topology. 
We will recall the proof below.
\end{rem}

The next lemma hints that tropical $\A$-charts are very useful to tropicalize $\delta$-forms.

\begin{lemma} \label{first step for affine tropicalization}
Let $X$ be an affine variety over $K$ and let $\beta$ be a $\delta$-form on an open subset $W$ of $\Xan$. Then for every $x \in W$, there is a closed embedding $\psi\colon X \to \A^n$,  $\betatrop \in B(\R_\infty^n)$ and an open neighbourhood  $\Omega$ of $\psitrop(x)$ in  $\psitrop(\Xan)$ such that $V \coloneqq \psi_\trop^{-1}(\Omega)$ is an open neighbourhood of $x$ in $W$ and $\beta|_V= \psi_\trop^*(\betatrop)|_V$.
\end{lemma}

\begin{proof}
By Remark \ref{algebraization}, there is a Zariski open subset $U$ of
$X$, a morphism $\varphi \colon U \to \G_{\rm m}^n$ and  a relatively
compact open neighbourhood $V$ of $x$ in $W \cap \Uan$ such that
$\beta|_{V}=\phi_\trop^*(\alphatrop)|_{V}$ for some $\alphatrop \in
B(\R^n)$. Passing to  smaller $U$ and $V$, we may assume that $U$ is a
basic affine open subset of $X$, i.e.~$U=\{x \in X \mid h(x) \neq 0\}$
for some regular function {$h$} on $X$. 
The coordinate functions $\varphi_1, \dots, \varphi_n$ of $\varphi$ satisfy $\varphi_j= g_j/h^{m_j}$ for some regular functions $g_j$ on $X$ and $m_j \in \N$. Obviously, we may tropicalize $\beta$ on $V$ also by the morphism $g=(g_1,\dots,g_n)\colon X \to \A^n$ and so we assume from now on that $\varphi =g$ and that $\alphatrop \in B(\R_\infty^n)$.
 Using the definition of the Berkovich topology on the affine variety $X$, we may assume that 
$$V=\{r_1 <|f_1|<s_1\} \cap \dots \{r_k <|f_k|<s_k\}$$
for finitely many regular functions $f_1,\dots, f_k$ on $X$ and $r_1,s_1 \dots ,r_k,s_k \in \R$.
Let us consider the morphism $f\colon X \to \A^k$ given by $(f_1,\dots,f_k)$. Adding more regular functions to the list, we may assume that $f$ is a closed immersion. By construction, $(V,f)$ is a tropical $\A$-chart. The morphism $\psi \coloneqq (f,g)\colon X \to \A^{k+n}$ refines $f$ and $g$, hence $(V,\psi)$ is a tropical $\A$-subchart of $(V,f)$. Let $p\colon \A^{k+n}=\A^k \times \A^n\to \A^n$ be the second projection. Setting $\betatrop \coloneqq p_\trop^*(\alphatrop)$, we deduce the claim.
\end{proof}

\begin{prop} \label{global tropicalization for affine}
Let $X$ be an affine variety over $K$ and let $\beta$ be a $\delta$-form on $\Xan$ with compact support. 
Then there is a closed embedding $\psi\colon X \to \A^n$  and $\betatrop \in B_c(\R_\infty^n)$ such that $\beta = \psi_\trop^*(\betatrop)$. 
\end{prop}
	
\begin{proof}
For $x \in \Xan$, Lemma \ref{first step for affine tropicalization} shows that there is an open neighbourhood $V$ of $x$ in $\Xan$ and a closed embedding $\psi_V\colon X \to \A^{n_V}$ such that $(V,\varphi_V)$ is a tropical $\A$-chart which tropicalizes $\beta|_V$ in the sense that $\beta|_V=\psi_{V,\trop}^*(\beta_{V,\trop})$ for some $\beta_{V,\trop} \in B(\R_\infty^n)$. 
We may also assume that $V$ is relatively compact in $\Xan$.
Note that we may always replace $\psi_V$ by a morphism $X \to \A^N$ refining $\psi_V$.

	Since $S \coloneqq \supp(\beta)$ is compact, we can cover it by finitely many $V$'s as above. We set $I$ for this finite set of $V$'s. Replacing $\psi_V$ by $\prod_{V \in I} \psi_V$ for each $V \in I$, we may assume that all these $\psi_V$ agree with a closed embedding $\psi\colon X \to \A^n$. For all $V \in I$, the construction yields
	\begin{equation} \label{basic identity from construction}
		\beta|_V=\psi_{\trop}^*(\beta_{V,\trop})|_V.
    \end{equation}
	For each $V \in I$, there is a relatively compact open subset $\Omega_V$ of $\R_\infty^n$ such that $\psi_\trop^{-1}(\Omega_V)=V$. This is based on the fact that $(V,\psi)$ are $\A$-tropical charts. Since the open subset $V \in I$ cover $S$, we conclude that  $(\Omega_V)_{V \in I}$ is an open covering of $\psitrop(S)$. The latter is compact and hence we get a finite open covering $(\Omega_V)_{V \in I'}$ by setting $I' \coloneqq I \cup \{\infty\}$ and $\Omega_\infty \coloneqq \R_\infty^N\setminus \psitrop(S)$. We choose a partition of unity $(\phi_V)_{V \in I'}$ subordinate to $(\Omega_V)_{V \in I'}$. We define 
	$$\betatrop \coloneqq  \sum_{V \in I} \phi_V \beta_{V,\trop} \in B_c(\R_\infty^N)$$
	using that $\phi_V$ has compact support for $V \in I$ as $\Omega_V$ is relatively compact.

	We have to show that $\beta = \psi_\trop^*(\betatrop)$. We check this identity locally at $x \in \Xan$. 
	
	Suppose first that $x \not\in S$. Let us pick $V \in I$ such that $x \in V$ (if it exists). Since $x$ is not in the support $S$ of $\beta$, we conclude from \eqref{basic identity from construction} that $\psi_\trop^*(\beta_{V,\trop})|_W =0$ in a sufficiently small open neighbourhood $W$ of $x$ in $\Xan \setminus S$. As there are only finitely many such $V$'s, we may choose $W$ independently of the choice of $V$. On the other hand, let us consider $V \in I$ with $x \not \in V$. Using that $(V,\psi)$ is an $\A$-tropical chart, we deduce that $\psitrop(x)  \not\in \Omega_V$ and hence $\phi_V$ vanishes in an open neighbourhood $\Omega$ of $\psitrop(x)$. Again, we may choose that independently of the choice of such $V$'s and we may assume that $W \subset \psi_\trop^{-1}(\Omega)$. We conclude that for all $V \in I$, we have $\psi_\trop^*(\phi_V)|_W=0$ or $\psi_\trop^*(\beta_{V,\trop})|_W =0$. This proves 
	$\psi_\trop^*(\betatrop)|_W=0$ and hence the claimed identity follows on the open neighbourhood $W$ of $x$.
	
	It remains to check the identity locally at $x \in S$. Pick any $V \in I$. If $x \in V$, then \eqref{basic identity from construction} holds on $V$. 
	If $x \not\in V$, then it follows as above that $\psi_\trop^*(\phi_V)=0$ in a neighbourhood of $x$. In any case, finiteness of $I$ yields that there is an open neighbourhood $W$ of $x$ such that 
	\begin{equation} \label{conclusion of basic identity}
		\psi_\trop^*(\phi_V) \beta|_W = \psi_\trop^*(\phi_V\beta_{V,\trop})|_W.
	\end{equation}
	Note that the support of $\phi_\infty$ is a closed subset in $\R_\infty^n$ disjoint of $\psitrop(S)$ and hence we may assume that the open neighbourhood $W$ of $x$ is disjoint from the support of $\psi_\trop^*(\phi_\infty)$ which means that $\psi_\trop^*(\phi_\infty)|_W=0$. Now summing up \eqref{conclusion of basic identity} over all $V \in I$ and using that $(\phi_V)_{V \in I'}$ is a partition of unity of $\R_\infty^n$, we get $\beta|_W = \psi_\trop^*(\betatrop)|_W$. This proves the desired identity locally at $x$.
\end{proof}

We have the following generalization of the global tropicalization result for quasi-projective varieties.
Let $\Trop(\P^n)$ denote the tropical toric variety $X_{\Sigma(\P^n)}$ for the complete fan $\Sigma(\P^n)$ in $\R^n$ with one dimensional faces $-e_1-\ldots -e_n,e_1,\ldots,e_n$, with $\{e_{1},\dots,e_n\}$ the canonical basis of $\R^{n}$.

\begin{prop} \label{global tropicalization for quasi-projective}
Let $X$ be a quasi-projective variety over $K$ and let $\beta$ be a $\delta$-form on $\Xan$ with compact support. Then there exist an immersion $\psi\colon X \to \P^n$  and a $\delta$-form $\betatrop$ on $\Trop(\P^n)$ such that $\beta = \psi_\trop^*(\betatrop)$. 
\end{prop}

\begin{proof}
The proof is similar to the one in the affine case
and we indicate just the differences. We pick an immersion $X \to
\P^m$ into projective space $\P^m$ with projective coordinates
$x_0,\dots,x_m$. 
Let $\overline X$ be the closure of $X$ in $\P^{m}$.  
For $x \in \Xan$, there is some non-constant homogeneous polynomial $h
\in K[x_0,\dots,x_m]$ such that $x$ is contained in the basic affine
open subset  $U \coloneqq \{x \in \overline X\,|\, h \neq 0\}$.  
After multiplying $h$ by another polynomial, we may assume that $U
\subset X$ and that there is an $i\in\{0,\dots,m\}$ with that $x_i
\vert h$ which means $U \subset \{x_i \neq 0\}$.

Applying Lemma \ref{first step for affine tropicalization}, there is  a {relatively compact} open neighbourhood $V$ of $x$ in $U^\an$, a morphism $\psi_V\colon U \to \A^{n_V}$, an open neighbourhood $\Omega_V$ of $\psi_{V,\trop}(x)$ in $\R_\infty^{n_V}$ and $\beta_{V,\trop} \in B(\R_\infty^{n_V})$ such that $\beta|_{V} =\psi_{V,\trop}^*(\beta_{V,\trop})|_{V}$ on the open neighbourhood $V \coloneqq \psi_{V,\trop}^{-1}(\Omega_V)$ of $x$ in $U^\an$. 
From now on, we will write $U_V$ instead of $U$ for the basic affine
open subset involved with $V$ and $i(V)$ for the index $i$ used above. 

We cover the compact support $S \coloneqq \supp(\beta)$ by finitely
many such $V$'s collected in the set $I$ as before. Using that $U_V$
is a basic affine open subset, there are homogeneous polynomials $p_1,
\dots, p_n \in K[x_0,\dots,x_m]$ of the same degree $e$ such that
$\psi_V$ is given by $\psi_{V,j}=p_j/h^{k}$ on $U$. Replacing $h$ by
$h^k$, we may assume $k=1$. Then we set $p_0 \coloneqq h$.   A priori,
the integer $e$ depends on $V$. As we may enlarge it by multiplying
all polynomials by a power of the chosen $x_{i(V)}$ above, finiteness
of $I$ shows that we may assume $e$ independent of $V \in I$.   

We use now the list of all polynomials $p_j$ for varying $V \in I$ and
we expand the list to get a set $q_0,\dots,q_n$ generating the vector
space of homogeneous polynomials of degree $e$. Then we get a morphism
$\psi\colon X \to \P^n$ given by mapping $q_0,\dots,q_n$ to the
projective coordinates $y_0, \dots, y_n$. As we may choose $e$
sufficiently large, we may assume that the restrictions of
$q_0,\dots,q_n$ span $H^0(\overline X,\Ocal_{\P^m}(e))$ and hence
$\psi$ extends to a closed immersion ${\overline{X}} \to
\P^n$. In particular, $\psi$ is an immersion. By construction, the
restriction of $\psi$ to ${U_V}$ gives a closed immersion
$U_V \to \A^n$ which refines $\psi_V$ for any $V \in I$, where we
identify $\A^n$ with the standard open subset $\{y_{j(V)}\neq 0\}$ in
$\P^n$ with $y_{j(V)}$ the coordinate corresponding to the homogenous
polynomial $q_{i(V)}=p_0=h$ in the list from $V$.  Replacing $\psi_V$
by this refinement, we may assume that $\beta|_V$ is tropicalized by
the tropical chart $(\psi,V)$ with $V=\psi_\trop^{-1}(\Omega_V)$ for
an open subset $\Omega_V$ of $\Trop(\P^n)$ and by
$\beta_{V,\trop} \in B(\Trop(\P^m))$ using that $V$ is relatively compact.
Now the same partition of unity argument as in the proof of Proposition \ref{global tropicalization for affine} gives the claim.
\end{proof}

\begin{rem} \label{global tropicalization for toric varieties}
There is variant for global tropicalization. 
If $X$ is a normal variety which admits a closed embedding into a toric variety, then  every $\beta \in B_c(\Xan)$ can be tropicalized by a closed embedding $\psi\colon X \to X_\Sigma$ into a toric variety $X_\Sigma$ and by $\betatrop \in B_c(N_\Sigma)$. 
The crucial property is that  every closed embedding $\varphi\colon U \to \A^n$ of an affine open subset $U$ of $X$ can be refined to a closed embedding $\psi\colon U \to U_\sigma$ where $U_\sigma$ is the open affine subset of a toric variety $X_\Sigma$ associated to some $\sigma \in \Sigma$ in such a way that $\psi$ extends to an immersion of $X$ to $X_\Sigma$, see \cite[Example 4.13]{jell-2019}. 
Then we can proceed as in the proof of Proposition \ref{global tropicalization for affine}.	
\end{rem}

\subsection{Green functions} \label{subsec: non-arch Green functions} 
As we did in Definition \ref{formal delta-current partial compactifiction} for tropical toric varieties, we introduce $\delta$-currents on an analytic space $\Xan$ as linear
functionals on $\delta$-forms.  
For simplicity, we do not require any continuity of these linear functionals and call them \emph{formal $\delta$-currents}.

\begin{definition} \label{formal delta-current} 
Let $W\subset \Xan$ be an open subset.
A \emph{formal $\delta$-current} on $W$ is a linear functional on the
space $B_c(W)$ of compactly supported $\delta$-forms.  
It is a sheaf $E$ of bigraded differential $\R$-vector spaces with
differentials $d',d''$. 
We use the bigrading
\[
E^{p,q}(W)\coloneqq \Hom(B_c^{d-p,d-q}(W),\R)
\]
similarly as in Sections \ref{section: delta-forms} and \ref{sec:a tropical star-product}.
\end{definition}

\begin{ex} \label{current of integration}
We denote by $Z^p(X)$ the space of algebraic cycles of codimension $p$ on $X$.
For any $Z \in Z^p(X)$, we have the \emph{current of integration} $\delta_Z \in E^{p,p}(\Xan)$ given by 
\[
\delta_Z(\alpha)=\int_{\Zan} \alpha
\]
for $\alpha \in B_c^{d-p,d-p}(\Xan)$. 
Since the formula of Stokes holds \cite[Proposition 4.11]{mihatsch2021a}, we conclude that the formal $\delta$-current $\delta_Z$ is $d'$- and $d''$-closed.
\end{ex}

\begin{art} \label{piecewise smooth on analytic spaces}
We are mainly concerned with the sheaf $\PS$ of piecewise smooth forms
on $\Xan$ \cite[\S 4.2]{mihatsch2021a}. 
This is a sheaf of bigraded differential $\R$-algebras coming with
natural polyhedral differentials $\dpa,\dpb$ which do not necessarily agree with the derivatives $d',d''$ \emph{in the sense} of currents introduced above. 
Mihatsch \cite[Proposition 4.7]{mihatsch2021a} showed that 
\[
\PS^{p,q}(W)=B^{p,q,0}(W).
\]
The \emph{space of piecewise smooth functions on $W$} is {by definition} $\PS^{0,0}(W)$. 
\end{art}
  
\begin{definition}
  Let $D$ be a Cartier divisor of $X$. A function $g_{D}\colon
  X^{\an}\setminus |D|\to \R$ is called a \emph{Green function for $D$
    of piecewise smooth type} if for every open set $W\subset X$ and every
  local equation $t$ of $D$ on $W$ the function $g_{D}+\log|t|$ can
  be extended to a piecewise smooth function on the whole $W^{\an}$.  
\end{definition}

  \begin{ex}\label{exm:Green_functions_metric}
    A metric on a line bundle $L$ of $X$ is called \emph{piecewise
      smooth} if for any open subset $W$ and any nowhere vanishing
    section $t \in \Gamma(W,L)$, the function $-\log\|t\|$ is
    piecewise smooth on $W$.  
Any non-zero meromorphic section $s$ of $L$ gives rise to a
Green function $g_D \coloneqq -\log \|s\|$ of piecewise
  smooth type for the Cartier divisor $D=\Div(s)$ which we often consider as
a formal $\delta$-current in $E^{0,0}(\Xan)$. In fact every Green
function of piecewise smooth type is obtained in this way. Namely, if
$D$ is a divisor, $g_{D}$ a Green function for $D$ of piecewise smooth
type and $s$ a rational section of $\KO(D)$ with $\Div(s)=D$, then
there exists a unique piecewise smooth metric on $\KO(D)$ such that
$g_{D}=-\log\|s\|.$   
  \end{ex}

\begin{definition}\label{analytification-define-green-delta-forms}
	Given a pair $(Z,T)\in Z^p(X) \times E^{p-1,p-1}(\Xan)$, we define
	\[
	\omega(Z,T)\coloneqq d'd''T+\delta_Z\in E^{p,p}(\Xan).
	\]
	If the formal $\delta$-current $\omega(Z,T)$ is a $\delta$-form, then we call \emph{$T$ a Green $\delta$-current for $Z$}.
\end{definition}

\begin{rem} \label{closedness of cycles on X}
	Since $\delta_Z$ is $d'$- and $d''$-closed for any cycle $Z$, we conclude that the same is true for $\omega(Z,T)$.
\end{rem}

\begin{ex}\label{green-delta-function on X}
	Let $D$ be a  Cartier divisor on $X$ and let $g$ be a Green
        function of piecewise smooth type for $D$ with corresponding
        piecewise smooth metric $\metr$ on $L=O(D)$ as in
          Example \ref{exm:Green_functions_metric}. 
	The first Chern $\delta$-current $c_1(D,g)=c_1(L,\metr) \in
        E^{1,1}(\Xan)$ is locally on an open subset $W$ with nowhere
        vanishing $t \in \Gamma(W,L)$ defined by
        $-d'd''\log\|t\|$. Since $\log\|t\|$ is piecewise smooth, we
        conclude that $c_1(L,\metr) \in B^{1,1}(\Xan)$.  
	Then we have the 
	 \emph{Poincar\'e--Lelong equation} 
	\begin{equation} \label{Mihatsch's Poincare-Lelong}
		d'd''g=c_1(D,g) - \delta_D \in E^{1,1}(\Xan)
	\end{equation}
	which follows by applying locally \cite[Theorem 5.5]{mihatsch2021a} for the case $r=1$. In \ref{proof of non-archimedean Poincare--Lelong equation}, we will give an independent proof of the Poincar\'e--Lelong equation.
	We conclude that $\omega(D,g)=c_1(D,g)$ and hence $g$ is a Green $\delta$-function for $D$.
\end{ex}

\begin{rem} \label{product with green functions of piecewise smooth type}
Let $g$ be a Green function of piecewise smooth type for a Cartier divisor $D$ on $X$ and let $Z$ be a cycle of codimension $p$ in $X$ intersecting $D$ properly. 
Then $g|_Z$ is a Green function for the Cartier divisor $D \cdot Z =
D|_Z$ on $Z$  and we define $g\wedge \delta_Z \in E^{p,p}(\Xan)$ by
using the  push-forward of $g|_Z$  with respect to the immersion $Z
\to X$. More precisely, let $Z_{i}$, $i=1,\dots, r$ be the
irreducible components of $|Z|$ so that $Z=\sum n_{i}Z_{i}$. For each
$i$, let $\iota_{i}\colon Z_{i}\to X$ be the
inclusion. Then we define
\begin{displaymath}
  g\wedge \delta_Z=\sum_{i=1}^{r}
  n_{i}(\iota^{\an}_{i})_{\ast}([g|_{Z^{\an}_{i}}]). 
\end{displaymath}
Here $[g|_{Z^{\an}_{i}}]$ means the function $g|_{Z^{\an}_{i}}$ viewed as a formal
$\delta $-current on $Z^{\an}_{i}$.
\end{rem}

\begin{rem} \label{pull-back of Green functions}
Similarly, let $\varphi \colon X' \to X$ be a morphism of algebraic varieties.  
Let  $D$ be a Cartier divisor on $X$ which intersects $\varphi(X')$ properly and let $g$ be a Green function for $D$ of piecewise smooth type. 
Then $\varphi^*D$ is a well-defined Cartier divisor on $X'$ for which we have the Green function $\varphi^*g=g \circ \varphi$. We have 
\begin{equation}\label{pull-back-formula-divisors}
c_1(\varphi^*D,\varphi^*g) = \varphi^*c_1(D,g).
\end{equation}
\end{rem}

\subsection{Tropicalized Green functions}
We investigate extended tropicalizations of Green functions of piecewise smooth type and prove the Poincaré--Lelong equation.

\begin{art} \label{tropicalization of Cartier divisors}
Let $D$ be a Cartier divisor on $X$. If there is a morphism $\psi\colon X \to  X_\Sigma$
and a toric Cartier divisor $E$ of $X_\Sigma$ which intersects $\psi(X)$ properly, 
then we say that \emph{$D$ is tropicalized by $\psi$ and $D_\trop$}
where $D_\trop$ is the tropical toric Cartier divisor on $N_\Sigma$
corresponding to $E$. If $D$ is effective, then we require $E$
{to be} effective as well.

It is clear that $\supp(E)$ is a (closed) union of strata of
$X_\Sigma$ corresponding to cones $\tau$ of $\Sigma$ not containing
the sedentarity of {$\trop(\psi (X))$} and that
$\supp(D_\trop)=\trop(\supp(E))$ is the  union of the corresponding
$N(\tau)$, hence we conclude $D_\trop$ intersects $\psitrop(\Xan)$ properly.

It is clear that every (effective) Cartier divisor on an affine or projective variety can be tropicalized by a closed embedding to $\A^n$ and $\P^n$, respectively.
\end{art}

\begin{definition} \label{def: tropicalization of Green function}
Let $g$ be a Green function of piecewise smooth type for a Cartier divisor $D$ on $X$. We say that $g$ is \emph{tropicalized by a morphism $\psi\colon X \to {X_{\Sigma }}$ and by $g_\trop$ on $N_\Sigma$} if the Cartier divisor $D$ is tropicalized by $\psi$ and the tropical toric Cartier divisor $D_\trop$ on $N_\Sigma$ and if there is a Green function $g_\trop$ of piecewise smooth type for $D_\trop$ on $N_\Sigma$ such that $g= g_\trop \circ \psitrop$.

We say that $g$ is \emph{locally tropicalized at $x \in \Xan$} if there is a Zariski open neighbourhood $U$ of $x$ in $X$, a morphism $\psi \colon U \to X_\Sigma$ tropicalizing $D|_U$ and a Green function $g_\trop$ of piecewise smooth type for $D_\trop$ on $N_\Sigma$ such that $g=g_\trop \circ \psitrop$ on an open neighbourhood of $x$ in $\Uan$.
\end{definition}

\begin{prop} \label{first Chern form and tropicalization}
Let $g$ be a Green function of piecewise smooth type for the Car\-tier divisor $D$ on $X$ which is tropicalized by the morphism $\psi\colon X \to X_\Sigma$ and by the Green form $g_\trop$ for the tropical toric Cartier divisor $D_\trop$ on $N_\Sigma$. Then we have 
$$c_1(D,g)=\psi_\trop^*(c_1(D_\trop,g_\trop))$$
for the associated first Chern $\delta$-forms.
\end{prop}

\begin{proof}
We may check this locally in $x \in \Xan$ and hence we may assume that $X_\Sigma$ is the open toric subvariety $X_\sigma$ for some $\sigma \in \Sigma$.  Adding to $g_\trop$ a linear function, we may assume that $D_\trop=0$ and hence $D=0$. Then $g_\trop$ is a bounded piecewise linear function in a neighbourhood $\Omega$ of $\psitrop(x)$ in $N_\Sigma$ and we have $c_1(D_\trop,g_\trop)=d'd''g_\trop$ on $\Omega$. Similarly, we have $c_1(D,g)=d'd''g$ on $\psi_\trop^{-1}(\Omega)$. The claim follows from Proposition \ref{properties of tropicalized delta-forms}.
\end{proof}

\begin{prop} \label{local tropicalization of green functions}
	Every Green function $g$ of piecewise smooth type for a
        Cartier divisor $D$ on $X$ can be locally tropicalized by a
        closed embedding $\psi\colon U \to \A^n$ and by a Green
        function of piecewise smooth type $g_\trop$ for $D_\trop$ on
        $\R_\infty^n$. If $X$ is affine, we may choose $U=X$.
\end{prop}

\begin{proof}
	Passing to an affine open subset, we may assume $X$ affine. {By
        \ref{tropicalization of Cartier divisors}}
        there is a closed embedding $\phi\colon X \to \A^n$ such that
        $D$ is tropicalized by $\psi$ and a tropical Cartier divisor
        $D_\trop$. We pick any Green function $g'_\trop$ of piecewise
        smooth type for $D_\trop$ on $\R_\infty^n$. Then $g' \coloneqq
        g_\trop'$ is a Green function of piecewise smooth type for
        $D$ and hence $g-g'$ is a piecewise smooth function on
        $\Uan$. We pick any relatively compact open neighbourhood $W$
        of $x$ in $\Xan$. Using a partition of unity argument and
        using that the piecewise smooth functions are precisely the
        elements in $B^{0,0,0}$ both on the tropical side and on the
        non-archimedean side, we conclude from Proposition \ref{global
          tropicalization for affine} that there is closed embedding
        $\phi'\colon U \to \A^{n'}$ and a piecewise smooth function
        $h$ on $\R_\infty^{n'}$ which tropicalizes $g-g'$ on $W$.
        We
        may replace $\phi$ and $\phi'$ by their product $\psi$ as this
        {is} a common refinement. By construction, the Green
        function $g$ is tropicalized by $\psi$ and by the piecewise
        smooth Green function $g_\trop \coloneqq g_\trop'+h$. 
\end{proof}

\begin{prop} \label{tropicalization of Green functions}
Any Green function  of piecewise smooth type for an (effective) Cartier divisor $D$ on a projective variety $X$ can be tropicalized by a closed embedding $\psi \colon X \to X_\Sigma$  and by a Green function  of piecewise smooth type for $D_\trop$ on $N_\Sigma$. We may choose $X_\Sigma= \P^{n}\times \P^{n'}$.
\end{prop}

\begin{proof}
The proof is the same as for Proposition \ref{local tropicalization of green functions}. It is even easier as we may use $U=X$. Since we don't have the problem with the compact support of the piecewise smooth function $g-g'$, we can also choose $W=\Xan$ and hence the tropicalization is global.
\end{proof}

\begin{art} \label{push-forward of delta-currents}
Obviously, we have a \emph{push-forward} of formal $\delta$-currents
with respect to proper morphisms. Let us consider a proper morphism
$\psi\colon X \to X_\Sigma$, e.g.~a closed immersion. Using that
$\trop\colon X_\sigma \to N_\Sigma$ is a proper map, we conclude that
$\psitrop$ is a proper map of topological spaces. We conclude that
every formal $\delta$-current $T$ on $\Xan$ induces a formal
$\delta$-current $\psi_{\trop,*}(T)$ on $N_\Sigma$.  
\end{art}

\begin{rem} \label{compatibility of wedge and tropicalization}
Let $D$ be an effective Cartier divisor on $X$ which intersects $Z \in Z^p(X)$ properly. We consider a Green function $g$ of piecewise smooth type for $D$ which is tropicalized by the closed embedding $\psi\colon X \to X_\Sigma$ and by the Green function $g_\trop$ of piecewise smooth type for $D_\trop$ on $N_\Sigma$. Then it follows from Example \ref{integration and global tropicalization} that $C\coloneqq \psi_{\trop,*}(Z)$ is a classical tropical cycle of dimension $d$ on $N_\Sigma$. 

Using that the toric Cartier divisor $E$ on $X_\Sigma$ corresponding
to {$D_\trop$}  is effective, we may see $E$ as a closed subscheme whose underlying set is $\supp(E)$. This is compatible with preimages and hence $E$ intersects $\psi_*(Z)$ properly. Using that $E$ is toric with associated tropical toric divisor $D_\trop$, we conclude that $D_\trop$ intersects $\psi_{\trop,*}(Z)$ properly.
We claim that
\begin{equation} \label{product of Green functions with cycles and tropicalization}
\psi_{\trop,*}(g\wedge \delta_Z)= g_\trop \wedge \delta_C.
\end{equation}
Indeed, we have to test that against $\betatrop \in B_c^{d,d}(N_\Sigma)$. By linearity, we may assume that $Z=X$  and that $\psi(X)$ has sedentarity $0$. The support $S$  of the top-degree $\delta$-form $\beta \coloneqq \psi_\trop^*(\betatrop) \in B_c(\Xan)$ is contained in the dense orbit of $X_\Sigma$ and the support $S_\trop$ of the top-degree polyhedral current $\betatrop \wedge \delta_C$ on $N_\Sigma$ is contained in $N_\R$. It follows that the restrictions of $g$ and $g_\trop$ to these supports are given by piecewise smooth functions and hence  we have $g \beta \in B_c^{d,d}(\Xan)$ which is tropicalized by $\phi_\trop g_\trop \betatrop$ for a suitable  smooth function $\phi_\trop$ with compact support in $N_\R$ which is identically $1$ on $\psi_\trop(S) \cup S_\trop$. Then we may replace in the evalution of \eqref{product of Green functions with cycles and tropicalization} at $\betatrop$ the Green function $g_\trop$ by $\phi_\trop g$ and the claim  follows indeed from \eqref{integration of globally tropicalized forms}.
\end{rem}

\begin{art}[Proof of the Poincar\'e--Lelong equation] \label{proof of non-archimedean Poincare--Lelong equation}
We are going to proof \eqref{Mihatsch's Poincare-Lelong} using our
corresponding result on tropical toric varieties. Since the equation
is local, we may assume that $X$ is affine and that $D$ is
effective. By Proposition \ref{local tropicalization of green
  functions}, there is a closed embedding $\psi\colon X \to \A^n$
which tropicalizes $D$ and a Green function $g_\trop$ of piecewise
smooth type for $D_\trop$ which tropicalizes $g$. We have to test
\eqref{Mihatsch's Poincare-Lelong} against $\beta \in
B_c^{d-1,d-1}(\R_\infty^n)$. Replacing $\psi$ by a refinement and
using Proposition \ref{global tropicalization for affine}, we may
assume that $\beta$ is tropicalized by $\betatrop \in
B_c^{d-1,d-1}(\R_\infty^n)$.  
We have
$$ [d'd''g](\beta)=[d'd''g](\psi_\trop^*\betatrop)=\psi_{\trop,*}[d'd''g](\betatrop)=\psi_{\trop,*}[g](d'd''\betatrop).$$
Let $C$ be the classical tropical cycle in $N_\Sigma$ associated to $X$, see Example \ref{integration and global tropicalization}. 
Using \eqref{product of Green functions with cycles and tropicalization}, we get
$$ [d'd''g](\beta)=g_\trop \wedge \delta_C(d'd''\betatrop)=(d'd''(g_\trop \wedge \delta_C))(\betatrop).$$
Now the tropical Poincar\'e--Lelong equation for Green functions of piecewise smooth type in Theorem \ref{tropical PL equation for Green functions} shows that 
\[
d'd''(g_\trop \wedge \delta_C) = c_1(D_\trop,g_\trop)\wedge \delta_C - \delta_{D_\trop \cdot C}
\]
using that the Weil divisor associated to a tropical toric Cartier divisor agrees with its infinite part. We conclude that 
\[
[d'd''g](\beta)= \delta_C(c_1(D_\trop,g_\trop)\wedge \betatrop)-\delta_{D_\trop \cdot C}(\betatrop).
\]
Using that $D_\trop \cdot C$ is the classical tropical cycle associated to the Weil divisor $D \cdot X$ of $D$ and using Proposition \ref{first Chern form and tropicalization}, the same steps backwards as above show the Poincar\'e--Lelong equation \eqref{Mihatsch's Poincare-Lelong}.
\qed
\end{art}

\subsection{The $*$-product with Green functions of piecewise smooth type} \label{subsec: star-product with ps Green functions}
We construct a $*$-product with Green functions of piecewise smooth type for a divisor.

\begin{definition}[$*$-product]\label{star product on X}
For a  Green function $g$ of piecewise smooth type for $D$  and a pair $(Z,T)\in Z^p(X) \times E^{p-1,p-1}(\Xan)$ such that $D$ intersects $Z$ properly, let
\[
(D,g) * (Z,T) \coloneqq (D \cdot Z, g \wedge \delta_Z+\omega(D,g)\wedge T)\in {Z^{p+1}(X) \times E^{p,p}(\Xan).}
\]
Using that $\omega(D,g)$ is a $\delta$-form as seen in Example \ref{green-delta-function on X},  the term $\omega(D,g)\wedge T$ is defined as usual and the term $g\wedge\delta_Z$ is defined in Remark \ref{product with green functions of piecewise smooth type}.
\end{definition}

\begin{prop} \label{multiplicativity of omega on X}
In the situation of Definition \ref{star product on X}, we have
\[
\omega((D,g) * (Z,T)) =\omega(D,g)\wedge \omega(Z,T).
\]
\end{prop}

\begin{proof}
	Using the definitions, we have
	\[
	\omega((D,g) * (Z,T))=d'd''(g \wedge \delta_Z+\omega(D,g)\wedge T)+\delta_{D\cdot Z}.
	\]
	Since $\omega(D,g)$ is $d'$- and $d''$-closed {by Remark \ref{closedness of cycles on X}}, the Leibniz rule yields
	\[
	\omega((D,g) * (Z,T))=d'd''(g \wedge \delta_Z)+\omega(D,g)\wedge d'd''T+\delta_{D\cdot Z}.
	\]
	Now we use that $g\wedge \delta_Z$ is the push-forward of $g|_Z$ with respect to the immersion $Z \to N_\Sigma$ and that $d',d''$ commute with this push-forward of currents. In combination with the Poincar\'e--Lelong equation \eqref{Mihatsch's Poincare-Lelong} for $g|_Z$ on the  cycle $Z$, we deduce
	\[
	\omega((D,g) * (Z,T))= \omega(D,g)\wedge \delta_Z- \delta_{D\cdot Z}+\omega(D,g)\wedge d'd''T+\delta_{D\cdot Z}
	\]
	using that $g|_Z$ is a Green $\delta$-current for $D|_Z$ by Example \ref{green-delta-function on X}  and that the push-forward of $D|_Z$ to $X$ is $D \cdot Z$.
	Finally,  $d'd''T=\omega(Z,T)-\delta_Z$ leads to the claim.
\end{proof}

\begin{lem} \label{star-product and tropicalization}
Let $g$ be a Green function of piecewise smooth type for the effective Cartier divisor $D$ on $X$ and let  $(Z,T)\in Z^p(X) \times E^{p-1,p-1}(\Xan)$ such that $D$ intersects $Z$ properly.  We consider a closed embedding $\psi\colon X \to X_\Sigma$ such that $g$ is tropicalized by the Green function $g_\trop$ of piecewise smooth type for $D_\trop$. Then $D_\trop$ intersects the classical tropical cycle $C\coloneqq \psi_{\trop,*}(Z)$ properly  and we have
$$\psi_{\trop,*}((D,g)*(Z,T))=(D_\trop,g_\trop)*(C,\psi_{\trop,*}(T)).$$
\end{lem}

\begin{proof}
We have already seen in Remark \ref{compatibility of wedge and tropicalization} that $C$ is a classical tropical cycle of $N_\Sigma$ and that $D_\trop$ intersects $C$ properly. Now the claim follows from the definition of the $*$-product together with \eqref{product of Green functions with cycles and tropicalization} and Proposition \ref{first Chern form and tropicalization}.
\end{proof}

\begin{prop} \label{analytification-symmetry of star-product}
Consider a pair $(Z,T)\in Z^p(X) \times E^{p-1,p-1}(\Xan)$ and Cartier divisors $D_1,D_2$ on $X$ intersecting the  cycle $Z$ properly. 
For Green functions $g_1,g_2$ for $D_1,D_2$ of piecewise smooth type, we have
\[
(D_1,g_1)*(D_2,g_2)*(Z,T)=(D_2,g_2)*(D_1,g_1)*(Z,T) \quad \text{\rm modulo} \quad \mathrm{Im}(d'+d'')
\]	
where 
\[
\mathrm{Im}(d'+d'')=(0,{d'E^{p,p+1}(\Xan)+d''E^{p+1,p}(\Xan)})\subset {Z^{p+2}(X) \times E^{p+1,p+1}(\Xan)}.
\]
\end{prop}

\begin{proof}
More precisely, following Moriwaki's terminology \cite{moriwaki-book}, we will prove Weil's reciprocity law
\begin{equation} \label{Weil reciprocity law on X}
\begin{split}
(D_1,g_1)*(D_2,g_2)*(Z,T)&-(D_2,g_2)*(D_1,g_1)*(Z,T)\\
&{=(0,d''(g_1 d'g_2 \wedge \delta_Z)+d'(g_2 d''g_1\wedge \delta_Z))} 
\end{split}
\end{equation}
in $Z^{p+2}(X) \times  E^{p+1,p+1}(\Xan)$.
The claim is local and so we may assume that $X$ is affine. 
We may also assume that the Cartier divisors $D_1,D_2$ are effective.
Since we have an equality of cycles $D_1\cdot D_2\cdot Z=D_2\cdot D_1\cdot Z$ we are reduced to check an equality of formal $\delta $-currents, that can be tested against a compactly supported $\delta$-form $\beta$ of bidegree $(d-1,d-1)$. 
Using that $D_1,D_2$ intersect $Z$ properly, it is clear that the support of $\beta$ is disjoint from  $\supp(D_1)\cap\supp(D_2) \cap |Z|$. 
As we may check the claim locally and as \eqref{Weil reciprocity law on X} holds trivially outside the support $|Z|$, we conclude that we may assume that $g_1$  is a bounded piecewise smooth function and that $D_1=0$.

We can now apply the philosophy of the paper already applied in the proof of the Poincar\'e--Lelong equation in \ref{proof of non-archimedean Poincare--Lelong equation}. 
Using Proposition \ref{global tropicalization for affine} and Proposition \ref{local tropicalization of green functions}, all the objects in \eqref{Weil reciprocity law on X} and $\beta$ are tropicalized by a closed immersion  $\psi \colon W\to \A^n$ with the only exception of $(Z,T)$.
We may also assume that $D_{2,\trop}$ is an effective tropical Cartier divisors and hence the associated toric Cartier divisors $D_{2,\tor}$ is an effective toric Cartier divisor on $X_\Sigma$. 
Using that it is effective and restricts to $D_2$ on $X$, we see that $D_2$ intersects the cycle $Z$ properly. 
Let $C$ be the classical tropical cycle in $N_\Sigma$ associated to the cycle $Z$, see Example \ref{integration and global tropicalization}. 
It follows that the tropical toric Cartier divisor  $D_{2,\trop}$ intersects $C$ properly. 
Here, we use that it is toric and hence its support agrees with the union $N(\tau)$ with $\tau$ ranging over all $1$-dimensional cones of $\Sigma$ corresponding to the components of $D_1$.
Using the push-forward of $T$ to $\R_\infty^n$ with respect to $\psitrop$, we deduce \eqref{Weil reciprocity law on X} evaluated at $\beta$ from \eqref{Weil reciprocity law} for $N_\Sigma$ similarly as in the proof of the proof of the Poincar\'e--Lelong equation in \ref{proof of non-archimedean Poincare--Lelong equation}. 
	 
We work this out as follows. 
Using Lemma \ref{star-product and tropicalization}, we have 
\[
\psi_{\trop,*}((D_1,g_1)*(D_2,g_2)*(Z,T))=(D_{1,\trop},g_{2,\trop})*(D_{2,\trop},g_{2,\trop})*\psi_{\trop,*}(Z,T)
\]
and a similar identity holds for $(D_2,g_2)*(D_1,g_1)*(Z,T)$.
The evaluation of the term $d''(g_1 d'g_2 \wedge \delta_Z)$  in  \eqref{Weil reciprocity law on X} at $\beta$ is equal to
\[
d''(g_1 d'g_2 \wedge \delta_Z)(\beta)=-(d'g_2 \wedge \delta_Z)(g_1  d''\beta)=-(g_2 \wedge \delta_Z)(d'(g_1d''\beta)).
\]
Similarly, we have 
\[
d''(g_{1,\trop} d'g_{2,\trop} \wedge \delta_C)(\betatrop)=-(g_{2,\trop} \wedge \delta_C)(d'(g_{1,\trop}d''\betatrop))
\]
on $N_\Sigma$ as we already have seen in the proof of \eqref{Weil reciprocity law}. 
These two displays are equal by \eqref{product of Green functions with cycles and tropicalization} and Proposition \ref{properties of tropicalized delta-forms}. 
Similar arguments prove
\[
d'(g_2 d'' g_1 \wedge \delta_Z)(\beta) 
=d'(g_{2,\trop} d'' g_{1,\trop} \wedge \delta_C)(\betatrop).
\]
We conclude that all corresponding terms in \eqref{Weil reciprocity law on X} and in \eqref{Weil reciprocity law} match.
\end{proof}

\begin{prop}[Projection formula] \label{projection formula on X}
Let $\varphi\colon X' \to X$ be a proper morphism of algebraic varieties with $d \coloneqq \dim(X)$ and $d' \coloneqq \dim(X')$, let $D$ be a Cartier divisor on $X$ and $(Z',T')\in Z^p(X') \times E^{p-1,p-1}((X')^\an)$ such that $D$ intersects $\varphi(|Z'|)$ properly in $X$. For a Green function $g$ for $D$ of piecewise smooth type, we have
\[
\varphi_*(\varphi^*(D,g) * (Z',T'))= (D,g) * \varphi_*(Z',T').
\]
\end{prop}

\begin{proof}
If $Z'=0$, then the claim has to be read as 
$$\varphi_*(\varphi^*(c_1(D,g)) \wedge T')= c_1(D,g) \wedge \varphi_*(T')$$
which is obvious. So we may assume $Z' \neq 0$ and hence $D$
intersects $\varphi(X')$ properly. Then $\varphi^*D$ is a well-defined Cartier divisor
on $X'$ which intersects $Z'$ properly and $\varphi^*g$ is a Green
function for $\varphi^*D$ of piecewise smooth type. The claim follows
now from the transformation formula in Proposition \ref{transformation
  formula over X} and Remark \ref{pull-back of Green functions}.
\end{proof}

\section{Monge--Amp\`ere measures and local heights} \label{section MA-measures and local heights}

The goal of this section is to show that we can compute local heights tropically. We fix a $d$-dimensional algebraic variety $X$ over a non-archimedean field $K$.

\subsection{Monge--Amp\`ere measures} \label{subsection MA-measures}

We consider line bundles $\overline{L}_1, \dots, \overline{L}_d$ on
$X$ endowed with piecewise smooth metrics. 

\begin{art} \label{definition MA measure} 
We note that $c_1(\overline{L}_1)\wedge \dots \wedge
c_1(\overline{L}_d)$ is a $\delta$-form of bidegree $(d,d)$ on $\Xan$
which can be seen as a top-degree polyhedral current in a piecewise
linear space $S$ of pure dimension $d$ in $\Xan$
which is given 
locally by tropical skeletons.  
Writing the coefficients of this polyhedral current as the difference of two positive Lagerberg forms of top-degree, it is clear that we can write the {formal} $\delta$-current associated to $c_1(\overline{L}_1)\wedge \dots \wedge c_1(\overline{L}_d)$ as two positive linear functionals on $C_c^\infty(\Xan)$. 
Using that the compactly supported smooth functions are dense in {the space of compactly supported continuous functions} $C_c(V)$ for any open subset $V$ on $\Xan$ \cite[Proposition 3.3.5]{chambert-loir-ducros}, we deduce that the positive linear functionals extend uniquely to positive functionals on $C_c(\Xan)$ and hence define positive Radon measures by the Riesz representation theorem. 
We conclude that there is a unique real Radon measure $\MA(\overline{L}_1, \dots, \overline{L}_d)$ on $\Xan$ which agrees on $C^\infty_c(\Xan)$ with the formal $\delta$-current associated to $c_1(\overline{L}_1)\wedge \dots \wedge c_1(\overline{L}_d)$. 
We call $\MA(\overline{L}_1, \dots, \overline{L}_d)$ the \emph{Monge--Amp\`ere measure associated to $\overline{L}_1, \dots, \overline{L}_d$}. 
By abuse of notation, we will denote this Monge--Amp\`ere measure often by $c_1(\overline{L}_1)\wedge \dots \wedge c_1(\overline{L}_d)$ as well.
\end{art}

\begin{prop} \label{properties of MA measures}
The Monge--Amp\`ere measures $\MA(\overline{L}_1, \dots, \overline{L}_d)$ have the following properties:
\begin{enumerate}
\item \label{MA multilinearity}
They are multi-linear and symmetric in $\overline{L}_1, \dots, \overline{L}_d$.
\item \label{MA projection formula}
If $\varphi \colon X' \to X$ is a proper morphism of algebraic varieties, then we have the projection formula
\[
\varphi_*\left(\MA(\varphi^*\overline{L}_1, \dots, \varphi^*\overline{L}_d)\right)= \deg(\varphi) \MA(\overline{L}_1, \dots, \overline{L}_d).
\]
\item \label{MA total mass}
If $X$ is a proper variety, then $\MA(\overline{L}_1, \dots, \overline{L}_d)(\Xan)=\deg_{L_1,\dots,L_d}(X)$.
\end{enumerate}
\end{prop}

\begin{proof}
Property \eqref{MA multilinearity} follows from \ref{Mihatsch's delta forms}. 
The projection formula in \eqref{MA projection formula} is a consequence of the transformation formula in Proposition \ref{transformation formula over X}. 
Finally, the total mass in \eqref{MA total mass} is computed as in \cite[Proposition 6.4.3]{chambert-loir-ducros}.
\end{proof}

\begin{rem}  \label{Chambert-Loir measures}
A special case of piecewise smooth metrics are metrics on line bundles $L_1,\dots,L_d$ over a proper variety $X$ of dimension $d$ which are induced by models over the valuation ring $K^\circ$. 
Chambert-Loir introduced an associated discrete measure on $\Xan$ \cite{chambert-loir-2006} which is important in diophantine geometry for equidistribution results as shown e.g.~by Yuan \cite{yuan-2008}. 
It follows from \cite[\S 6.9]{chambert-loir-ducros} and \cite[Theorem 10.5]{gubler-kuenne2017} that these discrete measures agree with $\MA(\overline{L}_1, \dots, \overline{L}_d)$.
\end{rem}

\subsection{Local heights} \label{subsection local heights}

In this subsection, we assume that the $d$-dimensional variety $X$ is proper over $K$. 
Using the Arakelov theory from the previous section, we will define local heights for properly intersecting Cartier divisors and corresponding Green functions of piecewise smooth type.

\begin{definition}[Local heights] \label{definition of local heights}
{A \emph{Néron divisor on $X$} is  given by a pair $(D,g)$ where $D$ is a Cartier divisor on $X$ and $g$ is a Green function of piecewise smooth type for $D$ as defined in \ref{piecewise smooth on analytic spaces}.}
Let us consider N\'eron divisors $\widehat{D_0}=(D_0,g_0), \dots, \widehat{D_d}=(D_d,g_d)$ {on $X$}, where {the Cartier divisors} $D_0,\dots, D_d$  intersect properly on $X$. 
Then we define the \emph{local height of $X$ with respect to $\widehat{D_0}, \dots, \widehat{D_d}$} by
\[
\lambda_{\widehat{D_0}, \dots, \widehat{D_d}}(X) \coloneqq \bigl((D_0,g_0) * \dots  *(D_d,g_d)* (X,0)\bigr)(1).
\]
By multi-linearity, we can define the local height $\lambda_{\widehat{D_0}, \dots, \widehat{D_d}}(Z)$ of a $t$-dimensional cycle $Z$ with respect to N\'eron divisors $\widehat{D_0}, \dots, \widehat{D_t}$ which intersect $Z$ properly. Similarly as in \cite[Remark 12.3]{gubler-kuenne2017}, we can define this local height under the weaker condition that 
$
{\supp(D_0) \cap \ldots\cap\supp(D_t)}\cap |Z|
=\emptyset$.
\end{definition}

\begin{prop} \label{properties local heights}
The local heights $\lambda_{\widehat{D_0}, \dots, \widehat{D_d}}(X)$ have the following properties:
\begin{enumerate}
\item \label{multilinear and symmetric local height}
They are multi-linear and symmetric in $\widehat{D_0}, \dots, \widehat{D_d}$.
\item \label{induction formula}
Let $\cyc(D_0)$ be the Weil divisor, then we have the induction formula
\[
\lambda_{\widehat{D_0}, \dots, \widehat{D_d}}(X)= \lambda_{\widehat{D_1}, \dots, \widehat{D_d}}(\cyc(D_0)) + \int_\Xan g_0 \cdot c_1(D_1,g_1) \wedge \dots \wedge c_1(D_d,g_d).
\]
\item \label{projection formula local height}
If $\varphi\colon X' \to X$ is a morphism from a proper variety $X'$ over $K$, then
\[
\lambda_{\varphi^*\widehat{D_0}, \dots, \varphi^*\widehat{D_d}}(X')=
\deg{\varphi}\cdot \lambda_{\widehat{D_0}, \dots, \widehat{D_d}}(X).
\]
\item \label{change of metric formula}
Let $g_0'$ be another Green function of piecewise smooth type for $D_0$, then $\rho \coloneqq g_0'-g_0$ is a piecewise smooth function. For $\widehat{D_0}' \coloneqq (D_0,g_0)$, we have
\[
\lambda_{\widehat{D_0}', \dots, \widehat{D_d}}(X)-\lambda_{\widehat{D_0}, \dots, \widehat{D_d}}(X)=\int_\Xan \rho \cdot c_1(D_1,g_1) \wedge \dots \wedge c_1(D_d,g_d).
\]
\end{enumerate}
\end{prop}

\begin{proof}
This proposition was shown in \cite[\S 12]{gubler-kuenne2017} for $\delta$-metrics. 
These are special piecewise smooth metrics such that the first Chern currents are $\delta$-forms in the sense of \cite{gubler-kuenne2017}. 
Using Mihatsch's $\delta$-forms as we do in the paper here, all piecewise smooth metrics have this property. 
Hence the proposition holds for all piecewise smooth metrics using the same arguments as {in \emph{loc.cit}}.
\end{proof}

The following result shows that we can compute local heights tropicallly.

\begin{thm}\label{computation of local heights tropically}
Let $\widehat{D_0}=(D_0,g_0), \dots, \widehat{D_d}=(D_d,g_d)$ be N\'eron divisors on the $d$-dimensional projective variety $X$ such that $D_0,\dots,D_d$ intersect properly. Then there {exist a fan $\Sigma$ and} a closed immersion $\psi\colon X \to X_\Sigma$ into the {toric} variety $X_\Sigma$ which tropicalizes $g_j$ for every $j=0,\dots,d$ by a Green function $g_{j,\trop}$ of piecewise smooth type for $D_{j,\trop}$ on $N_\Sigma$. 
For the tropical cycle $C=\psi_{\trop,*}(X)$, we have
\[
\lambda_{\widehat{D_0}, \dots, \widehat{D_d}}(X)=((D_{0,\trop},g_{0,\trop})* \dots * (D_{d,\trop},g_{d,\trop})*(C,0))(1).
\]
\end{thm}

\begin{proof}
The existence of the joint tropicalization $\psi\colon X \to X_\Sigma$ follows from Proposition \ref{tropicalization of Green functions}. 
In fact, we might choose a closed embedding into a multi-projective space. 
By definition, we have
\[
\lambda_{\widehat{D_0}, \dots, \widehat{D_d}}(X)=(\psi_{\trop,*}((D_0,g_0)* \dots * (D_d,g,d)*(X,0)))(1).
\]
Applying  Lemma \ref{star-product and tropicalization} to the right hand side, we get the claim.
\end{proof}


\end{document}